\RequirePackage{fix-cm}
\documentclass[smallextended]{svjour3}       
\smartqed  
\usepackage{amsmath, amssymb, bm, mathrsfs}
\usepackage{cite, color, graphicx,dsfont,accents}
\usepackage[colorlinks=true,linkcolor=blue,backref=page]{hyperref}
\renewcommand*{\backref}[1]{}
\makeatletter
\def\widebar{\accentset{{\cc@style\underline{\mskip10mu}}}}
\makeatother
\newcommand{\cl}{\mathop{\rm cl}\nolimits}

\spnewtheorem{theorem}{Theorem}[section]{\bfseries}{\itshape}

\spnewtheorem{lemma}[theorem]{Lemma}{\bfseries}{\itshape}
\spnewtheorem{corollary}[theorem]{Corollary}{\bfseries}{\itshape}
\spnewtheorem{problem}[theorem]{Problem}{\bfseries}{\itshape}
\spnewtheorem{definition}[theorem]{Definition}{\bfseries}{\itshape}

\spnewtheorem{remark}[theorem]{Remark}{\bfseries}{\upshape}
\spnewtheorem{assumption}[theorem]{Assumption}{\bfseries}{\itshape}

\newcommand{\re}{\mathop{\rm Re}\nolimits}

\newcommand{\dom}{\mathop{\rm dom}}

 \makeatletter

\@addtoreset{equation}{section}
\makeatother

\begin{document}

\title{Sampled-data Output Regulation of Unstable Well-posed Infinite-dimensional Systems
	with Constant Reference and Disturbance Signals
	\thanks{This work was supported by JSPS KAKENHI Grant Numbers JP17K14699.}
}

\titlerunning{Sampled-data Output Regulation of Unstable Well-posed Systems}        

\author{Masashi Wakaiki        \and
        Hideki Sano 
}


\institute{M.~Wakaiki \at
              Graduate School of System Informatics, Kobe University, Nada, Kobe, Hyogo 657-8501, Japan \\
              Tel.: +8178-803-6232 \\
              Fax: +8178-803-6392\\
              \email{wakaiki@ruby.kobe-u.ac.jp}           
           \and
           H.~Sano \at
              Graduate School of System Informatics, Kobe University, Nada, Kobe, Hyogo 657-8501, Japan \\
              Tel.: +8178-803-6380 \\
              Fax: +8178-803-6392\\
              \email{sano@crystal.kobe-u.ac.jp} 
}

\date{Received: date / Accepted: date}

\maketitle

\begin{abstract}
We study the sample-data control problem of 
output tracking and disturbance rejection for
unstable well-posed linear infinite-dimensional systems with 
constant reference and disturbance signals.
We obtain a sufficient condition for the existence of
finite-dimensional sampled-data controllers that 
are solutions of this control problem.
To this end, we study the problem of
output tracking and disturbance rejection for infinite-dimensional discrete-time  systems
and propose a design method of finite-dimensional controllers by using a solution of
the Nevanlinna-Pick interpolation problem with both interior and boundary conditions.
We apply our results to systems  with state and 
output delays.
\keywords{
	Frequency-domain methods \and Output regulation \and Sampled-data control \and
	State-space methods \and Well-posed infinite-dimensional systems}
\subclass{93B52 \and 93C05 \and 93C25 \and 93C35 \and 93C57 \and 93D15}
\end{abstract}

\section{Introduction}
Due to the development of computer technology,
digital controllers are commonly implemented for continuous-time plants.
We call such closed-loop systems {\em sampled-data systems}.
In addition to their practical motivation,
sampled-data systems yield theoretically interesting problems 
related to
 the combination of both continuous-time and discrete-time dynamics, and
various techniques such as the lifting approach \cite{Bamieh1992, Yamamoto1994, Yamamoto1996} and 
the frequency response operator approach \cite{Hagiwara1995, Araki1996}
have been also developed for the analysis and synthesis of sampled-data finite-dimensional systems.
Sampled-data control theory for infinite-dimensional systems has been
developed, e.g., in \cite{Rebarber1998,
	Rebarber2002, Logemann1997, Logemann2003, Logemann2005, Rebarber2006, Ke2009SIAM, Ke2009IEEE, Ke2009SCL, 
	Logemann2013,Selivanov2017}.
Several specifically relevant studies will be cited below again.
In this paper, we study the problem of sampled-data output regulation for 
unstable well-posed systems.
The main objective in our  control problem is
to find finite-dimensional digital controllers achieving  the
output tracking of
given constant reference signals in the presence of 
external constant disturbances.
A theory for well-posed systems has been extensively developed; see, e.g., 
the survey \cite{Weiss2001, Tucsnak2014} and the book \cite{Staffans2005}.
Well-posed systems allow unbounded control and observation operators and
provide a framework to formulate control problems for
systems governed by partial differential equations with
point control and observation and by 
functional differential equations with delays
in the state, input, and output variables.

Our output regulation method is based on the internal model principle,
which was originally developed for finite-dimensional systems in \cite{Francis1975}
and was later generalized for infinite-dimensional systems with
finite-dimensional and infinite-dimensional exosystems in \cite{Yamamoto1988servo, Paunonen2010, Paunonen2014, 
	Paunonen2017TAC,Paunonen2017SIAM} and references therein.
In particular, output regulation of nonsmooth periodic signals has
been extensively studied as {\em repetitive control} \cite{hara1988}.
For regular systems, which is a subclass of well-posed systems,
the authors of \cite{Xu2004, Paunonen2016TAC, Paunonen2017SIAM,Boulite2018} have provided design methods of
continuous-time controllers for robust output regulation.
For stable well-posed systems with finite-dimensional exosystems,
low-gain controllers suggested by the internal model principle
have been constructed for the continuous-time setup in \cite{Logemann1997SIAM, Rebarber2003}
and for the sampled-data setup in \cite{Logemann1997, Ke2009SIAM, Ke2009IEEE, Ke2009SCL}.
The difficulty of the problem we consider arises from
the instability of well-posed systems.
If the system is unstable, then low-gain controllers
cannot achieve closed-loop stability.
Ukai and Iwazumi \cite{Ukai1990} have developed a state-space-based design method 
of finite-dimensional controllers for output regulation of unstable continuous-time
infinite-dimensional systems, by using residue mode filters proposed in \cite{Sakawa1983}.
On the other hand, we employ a frequency-domain technique based on coprime factorizations as in
\cite{Logemann1997SIAM,Logemann1997, Hamalainen2000, Laakkonen2015, Laakkonen2016}.
In particular, we extend a design method of  stabilizing sampled-data controllers in \cite{Logemann2013}
to output regulation.

Let $(A,B,C)$ and $\bf G$ be generating operators and a transfer function of a well-posed system $\Sigma$,
respectively.
The operator $A$ is the generator of a strongly continuous semigroup $\bf T$, which
governs the dynamics of the system without control.
The operators $B$ and $C$ are control and observation operators, respectively.
We consider only infinite-dimensional systems that has finite-dimensional input and output spaces with the 
same dimension. In other words,
the transfer function $\bf G$ is a square-matrix-valued function.
The well-posed system is connected with a discrete-time linear time-invariant 
controller $\Sigma_{\rm d}$ through
a zero-order hold $\mathcal{H}_{\tau}$ and a generalized sampler $\mathcal{S}_{\tau}$, 
where $\tau>0$ is a sampling period. 
Let $u,y$ be the input and output of the well-posed system $\Sigma$ and $u_{\rm d}, y_{\rm d}$ be
the input and output of the digital controller $\Sigma_{\rm d}$, respectively.
The generalized sampler $\mathcal{S}_{\tau}$ is written
as
\[
(\mathcal{S}_{\tau}y) (k) = \int^\tau_0 w(t) y(k\tau+t) dt \qquad \forall k \in \mathbb{Z}_+,
\] 
where the scalar weighting function $w$ belongs to $L^2(0,\tau)$ and satisfies $\int^\tau_0 w(t) dt = 1$.
In well-posed systems, the output $y$ is in $L^2_{\rm loc}$, and hence
the ideal sampling, i.e., point evaluation does not make sense.
The weighting function $w$ should be chosen so that the sampled-data system
is detectable.

Using the zero-order hold $\mathcal{H}_{\tau}$ and the generalized sampler $\mathcal{S}_{\tau}$,
we consider the sampled-data feedback of the form
\[
u = \mathcal{H}_{\tau} y_{\rm d} + v\mathds{1}_{\mathbb{R}_+}\qquad 
u_{\rm d} = y_{\rm ref}\mathds{1}_{\mathbb{Z}_+}  - \mathcal{S}_\tau y,
\]
where 
$v\mathds{1}_{\mathbb{R}_+}$ and $y_{\rm ref}\mathds{1}_{\mathbb{Z}_+} $ are 
constant external reference and disturbance signals, respectively.
Fig.~\ref{fig:sampled_data_sys} illustrates the sampled-data system we study.
Since the output $y$ of well-posed systems belongs to $L^2_{\rm loc}$, 
the output $y$ is not guaranteed to converge
to $y_{\rm ref}$ as $t \to \infty$. In this paper, we therefore consider
the convergence of the output in the ``energy'' sense, i.e.,
there exist constants $\Gamma_{\rm ref} > 0$ and $\alpha<0$ such that 
\begin{align*}
	\| 
	y - y_{\rm ref}\mathds{1}_{\mathbb{R}_+}
	\|_{L^2_\alpha} \leq 
	\Gamma_{\rm ref} &\left(
	\left\|
	\begin{bmatrix}
		x(0) \\
		x_{\rm d}(0)
	\end{bmatrix}
	\right\|_{X \times X_{\rm d}}
	+\|v\|_{\mathbb{C}^p}+ \|y_{\rm ref}\|_{\mathbb{C}^p}
	\right) 
\end{align*}
for 
all initial states $x(0) \in X$ of $\Sigma$ and $x_{\rm d}(0) \in X_{\rm d}$ of $\Sigma_{\rm d}$ and all 
$v,y_{\rm ref} \in \mathbb{C}^p$, where ${L^2_\alpha} $
is the $L^2$-space weighted by the exponential function $e^{-\alpha t}$.
The above condition means that as $t \to \infty$,
the ``energy'' of the restricted tracking error $(y - y_{\rm ref})|_{[t,\infty)}$ exponentially converges to zero.
If we embed a smoothing precompensator between the plant and the zero-order 
hold as proposed in \cite{Logemann2013},
then the  output $y$ exponentially converges to $y_{\rm ref}$ in the usual sense under 
a certain regularity assumption on initial states.

\begin{figure}[tb]
	\centering
	\includegraphics[width = 6cm]{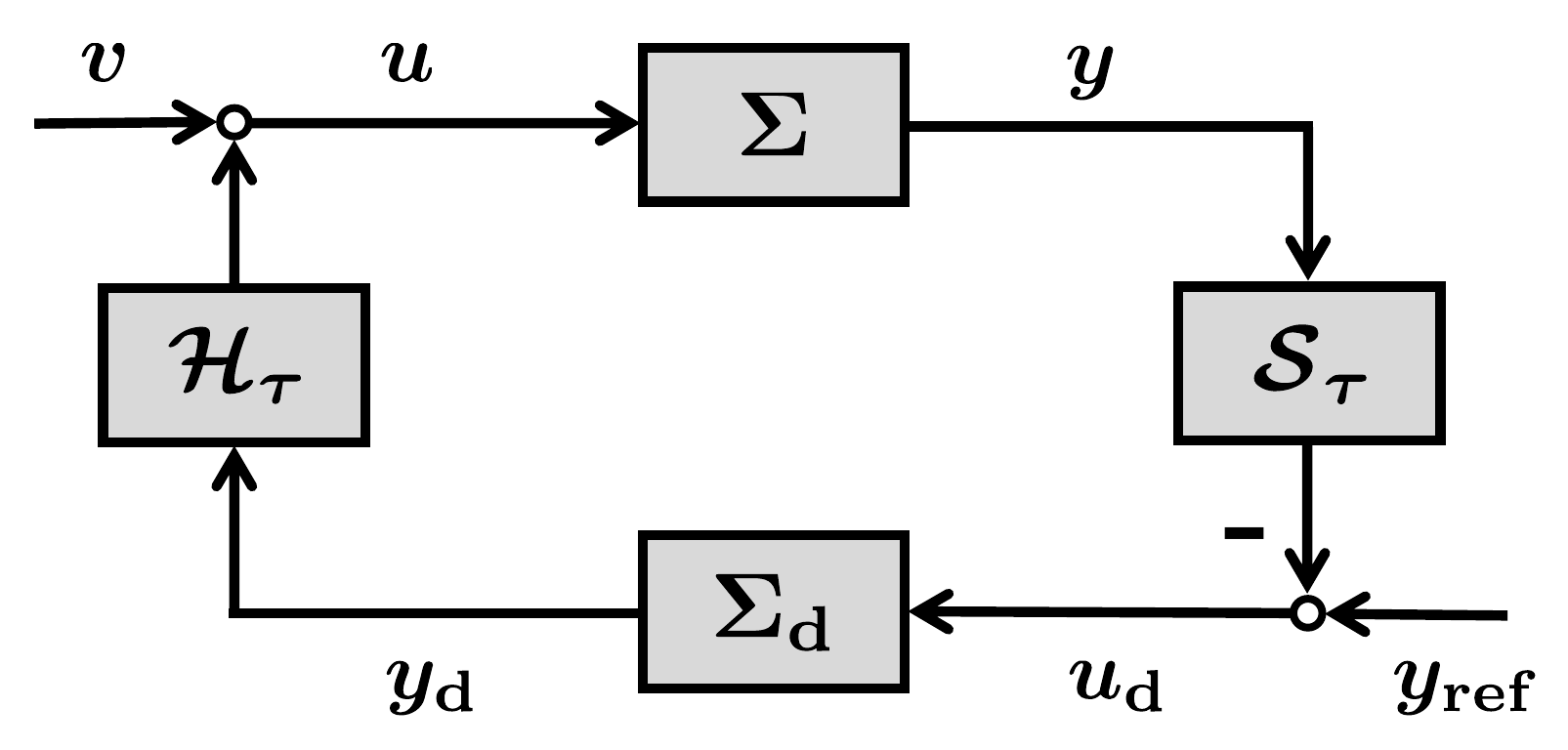}
	\caption{Sampled-data system.}
	\label{fig:sampled_data_sys}
\end{figure}

Before studying the sampled-data output regulation problem, 
we investigate an output regulation problem for  
infinite-dimensional discrete-time systems. In the discrete-time setup,
we propose a design method of
finite-dimensional controllers that achieve output regulation.
Although in the 
sampled-data setup, 
we consider only constant reference and disturbance signals, 
the proposed method in the discrete-time setup allows
reference and disturbance signals that are finite superpositions of sinusoids.
The construction of regulating controllers consists of two steps: 
First we design stabilizing controllers with a free parameter in $H^{\infty}$, using
the techniques developed in \cite{Logemann1992, Logemann2013}.
Next we choose the free parameter so that the controller incorporates an internal model
for output regulation.
The design problem of regulating controllers is reduced to the Nevanlinna-Pick interpolation problem with
both interior and boundary conditions.
In the reduced interpolation problem,
interior conditions are required for stabilization, whereas
boundary conditions arise from output tracking and disturbance rejection.

Our main result, Theorem \ref{thm:servo_tracking}, states that 
there exists a finite-dimensional digital controller that achieves 
output regulation for constant reference and disturbance signals if the following conditions are satisfied:
\begin{enumerate}
	\renewcommand{\labelenumi}{(\roman{enumi})}
	\item The resolvent set of $A$ contains $0$.
	\item $\det {\bf G}(0) \not= 0$.
	\item The unstable part of the spectrum of $A$ consists of 
	finitely many eigenvalues with finite multiplicities.
	\item The semigroup generated by the stable part of $A$ is exponentially stable.
	\item The unstable part of $(A,B,C)$ is controllable and observable.
	\item For every nonzero integer $\ell$, $2\ell \pi i/ \tau$ does not belong to the spectrum of $A$.
	\item For every unstable eigenvalues $\lambda$ of $A$, $\int^\tau_0 w(t) e^{\lambda t} dt\not=0$.
	\item For every unstable eigenvalues $\lambda,\mu$ of $A$ and nonzero integer $\ell$, 
	$\tau(\lambda - \mu) \not= 2\ell \pi i$.
	\item The multiplicities of all unstable eigenvalues of $A$ are one. 
\end{enumerate}
The assumptions (iii)-- (viii) are used for 
sampled-data stabilization in
\cite{Logemann2013}. In fact, 
(iii)--(vii) are sufficient for the existence of sampled-data stabilizing controllers, and further,
(iii)--(viii) are necessary and sufficient in the single-input and single-output case.
In particular, (v)--(viii) is used to guarantee that the unstable part of 
the sampled-data system is controllable and observable.
We place
the assumptions (i) and (ii)  for output regulation.
The remaining assumption (ix) is used to 
reduce the design problem of stabilizing controllers to 
an interpolation problem of functions in the $H^{\infty}$-space.
In the multi-input and multi-output case,
the assumption (ix) makes it easy to obtain the associated interpolation conditions.
We can remove (ix) in the single-input and single-output case.

The paper is structured as follows.
In Section 2, we study an output regulation problem for infinite-dimensional discrete-time  systems.
In Section 3, we obtain a sufficient condition for the existence of 
finite-dimensional sampled-data regulating controllers for unstable well-posed systems with
constant reference and disturbance signals.
In \mbox{Section 4}, we illustrate our results by applying them to systems  with state and 
output delays.

\subsubsection*{Notation and terminology}
We denote by $\mathbb{Z}_+$ and $\mathbb{R}_+$ the set of nonnegative integers and the set of nonnegative real numbers, respectively.
For $\alpha \in \mathbb{R}$, we define $\mathbb{C}_{\alpha} :=
\{s \in \mathbb{C}: \re s > \alpha\}$, and
for $\eta >0$, $\mathbb{E}_{\eta} :=
\{z \in \mathbb{C}:|z|> \eta \}$.
We also define 
$\mathbb{D}:=
\{z \in \mathbb{C}:|z| < 1 \}$ and $\mathbb{T}:=
\{z \in \mathbb{C}:|z| =1 \}$. 
For a set $\Omega \subset \mathbb{C}$, its closure is denoted by $\cl (\Omega)$.
For an arbitrary set $\Omega_0$,
the indicator function of $\Omega \subset \Omega_0$ is denoted by
$\mathds{1}_{\Omega}$.
For a matrix $M \in \mathbb{C}^{p\times m}$, let us denote by $M^*$, $\widebar M$, and $M^{\rm adj}$ 
the conjugate transpose, the matrix with complex conjugate entries, 
and the adjugate matrix of $M$, respectively.

Let $X$ and $Y$ be Banach spaces. Let $\mathcal{L}(X,Y)$ denote
the space of all bounded linear operators from $X$ to $Y$.
We set $\mathcal{L}(X):= \mathcal{L}(X,X)$.
An operator $T \in \mathcal{L}(X)$ is said to be 
{\em power stable} if there exist $\Gamma \geq 1$ and $\rho \in (0,1)$
such that 
$
\|T^k\|_{\mathcal{L}(X)} \leq \Gamma \rho^k
$
for every $k \in \mathbb{Z}_+$.
Let  ${\bf T} = ({\bf T}_t)_{t\geq 0}$ be
a strongly continuous semigroup on $X$.
The exponential growth bound of ${\bf T}$ is denoted by
$\omega({\bf T})$, that is, $\omega({\bf T}) := \lim_{t \to \infty} \ln \|{\bf T}_t\|/t$.
We say that the strongly continuous semigroup ${\bf T}$ is {\em exponentially stable}
if $\omega({\bf T}) < 0$. 
For a linear operator $A$ from $X$ to $Y$,
let $\dom (A)$
denote the domain of $A$.
The spectrum and resolvent set of a linear operator $A: \dom (A) \subset X \to X$ are denoted by
$\sigma(A)$ and $\varrho(A)$, respectively.

For $\alpha \in \mathbb{R}$, we define
the weighted $L^2$-space $L_\alpha^2(\mathbb{R}_+,X)$ by 
$L_\alpha^2(\mathbb{R}_+,X) := \{
f:\mathbb{R}_+ \to X: e_{-\alpha} f \in L^2 (\mathbb{R}_+,X)
\}$, 
where $e_{-\alpha}(t) := e^{-\alpha t}$ for $t \in \mathbb{R}_+$, with
the norm $\|f\|_{L^2_\alpha} := \|e_{-\alpha}f\|_{L^2}$.
The space of all functions from $\mathbb{Z}_+$ to $X$ is denoted by $F(\mathbb{Z}_+,X)$. Set
$f^{\bigtriangledown}(k) := f(k+1)$ for every $k \in \mathbb{Z}_+$ and every $f \in F(\mathbb{Z}_+,X)$.
Let $\Omega = \mathbb{C}_{\alpha}$, $\Omega = \mathbb{E}_\eta$, or $\Omega =\mathbb{D}$.
Let $H^{\infty}(\Omega, \mathbb{C}^{p \times m})$ denote 
the space of all bounded holomorphic functions from $\Omega $ to $\mathbb{C}^{p\times m}$.
The norm of $H^{\infty}(\Omega, \mathbb{C}^{p \times m})$  is given by $\|\Phi\|_{\infty} := \sup_{s \in \Omega}\|\Phi(s)\|$.
We write $H^{\infty}(\Omega)$ for $H^{\infty}(\Omega, \mathbb{C})$.

\section{Discrete-time output regulation}
\label{sec:DTOR}
In this section, we construct finite-dimensional controllers for the
robust output regulation of infinite-dimensional discrete-time  systems.
Before proceeding to technical details, we describe the overview of this section.
A fundamental assumption throughout this paper is that 
an infinite-dimensional plant can be decomposed into
a finite-dimensional unstable part and an infinite-dimensional stable part.
To avoid spill-over effects \cite{Balas1978},
we cannot ignore the infinite-dimensional stable part completely 
in the design of stabilizing controllers. However,
it has been shown in \cite{Logemann1992, Logemann2013} that 
if  the infinite-dimensional stable part is  appropriately approximated
by a finite-dimensional stable system,
then we can design  finite-dimensional stabilizing controllers.
Now one may ask the following question for the problem of output regulation:
\begin{quote}
{\em By such an approximation-based method,
can we always construct stabilizing controllers that incorporate an internal model?}
\end{quote}

To stabilize the plant, the approximation error should be small. However,
it is possible that if the approximation error is smaller than a certain threshold,
then we cannot design stabilizing controllers with internal models by using
the finite-dimensional approximating  system.
We will show that 
such a situation does not occur under certain assumptions on the plant. 
The proof is based on two key facts: First, 
controllers incorporate  internal models if and only if their free parameters in $H^{\infty}(\mathbb{E}_1,\mathbb{C}^{p \times p})$
satisfy
certain
interpolation conditions on the boundary $\mathbb T$.
Second, 
the boundary Nevanlinna-Pick interpolation problem (see Problem A.7 in the appendix for details)
is always solvable.

In Section~\ref{subsec:DTassumption}, we formulate the problem of robust output regulation and
recall the concept of $p$-copy internal models.
In Section~\ref{subsec:DTmain_result}, we introduce assumptions of the plant and
provide the main result of this section, Theorem~\ref{thm:existence_servo_cont}.
We provide the proof of this theorem
in Sections~\ref{subsec:Prelim_lemma_DT} and \ref{subsec:Proof_main_result}.
In particular, Sections~\ref{subsec:Prelim_lemma_DT} is devoted to preliminary lemmas for the multi-input
and the multi-output case.
Section~\ref{subsec:Prelim_lemma_DT} may be skipped by
the readers interested only in the single-input and single-output case.

\subsection{Control objective}
\label{subsec:DTassumption}
In this section, we consider the following 
infinite-dimensional discrete-time  system:
\begin{subequations}
	\label{eq:plant}
	\begin{align}
		x^{\bigtriangledown}(k) &= A x(k) + Bu(k),\quad x(0) = x^0 \in X\\
		y(k) &= Cx(k) + Du(k),
	\end{align}
	where the state space $X$ is a separable complex Hilbert space, $A \in \mathcal{L}(X)$, $B \in \mathcal{L}(\mathbb{C}^p, X)$, 
	$C \in \mathcal{L}(X, \mathbb{C}^p)$, and $D \in \mathbb{C}^{p \times p}$.
\end{subequations}
We use a strictly causal controller
\begin{subequations}
	\label{eq:controller}
	\begin{align}
		x_{\rm d}^{\bigtriangledown}(k) &= P x_{\rm d}(k) + Qu_{\rm d}(k),\quad x_{\rm d}(0) = x_{\rm d}^0 \in X_{\rm d}\\
		y_{\rm d}(k) &= R x_{\rm d}(k),
	\end{align}
\end{subequations}
where the state space $X_{\rm d}$ is a complex Hilbert space,
$P \in \mathcal{L}(X_{\rm d})$, $Q \in \mathcal{L}(\mathbb{C}^p, X_{\rm d})$, and
$R \in \mathcal{L}(X_{\rm d}, \mathbb{C}^p)$.
The control objective is that the output $y$ tracks a given reference signal $y_{\rm ref}$ in the presence of
an external disturbance signal $v$. The reference and disturbance signals $y_{\rm ref}$ and $v$
are assumed to be generated by an exosystem of the form
\begin{subequations}
	\label{eq:exosystem}
	\begin{align}
		\xi^{\bigtriangledown}(k) &= S\xi(k),\qquad \xi(0) = \xi^0 \in  \mathbb{C}^{n} \\
		v(k) &= E\xi(k) \\
		y_{\rm ref}(k) &= F\xi(k),
	\end{align}
\end{subequations}
where $E \in \mathbb{C}^{p \times n}$, $F \in \mathbb{C}^{p \times n}$, and
\[
S := \text{diag} \big(e^{i\theta_1},\dots, e^{i\theta_n}\big) \text{~with $\theta_1,\dots,\theta_n \in [0,2\pi)$ distinct}.
\]

The input $u$ of the plant and the input $u_{\rm d}$ of the controller are given by
\[
u(k) = y_{\rm d}(k) + v(k),\qquad 
u_{\rm d}(k) =  y_{\rm ref}(k) - y(k) =: e(k).
\]
We can write the dynamics of the closed-loop system as
\begin{subequations}
	\label{eq:closed}
	\begin{align}
		\label{eq:closed_state}
		x_{\rm e}^{\bigtriangledown}(k) &= A_{\rm e}x_{\rm e}(k) + B_{\rm e} \xi(k),\quad x_{\rm e}(0) = x_{\rm e}^0 \\
		\label{eq:closed_error}
		e(k) &= C_{\rm e}x_{\rm e}(k) + D_{\rm e}\xi(k),
	\end{align}
\end{subequations}
where $x_{\rm e}(k) = 
\begin{bmatrix}
x(k) \\ x_{\rm d}(k)
\end{bmatrix}$, 
$x_{\rm e}^0 = 
\begin{bmatrix}
x^0\\ x_{\rm d}^0
\end{bmatrix}$,  and
\begin{subequations}
		\label{eq:Ae_def}
\begin{align}
	A_{\rm e} &:=
	\begin{bmatrix}
		A & ~~BR \\
		-QC & ~~P - QDR
	\end{bmatrix},\quad 
	B_{\rm e} :=
	\begin{bmatrix}
		BE \\
		Q(F-DE)
	\end{bmatrix}\\
	C_{\rm e} &:=
	-
	\begin{bmatrix}
		C &
		~~DR
	\end{bmatrix},\quad 
	D_{\rm e} :=
	F-DE.
\end{align}
\end{subequations}

For the controller in \eqref{eq:controller} represented by 
the operators $(P,Q,R)$, we consider a set of perturbed plants and exosystems $\mathcal{O}(P,Q,R)$ defined as follows.
\begin{definition}[Set of perturbed plants and exosystems]
	For given operators 
	$P \in \mathcal{L}(X_{\rm d})$, $Q \in \mathcal{L}(\mathbb{C}^p, X_{\rm d})$, and
	$R \in \mathcal{L}(X_{\rm d}, \mathbb{C}^p)$,
$\mathcal{O}(P,Q,R)$ is the set of all $(\widetilde{A},\widetilde{B}$, $\widetilde{C},\widetilde{D},\widetilde{E},\widetilde{F})$
satisfying
	the following two conditions:
	\begin{enumerate}
		\item  
		$\widetilde{A} \in \mathcal{L}(X)$, $\widetilde{B} \in \mathcal{L}(\mathbb{C}^p, X)$, 
		$\widetilde{C} \in \mathcal{L}(X, \mathbb{C}^p)$, $\widetilde{D} \in \mathbb{C}^{p \times p}$,
		$\widetilde{E} \in \mathbb{C}^{p \times n}$, and $\widetilde{F} \in \mathbb{C}^{p \times n}$.

		\item  The perturbed operator $\widetilde{A}_{\rm e}$ defined by
		\[
		\widetilde{A}_{\rm e} := \begin{bmatrix}
		\widetilde{A} & ~~\widetilde{B}R \\
		-Q\widetilde{C} & ~~P - Q\widetilde{D}R
		\end{bmatrix}
		\]
		is power stable.
	\end{enumerate}
\end{definition}

If $A_{\rm e}$ is power stable, the conditions above are satisfied for any
bounded perturbations of sufficiently small norms.

In this section, we study a robust output regulation problem.
\begin{problem}[Robust output regulation for discrete-time systems]
	\label{prob:ROR}
	Given the plant \eqref{eq:plant} and the exosystem \eqref{eq:exosystem},
	find a controller \eqref{eq:controller} satisfying the following properties:
	\begin{description}
		\item[Stability:]
		The operator $A_{\rm e}$ is power stable. 
		
		\item[Tracking:]
		There exist $M_e >0$ and $\rho_e \in (0,1)$ such that 
		for	every initial state $x^0 \in X$, $x_{\rm d}^0 \in X_{\rm d}$, and $\xi^0 \in  \mathbb{C}^{n} $,
		the tracking error $e$ satisfies
		\[
		\|e(k)\|_{\mathbb{C}^{p}} \leq M_e \rho_e^k
		\left(
		\left\|
		\begin{bmatrix}
		x^0 \\
		x_{\rm d}^0
		\end{bmatrix}
		\right\|_{X\times X_{\rm d}} + \|\xi^0\|_{\mathbb{C}^{n} }
		\right)
		\qquad \forall k \in \mathbb{Z}_+.
		\]
		
		\item[Robustness:]
		If the operators $(A,B,C,D,E,F)$ are changed to
		$(\widetilde{A},\widetilde{B},\widetilde{C},\widetilde{D},\widetilde{E},\widetilde{F}) \in \mathcal{O}(P,Q,R)$,
		then the above tracking condition still holds.
	\end{description}
\end{problem}

Before proceeding to the construction of finite-dimensional regulating controllers, 
we recall  the internal model principle.
In \cite{Paunonen2010}, a $p$-copy internal model has been introduced for
continuous-time systems.
The 
discrete-time counterpart has appeared in Section IV.B of \cite{Paunonen2017TAC}.
\begin{definition}[Definition~6.1 in \cite{Paunonen2010}]
	A controller  \eqref{eq:controller} is said to incorporate a $p$-copy internal model
	of the exosystem  \eqref{eq:exosystem} if 
	\begin{equation}
		\label{eq:internal_model}
		\dim {\rm ker}(e^{i\theta_\ell}I - P) \geq p\qquad \forall \ell \in \{1,\dots,n \}.
	\end{equation}
\end{definition}
\begin{theorem}[Theorem~IV.5 in \cite{Paunonen2017TAC}]
	\label{thm:internal_model}
	Suppose that $A_{\rm e}$ is power stable.
	The controller \eqref{eq:controller} incorporates a $p$-copy internal model of the 
	exosystem \eqref{eq:exosystem} if and only if it is a solution of Problem~\ref{prob:ROR}.
\end{theorem}

\subsection{Output regulation by a finite-dimensional controller}
\label{subsec:DTmain_result}
Throughout this section, we impose the following assumptions:
\begin{enumerate}
	\renewcommand{\labelenumi}{$\langle$a\arabic{enumi}$\rangle$}
	\item $e^{i\theta_\ell} \in \varrho(A)$ for every $\ell \in \{1,\dots,n\}$. \label{enu_Resol}
	\item $\det \mathbf{G}(e^{i\theta_\ell}) \not= 0$  for every $\ell \in \{1,\dots,n \}$. \label{enu_Zero}
	\item There exist subspaces $X^+$ and $X^-$ of $X$ such that \label{enu_Dec}
	$\dim X^+ < \infty$ and $X = X^+ \oplus X^-$. 
	\item $A X^+ \subset X^+ $ and $AX^- \subset X^-$. \label{enu_Ainv}
\end{enumerate}

Let us denote the projection operator from $X$ to $X^+$ by $\Pi$, and define
\begin{align*}
A^+ &:= A|_{X^+},\quad 
B^+ := \Pi B,\quad 
C^+:= C|_{X^+}\\
A^- &:= A|_{X^-},\quad 
B^- := (I-\Pi) B,\quad 
C^-:= C|_{X^-}.
\end{align*}
We place the remaining assumptions.
\begin{enumerate}
	\renewcommand{\labelenumi}{$\langle$a\arabic{enumi}$\rangle$}
	\setcounter{enumi}{4}
	\item $\sigma(A) \cap \cl(\mathbb{E}_1)$ consists of
	finitely many eigenvalues with finite algebraic multiplicities, $\sigma(A^+) = 
	\sigma(A) \cap \cl(\mathbb{E}_{1}) $, and there exists $\eta_0 \in (0,1)$ such that $\sigma(A^-) = \sigma(A) \cap 
	\big(\mathbb{C} \setminus \cl(\mathbb{E}_{\eta_0}) \big)$. \label{enu_A_minus_EXS}
	\item $(A^+,B^+,C^+)$ is controllable and observable.\label{enu_CD}
	\item The zeros of $\det (zI-A^+)$ are simple. \label{enu_Simple}
\end{enumerate}

We place the assumptions $\langle$a\ref{enu_Resol}$\rangle$ and 
$\langle$a\ref{enu_Zero}$\rangle$ for robust output regulation.
The assumptions $\langle$a\ref{enu_Dec}$\rangle$--$\langle$a\ref{enu_CD}$\rangle$ are used for stabilization of 
infinite-dimensional discrete-time  systems; see, e.g.,  \cite{Logemann1992}.
We will show in Lemma~\ref{lem:det_zeros} below that the assumption
$\langle$a\ref{enu_Simple}$\rangle$ guarantees that 
$\dim \ker (\lambda I-A^+) = 1$ for every 
$\lambda \in \mathbb{C}$ satisfying $\det (\lambda I-A^+) = 0$.
This allows us to reduce the design problem of stabilizing controllers
to the problem of finding functions in $H^{\infty}(\mathbb{E}_1,\mathbb{C}^{p \times p})$
that satisfy elementary interpolation conditions, which will be shown 
in Lemma~\ref{lem:Bezout_interpolation}.
In the single-input and single-output case $p=1$,
we can remove $\langle$a\ref{enu_Simple}$\rangle$ as mentioned at the end of this section.
This is because 
it is much easier to translate stabilization into interpolation in the scalar-valued case
than in the matrix-valued case.

Under  $\langle$a\ref{enu_Ainv}$\rangle$ and $\langle$a\ref{enu_A_minus_EXS}$\rangle$,
we fix $\eta \in (\eta_0,1)$
and define the transfer function $\mathbf{G}$ of the plant \eqref{eq:plant}
by
\[
\mathbf{G}(z) := C(zI-A)^{-1}B+D \qquad \forall z \in \mathbb{E}_\eta \cap \varrho(A).
\]
We can decompose $\mathbf{G}$ into
\begin{equation*}
\mathbf{G}(z) = \mathbf{G}^+(z) + \mathbf{G}^-(z) \qquad \forall z \in \mathbb{E}_\eta \cap \varrho(A),
\end{equation*}
where
\begin{equation}
\label{eq:G_plus_minus}
\mathbf{G}^+(z) := C^+(zI - A^+)^{-1}B^+,\quad
\mathbf{G}^-(z) := C^-(zI - A^-)^{-1}B^-+D
\end{equation}
and $\mathbf{G}^-  \in H^{\infty}(\mathbb{E}_\eta, \mathbb{C}^{p\times p})$.
By $\langle$a\ref{enu_CD}$\rangle$, the unstable part $\mathbf{G}^+$ of the plant has no
unstable pole-zero cancellations.
There exist $\mathbf{N}_+$, $\mathbf{D}_+$ with rational entries
in $H^{\infty}(\mathbb{E}_1,\mathbb{C}^{p \times p})$ such that 
\[
\mathbf{G}^+ = \mathbf{D}_+^{-1}\mathbf{N}_+
\]
and $\mathbf{N}_+, \mathbf{D}_+$ are left coprime over the sets of
rational functions in $H^{\infty}(\mathbb{E}_1,\mathbb{C}^{p \times p})$.
Choose such $\mathbf{N}_+, \mathbf{D}_+$ arbitrarily, and
let $\chi_1,\dots,\chi_\Upsilon$ be the zeros of $\det \mathbf{D}_+$ in
$\cl(\mathbb{E}_1)$.
Together with $\langle$a\ref{enu_CD}$\rangle$ and $\langle$a\ref{enu_Simple}$\rangle$,
Lemma~A.7.39 of \cite{Curtain1995} shows that 
these zeros are equal to the eigenvalues of $A^+$
and that
the orders of the zeros are one.

The objective of this section is to prove the following theorem constructively:
\begin{theorem}
	\label{thm:existence_servo_cont}
	Assume that {\rm $\langle$a\ref{enu_Resol}$\rangle$--$\langle$a\ref{enu_Simple}$\rangle$} hold.
	There exists a finite-dimensional controller \eqref{eq:controller} that 
	is a solution of the robust output regulation problem, Problem~\ref{prob:ROR}.
\end{theorem}	

\subsection{Preliminary lemmas}
\label{subsec:Prelim_lemma_DT}
Before proceeding to the proof of Theorem~\ref{thm:existence_servo_cont},
we show three preliminary results, all of which
are used for the multi-input and multi-output case $p>1$.
Hence the readers who are interested only in the single-input
and single-output case $p=1$ can skip this subsection.

The first lemma provides an upper bound on the norm of inverse matrices.
\begin{lemma}
	\label{lem:M2inverse}
	Let $V,W \in \mathbb{C}^{p \times p}$.
	If $V$ is invertible and if 
	\[\|V^{-1}\|_{\mathbb{C}^{p \times p}} \cdot \|V - W\|_{\mathbb{C}^{p \times p}} < 1,\] then $W$ is also invertible and 
	\begin{equation}
		\label{eq:M2_inverse}
		\|W^{-1}\|_{\mathbb{C}^{p \times p}} \leq 
		\frac{\|V^{-1}\|_{\mathbb{C}^{p \times p}}}{1 - \|V^{-1}\|_{\mathbb{C}^{p \times p}} \cdot\|V - W\|_{\mathbb{C}^{p \times p}} }.
	\end{equation}
\end{lemma}
\begin{proof}
	Since 
	\[
	\|I - V^{-1} W\|_{\mathbb{C}^{p \times p}} \leq \|V^{-1}\|_{\mathbb{C}^{p \times p}}  \cdot \|V - W\|_{\mathbb{C}^{p \times p}} < 1,
	\]
	it follows that $V^{-1}W$ and hence $W$ are invertible.
	
	Using the identity
	\[
	V^{-1} - W^{-1} = V^{-1} (W - V) W^{-1},
	\]
	we obtain
	\[
	\|
	V^{-1} - W^{-1} 
	\|_{\mathbb{C}^{p \times p}}
	\leq 
	\|
	V^{-1}
	\|_{\mathbb{C}^{p \times p}}\cdot
	\|
	V - W
	\|_{\mathbb{C}^{p \times p}}	\cdot
	\|
	W^{-1}
	\|_{\mathbb{C}^{p \times p}}.
	\]
	This yields
	\begin{align*}
		\|
		W^{-1}
		\|_{\mathbb{C}^{p \times p}}	
		&\leq 
		\|
		V^{-1}
		\|_{\mathbb{C}^{p \times p}} + 	\|
		V^{-1} - W^{-1} 
		\|_{\mathbb{C}^{p \times p}} \\
		&\leq 
		\|
		V^{-1}
		\|_{\mathbb{C}^{p \times p}}  + 	\|
		V^{-1}
		\|_{\mathbb{C}^{p \times p}}\cdot
		\|
		V - W
		\|_{\mathbb{C}^{p \times p}}	\cdot
		\|
		W^{-1}
		\|_{\mathbb{C}^{p \times p}}.
	\end{align*}
	Thus, we obtain the desired inequality \eqref{eq:M2_inverse}.
	\qed
\end{proof}

The second preliminary result characterizes adjugate matrices.
\begin{lemma}
	\label{lem:det_zeros}
	For a region $\Omega \subset \mathbb{C}$,
	consider a holomorphic function $\Delta: \Omega \to \mathbb{C}^{p \times p}$.
	Suppose that $z_0 \in \Omega$ is a simple zero of $\det \Delta$. 
	Then $\dim \ker \Delta(z_0) = 1$.
	Furthermore, if a nonzero vector $\psi \in \mathbb{C}^p$ satisfies
	$\ker \Delta(z_0)^* = \{\alpha \psi:\alpha \in \mathbb{C}\}$, then
	there exist $\alpha_1,\dots,\alpha_p \in \mathbb{C}$ such that 
	$\alpha _\ell \not= 0$ for some $\ell \in \{1,\dots,p \}$ and 
	$\Delta^{\rm adj}(z_0)$ can be written as
	\begin{equation}
		\label{eq:adj_vector}
		\Delta^{\rm adj}(z_0) = 
		\begin{bmatrix}
			\alpha_1 \psi^* \\ \vdots \\ \alpha_p \psi^*
		\end{bmatrix}.
	\end{equation}
\end{lemma}
\begin{proof}
	Suppose, to get a contradiction, that $\dim \ker \Delta(z_0) \geq 2$.
	There exist nonzero vectors $\psi_1,\psi_2 \in \mathbb{C}^p$ such that $\psi_1, \psi_2$ are linearly independent and
	$\Delta(z_0)\psi_1 = 0$, $\Delta(z_0)\psi_2 = 0$. 
	Let ${\rm e}_1,\dots,{\rm e}_p$ be the standard basis of the $p$-dimensional Euclidean space.
	There exists an invertible matrix $U \in \mathbb{C}^{p \times p}$ such that  
	$\psi_1 = U {\rm e}_1$ and $\psi_2 = U {\rm e}_2$.
	Let us denote by $\Delta_\ell$
	the $\ell$th column vector of the product $\Delta U$.
	Then 
	\[
	\Delta_\ell(z_0) = \Delta(z_0)U{\rm e}_\ell =\Delta(z_0)\psi_\ell = 0\qquad \forall \ell \in \{1,2\}.
	\]
	Since each element of $\Delta U$ is  holomorphic, there exist
	vector-valued functions $\widehat \Delta_1$ and $\widehat \Delta_2$ with each entry holomorphic such that 
	$\Delta_1(z) = (z-z_0) \widehat \Delta_1(z)$ and
	$\Delta_2(z) = (z-z_0) \widehat \Delta_2(z)$.
	Thus,
	\begin{align*}
		\det \Delta(z) &= 
		\det (\Delta(z)  U)  \det U^{-1} \\
		&= (z - z_0 )^2 \det
		\begin{bmatrix}
			\widehat \Delta_1(z) &~~ \widehat \Delta_2(z) &~~ \Delta_3(z)
			&~~ \cdots &~~  \Delta_p(s)
		\end{bmatrix}
		\det U^{-1},
	\end{align*}
	which contradicts that $z_0$ is a simple zero.
	
	To prove the second assertion,  we employ Cramer's rule
	\begin{equation}
		\label{eq:cramers_rule}
		\Delta \Delta^{\rm adj} =  \Delta^{\rm adj}  \Delta = \det \Delta \cdot I.
	\end{equation}
	We obtain
	\[
	\Delta^{\rm adj}  (z_0) \Delta (z_0) = \det \Delta (z_0) I = 0.
	\]
	Since $\ker 
	\Delta(z_0)^*= 
	\{
	\alpha \psi :\alpha \in \mathbb{C}
	\}$, it follows that 
	all  the row vectors of $\Delta^{\rm adj}  (z_0)$ can be written as 
	$\alpha \psi^*$ for some $\alpha \in \mathbb{C}$.
	Thus \eqref{eq:adj_vector} holds.
	
	Finally, let us show the existence of a nonzero coefficient $\alpha_\ell$. By contradiction, assume that
	$\alpha_\ell = 0$ in \eqref{eq:adj_vector} for every $\ell \in \{1,\dots,p \}$. Then $\Delta^{\rm adj}(z_0) = 0$.
	Since $\Delta^{\rm adj}$ and $\det \Delta$ are holomorphic, 
	then there exist holomorphic functions $F$ and $f$ such that 
	\begin{equation}
	\label{eq:F_f}
	\Delta^{\rm adj}(s) = (s-z_0) F,\qquad \det \Delta(s) = (s-z_0) f.
	\end{equation}
	Since $z_0$ is a simple zero of $\det \Delta$, it follows that $f(z_0) \not=0$.
	Substituting \eqref{eq:F_f} to Cramer's rule \eqref{eq:cramers_rule}, we obtain
	\[
	\Delta F = f I.
	\]
	It follows that 
	\[
	0 = \psi^* \Delta(z_0) F(z_0)= f(z_0) \psi^*,
	\]
	which contradicts $f(z_0) \not= 0$ and $\psi \not=0$.
	\qed
\end{proof}

The third preliminary lemma provides a stabilizable and 
detectable realization of the series interconnection of two finite-dimensional 
systems.
\begin{lemma}
	\label{lem:no_pole_zero_can}
	For $\ell \in \{1,2\}$, consider the matrix pair $(P_{\ell},Q_\ell,R_\ell,S_\ell)$ with appropriate 
	dimensions and define
	the transfer function
	\[
	\mathbf{K}_\ell(z) := R_\ell(zI - P_\ell)^{-1}Q_\ell + S_\ell.
	\]
	Assume that $\sigma(P_1) \cap \sigma(P_2) \cap \cl(\mathbb{E}_1)= \emptyset$.
	Assume also that 
	$\mathbf{K}_1(\lambda)$ is full column rank for every $\lambda \in \sigma(P_2)\cap 
	\cl(\mathbb{E}_1)$ and that
	$\mathbf{K}_2(\lambda)$ is full row rank for every $\lambda \in \sigma(P_1) \cap 
	\cl(\mathbb{E}_1)$.
	If $(P_\ell,Q_\ell,R_\ell,S_\ell)$ is stabilizable and detectable for $\ell \in \{1,2\}$, then
	the realization of $\mathbf{K}_1\mathbf{K}_2$ given by
	\begin{equation}
		\label{eq:K1K2_realization}
		\left(
		\begin{bmatrix}
			P_1 & ~~Q_1R_2 \\ 0 & ~~P_2
		\end{bmatrix},~
		\begin{bmatrix}
			Q_1 S_2 \\  Q_2
		\end{bmatrix},~
		\begin{bmatrix}
			R_1 &  ~~S_1R_2
		\end{bmatrix},~
		S_1S_2
		\right)
	\end{equation}
	is stabilizable and detectable.
\end{lemma}
\begin{proof}
	It is well known that \eqref{eq:K1K2_realization} is
	a realization of $\mathbf{K}_1\mathbf{K}_2$; see, e.g.,
	Section 3.6 of \cite{zhou1996}.
	It suffices to show that the realization \eqref{eq:K1K2_realization}
	is stabilizable and detectable.
	
	Assume, to reach a contradiction, that the realization \eqref{eq:K1K2_realization}
	is not stabilizable.
	Then there exist an eigenvalue $\lambda \in \sigma(P_1) \cup \sigma(P_2) $ with $|\lambda| \geq 1$ and
	vectors $\psi_1,\psi_2$ such that 
	\[
	\begin{bmatrix}
	\psi_1 \\ \psi_2 
	\end{bmatrix} \not= 0,\quad
	\begin{bmatrix}
	\psi_1^* & ~~\psi_2^*
	\end{bmatrix}
	\begin{bmatrix}
	\lambda I - P_1 & ~~-Q_1 R_2 \\
	0 & ~~\lambda I - P_2
	\end{bmatrix}	= 0,\quad
	\begin{bmatrix}
	\psi_1^* & ~~\psi_2^*
	\end{bmatrix}
	\begin{bmatrix}
	Q_1 S_2 \\  Q_2
	\end{bmatrix}
	=0.
	\]
	
	For the case $\lambda \in\sigma (P_2)$, 
	we obtain $\lambda \in \varrho(P_1)$ by the assumption 
	$\sigma(P_1) \cap \sigma(P_2) \cap \cl(\mathbb{E}_1)= \emptyset$,
	and hence $\psi_1 = 0$ from $\psi_1^*(\lambda I -P_1) = 0$.
	Therefore,
	\[
	\psi_2^* (\lambda I - P_2) = 0, \quad
	\psi_2^*Q_2 = 0. 
	\]
	Using the stabilizability of $(P_2,Q_2)$, we find $\psi_2=0$. This is a contradiction.
	
	Suppose next that $\lambda \in \sigma(P_1)$. Then $\lambda I - P_2$ is invertible
	by the assumption
	$\sigma(P_1) \cap \sigma(P_2) \cap \cl(\mathbb{E}_1)= \emptyset$.
	Therefore,
	\begin{equation}
		\label{eq:w1_w2_relation}
		\psi_2^* = \psi_1^* Q_1R_2 (\lambda I - P_2)^{-1}.
	\end{equation}
	We obtain 
	\[
	\psi_1^*Q_1 \mathbf{K}_2(\lambda)=
	\psi_1^*Q_1(R_2 (\lambda I - P_2)^{-1}Q_2  +S_2) = \begin{bmatrix}
	\psi_1^* & ~~\psi_2^*
	\end{bmatrix}
	\begin{bmatrix}
	Q_1 S_2 \\  Q_2
	\end{bmatrix} = 0.
	\]
	Since $ \mathbf{K}_2(\lambda)$ is full row rank, it follows that $\psi_1^*Q_1 = 0$.
	Together with $\psi_1^*(\lambda I - P_1) = 0$, this implies  $\psi_1= 0$ by
	the stabilizability of $(P_1,Q_1)$. Hence $\psi_2 = 0$ by \eqref{eq:w1_w2_relation}.
	This is a contradiction.
	Thus, the realization \eqref{eq:K1K2_realization}
	is  stabilizable. The detectability of the realization \eqref{eq:K1K2_realization}
	can be obtained in a similar way.
	\qed
\end{proof}

\subsection{Proof of Theorem~\ref{thm:existence_servo_cont}}
\label{subsec:Proof_main_result}
Let us start to prove Theorem~\ref{thm:existence_servo_cont}, by using 
Lemmas~\ref{lem:M2inverse}--\ref{lem:no_pole_zero_can}.
To construct finite-dimensional regulating controllers,
we approximate the infinite-dimensional stable part $\bf G^-$ in \eqref{eq:G_plus_minus}
by
a rational function.
In the next result, the approximation error is used 
to characterize  the norm of a certain matrix,
which will appear in interpolation conditions on the boundary $\mathbb{T}$.
\begin{lemma}
	\label{lem:nonzero_int_cond}
	Assume that {\rm $\langle$a\ref{enu_Resol}$\rangle$--$\langle$a\ref{enu_Simple}$\rangle$} hold.
	Define
	\begin{equation}
		\label{eq:delta_star_def}
		\delta^* := \max \left\{\left\|
		(\mathbf{D}_+\mathbf{G})^{-1}(e^{i\theta_\ell})\right\|_{\mathbb{C}^{p \times p}}	: \ell \in \{1,\dots,n\} \right\}.
	\end{equation}
	For every rational function $\mathbf{R} \in H^{\infty}(\mathbb{E}_1,\mathbb{C}^{p \times p})$ satisfying
	\begin{equation}
		\label{eq:R_cond}
		\| 
		\mathbf{G}^- - \mathbf{R}
		\|_{H^{\infty}(\mathbb{E}_1)}  < \frac{1}{2 \delta^*\|\mathbf{D}_+\|_{H^{\infty}(\mathbb{E}_1)} },
	\end{equation}
	we obtain 
	\begin{equation}
		\label{eq:ND_min}
		\left\|(\mathbf{N}_++ \mathbf{D}_+ \mathbf{R})^{-1}(e^{i\theta_\ell}) \right\|_{\mathbb{C}^{p \times p}}	
		< 2\delta^*\qquad \forall \ell \in \{1,\dots,n  \}.
	\end{equation}
\end{lemma}
\begin{proof}
	The assumption $\langle$a\ref{enu_Resol}$\rangle$ yields $\det \mathbf{D}_+(e^{i\theta_\ell}) \not= 0$ for every
	$\ell \in \{1,\dots,n\}$, 
	which together with $\langle$a\ref{enu_Zero}$\rangle$ implies that $\delta^*$ is well defined.
	Since 
	\[
	\mathbf{G} = \mathbf{G}^+ + \mathbf{G}^- = \mathbf{D}_+^{-1}\mathbf{N}_+ + \mathbf{G}^-,
	\]	
	we have from \eqref{eq:delta_star_def} that 
	\[
	\left\|(\mathbf{N}_++ \mathbf{D}_+ \mathbf{G}^{-})^{-1}(e^{i\theta_\ell}) \right\|_{\mathbb{C}^{p \times p}}	 =
	\left\|(\mathbf{D}_+\mathbf{G})^{-1}(e^{i\theta_\ell})\right\|_{\mathbb{C}^{p \times p}}	 \leq \delta^*
	\]
	for every $\ell \in \{1,\dots,n  \}$.
	Moreover, for every $\ell \in \{1,\dots, n\}$,
	\begin{align*}
		\left\|\mathbf{D}_+(e^{i\theta_\ell})\right\|_{\mathbb{C}^{p \times p}}	 &\leq \|\mathbf{D}_+\|_{H^{\infty}(\mathbb{E}_1)}\\		
		\left\|\mathbf{G}^-(e^{i\theta_\ell}) - \mathbf{R}(e^{i\theta_\ell})\right\|_{\mathbb{C}^{p \times p}}	 &\leq 	
		\| 
		\mathbf{G}^- - \mathbf{R}
		\|_{H^{\infty}(\mathbb{E}_1)}.
	\end{align*}
	Thus we conclude from Lemma~\ref{lem:M2inverse} and \eqref{eq:R_cond} 
	that for every $\ell \in \{1,\dots,n \}$, the matrix 
	$(\mathbf{N}_+ + \mathbf{D}_+\mathbf{R})(e^{i\theta_\ell})$
	is invertible and satisfies
	\begin{align*}
		&\left\|(\mathbf{N}_+ + \mathbf{D}_+\mathbf{R})^{-1}(e^{i\theta_\ell})\right\|_{\mathbb{C}^{p \times p}}\\
		&\qquad \leq 
		\frac{
			\left\|(\mathbf{N}_++ \mathbf{D}_+ \mathbf{G}^{-})^{-1}(e^{i\theta_\ell})\right\|_{\mathbb{C}^{p \times p}}	 
		}
		{
			1 - \left\|(\mathbf{N}_++ \mathbf{D}_+ \mathbf{G}^{-})^{-1}(e^{i\theta_\ell})\right\|_{\mathbb{C}^{p \times p}}	 \cdot
		 \|\mathbf{D}_+\|_{H^{\infty}(\mathbb{E}_1)} \cdot
			\| 
			\mathbf{G}^-  - \mathbf{R}
			\|_{H^{\infty}(\mathbb{E}_1)} 
		} \\
		&\qquad < 2\delta^*,
	\end{align*}
	which is the desired inequality.
	\qed
\end{proof}

For the rational functions $\mathbf{N}_+, \mathbf{D}_+$, which are left coprime over the sets of
rational functions in $H^{\infty}(\mathbb{E}_1,\mathbb{C}^{p \times p})$,
there exists a strictly proper rational function
$\mathbf{Y}_+ \in H^{\infty}(\mathbb{E}_1,\mathbb{C}^{p \times p})$ and 
a rational function $\mathbf{Z}_+\in H^{\infty}(\mathbb{E}_1,\mathbb{C}^{p \times p})$ such that
the B\'ezout identity
\begin{equation}
	\label{eq:bezout}
	\mathbf{N}_+ \mathbf{Y}_+ + \mathbf{D}_+ \mathbf{Z}_+ = I
\end{equation}
holds; see, e.g., Lemma~5.2.9 of \cite{vidyasagar1985} and its proof.
We provide interpolation conditions that such a rational function
$\mathbf{Y}_+ \in H^{\infty}(\mathbb{E}_1,\mathbb{C}^{p \times p})$ satisfies,
as in Theorem~IV.3 of \cite{Wakaiki2014}. To that purpose, we see from 
Lemma~\ref{lem:det_zeros} and $\langle$a\ref{enu_Simple}$\rangle$ that,
for every $r \in \{1,\dots,\Upsilon\}$, there exists a nonzero vector $\psi_r \in \mathbb{C}^{p}$ such that 
$\ker \mathbf{D}_+(\chi_r)^* = \{
\alpha \psi_r: \alpha \in \mathbb{C}
\}$.
\begin{lemma}
	\label{lem:Bezout_interpolation}
	Suppose that {\rm $\langle$a\ref{enu_Resol}$\rangle$--$\langle$a\ref{enu_Simple}$\rangle$} are satisfied.
	A rational function  $\mathbf{Y}_+\in H^{\infty}(\mathbb{E}_1,\mathbb{C}^{p \times p})$ is strictly proper and
	satisfies the B\'ezout identity \eqref{eq:bezout}   for some rational function
	$\mathbf{Z}_+\in H^{\infty}(\mathbb{E}_1,\mathbb{C}^{p \times p})$ if and only if
	the interpolation conditions
	\begin{equation}
		\label{eq:first_pole}
		\mathbf{Y}_+(\infty) = 0,\qquad 
		\psi^*_{r} \mathbf{N}_+(\chi_r)
		\mathbf{Y}_+(\chi_r) = \psi^*_{r} 
		\quad \forall r \in \{1,\dots,\Upsilon  \}
	\end{equation}
	hold. Moreover, if 
	the latter part of 
	the interpolation conditions \eqref{eq:first_pole} holds, then
	a rational function
	\begin{equation}
		\label{eq:Z_plus_def}
		\mathbf{Z}_+ :=
		\frac{\mathbf{D}_+^{\rm adj} - \mathbf{D}_+^{\rm adj} \mathbf{N}_+\mathbf{Y}_+}{\det \mathbf{D}_+}
	\end{equation}
	satisfies $\mathbf{Z}_+\in H^{\infty}(\mathbb{E}_1,\mathbb{C}^{p \times p})$ and the B\'ezout idendity \eqref{eq:bezout}.
\end{lemma}
\begin{proof}
	It is clear that the strict properness of $\mathbf{Y}_+$ is equivalent to
	$\mathbf{Y}_+(\infty) = 0$.
	Suppose that rational functions $\mathbf{Y}_+,
	\mathbf{Z}_+\in H^{\infty}(\mathbb{E}_1,\mathbb{C}^{p \times p})$
	satisfy the B\'ezout identity \eqref{eq:bezout}.
	Using Cramer's rule for $\mathbf{D}_+$,
	we obtain
	\begin{equation}
		\label{eq:Dadj_Bezout}
		\mathbf{D}_+^{\rm adj} = \mathbf{D}_+^{\rm adj} 
		(\mathbf{N}_+ \mathbf{Y}_+ + \mathbf{D}_+ \mathbf{Z}_+) =
		\mathbf{D}_+^{\rm adj} 
		\mathbf{N}_+ \mathbf{Y}_+ + \det \mathbf{D}_+ \cdot \mathbf{Z}_+.
	\end{equation}
	For every $r \in \{1,\dots, \Upsilon\}$, we obtain $\det \mathbf{D}_+(\chi_r) = 0$ and hence
	\[
	\mathbf{D}_+^{\rm adj}(\chi_r) =
	(\mathbf{D}_+^{\rm adj} 
	\mathbf{N}_+ \mathbf{Y}_+)(\chi_r).
	\]
	The second statement of
	Lemma~\ref{lem:det_zeros} shows that $\psi^*_{r} \mathbf{N}_+(\chi_r)
	\mathbf{Y}_+(\chi_r) = \psi^*_{r} $
	for every $r \in \{1,\dots,\Upsilon  \}$.
	
	Conversely, suppose that a rational function
	$\mathbf{Y}_+\in H^{\infty}(\mathbb{E}_1,\mathbb{C}^{p \times p})$
	satisfies
	the interpolation conditions \eqref{eq:first_pole}.
	To show that 
	the B\'ezout identity \eqref{eq:bezout} holds  for some rational function
	$\mathbf{Z}_+\in H^{\infty}(\mathbb{E}_1,\mathbb{C}^{p \times p})$, it suffices to prove that
	$\mathbf{Z}_+$ defined by \eqref{eq:Z_plus_def} satisfies the B\'ezout identity \eqref{eq:bezout}
	and $\mathbf{Z}_+\in H^{\infty}(\mathbb{E}_1,\mathbb{C}^{p \times p})$.
	
	Using Cramer's rule for $\mathbf{D}_+$, we find that 
	$\mathbf{Z}_+$ satisfies the B\'ezout identity \eqref{eq:bezout}.
	By way of contradiction, assume that $\mathbf{Z}_+ \not\in H^{\infty}(\mathbb{E}_1,\mathbb{C}^{p \times p})$.
	Let the $(j,\ell)$th entry $\mathbf{Z}_{+}^{j,\ell}$ of 
	$\mathbf{Z}_+ $ satisfy $\mathbf{Z}_{+}^{j,\ell}\not\in H^{\infty}(\mathbb{E}_1)$.
	By definition, $\mathbf{Z}_{+}^{j,\ell}$ is rational. Using again Cramer's rule for $\mathbf{D}_+$,
	we derive
	\[
	\det \mathbf{D}_+ \cdot \mathbf{Z}_+ =\mathbf{D}_+^{\rm adj}( I -
	\mathbf{N}_+ \mathbf{Y}_+) \in H^{\infty}(\mathbb{E}_1,\mathbb{C}^{p \times p}).
	\]
	Since a rational function $\det \mathbf{D}_+$ is not strictly proper by Theorem~4.3.12 of \cite{vidyasagar1985}, it follows that
	$\mathbf{Z}_{+}^{j,\ell}$ is proper. Therefore, there exists 
	a pole of the rational function $\mathbf{Z}_{+}^{j,\ell}$ in $\cl(\mathbb{E}_1)$ that
	is equal to a zero $\chi_{r_0}$ of $\det \mathbf{D}_+$. 
	Since $\chi_{r_0}$ is a simple zero, it follows that 
	\begin{equation}
		\label{eq:jl_Z}
		\big(\!\det \mathbf{D}_+ \cdot \mathbf{Z}_{+}^{j,\ell}\hspace{1pt}\big)(\chi_{r_0}) \not = 0. 
	\end{equation}
	However,
	by  the latter part of the interpolation conditions \eqref{eq:first_pole} and Lemma~\ref{lem:det_zeros}, we obtain
	\[
	(\det \mathbf{D}_+ \cdot \mathbf{Z}_+)(\chi_{r_0}) =
	\mathbf{D}_+^{\rm adj}(\chi_{r_0}) ( I -
	\big(\mathbf{N}_+ \mathbf{Y}_+)(\chi_{r_0})\big) = 0.
	\]
	This contradicts \eqref{eq:jl_Z}. 
	\qed
\end{proof}

Set $M > 0$ as in
\begin{equation}
	\label{eq:M_condition}
	M > \inf\big\{
	\|Y\|_{H^{\infty}(\mathbb{E}_1)}:
	\text{$\mathbf{Y}_+\in H^{\infty}(\mathbb{E}_1,\mathbb{C}^{p \times p})$ is rational and
		satisfies \eqref{eq:first_pole}}
	\big\}.
\end{equation}
Since there always exists a rational function $\mathbf{Y}_+\in H^{\infty}(\mathbb{E}_1,\mathbb{C}^{p \times p})$
satisfying the interpolation conditions \eqref{eq:first_pole}, 
the right side of \eqref{eq:M_condition} belongs to $\mathbb{R}_+$.

The boundary interpolation conditions in Lemma~\ref{lem:YZ_cond} below
is used for the incorporation of a $p$-copy
internal model.
\begin{lemma}
	\label{lem:YZ_cond}
	Assume that {\rm $\langle$a\ref{enu_Resol}$\rangle$--$\langle$a\ref{enu_Simple}$\rangle$}  hold, and define 
	$\delta^*>0$ by \eqref{eq:delta_star_def}.
	For every rational function $\mathbf{R}\in H^{\infty}(\mathbb{E}_1,\mathbb{C}^{p\times p})$ satisfying \eqref{eq:ND_min},
	there exist a strictly proper rational function $\mathbf{Y}_+ \in H^{\infty}(\mathbb{E}_1,\mathbb{C}^{p\times p})$ and 
	a rational function $\mathbf{Z}_+ \in H^{\infty}(\mathbb{E}_1,\mathbb{C}^{p\times p})$ 
	such that $\mathbf{Y}_+$ satisfies the interpolation conditions
	\begin{subequations}
		\label{eq:Y_cond}
		\begin{align}
			\label{eq:boundary_int}
			&\mathbf{Y}_+(e^{i\theta_\ell}) = (\mathbf{N}_+ + \mathbf{D}_+\mathbf{R})^{-1}(e^{i\theta_\ell})
			\qquad \forall \ell =\{1,\dots,n \} \\
			\label{eq:boundary_int2}
			&\mathbf{Y}_+^{\prime}(e^{i\theta_\ell}) =- (\mathbf{N}_+ + \mathbf{D}_+\mathbf{R})^{-1}(e^{i\theta_\ell})
			\big(\mathbf{D}_+(e^{i\theta_\ell}) \\
			&\hspace{53pt}+ (\mathbf{N}_+ + \mathbf{D}_+\mathbf{R})^{\prime}(e^{i\theta_\ell}) (\mathbf{N}_+ 
			+ \mathbf{D}_+\mathbf{R})^{-1}(e^{i\theta_\ell}) \big)
			~~\quad \forall \ell =\{1,\dots,n \}, \notag
		\end{align}
	\end{subequations}
	the norm condition 
	\begin{align}
		\label{eq:norm_bound}
		\|\mathbf{Y}_+\|_{H^{\infty}(\mathbb{E}_1)} <
		\max 
		\left\{
		2\delta^*, M
		\right\},
	\end{align}
	and  the B\'ezout identity \eqref{eq:bezout} hold.
\end{lemma}
\begin{proof}
	Lemma~\ref{lem:Bezout_interpolation} shows that
	a rational function	$\mathbf{Y}_+ \in H^{\infty}(\mathbb{E}_1, \mathbb{C}^{p\times p})$ is 
	strictly proper and satisfies the B\'ezout identity \eqref{eq:bezout}  for some rational function
	$\mathbf{Z}_+\in H^{\infty}(\mathbb{E}_1, \mathbb{C}^{p\times p})$  if and only if
	the interpolation conditions \eqref{eq:first_pole} hold. Hence
	the problem of finding the desired $\mathbf{Y}_+,\mathbf{Z}_+ \in H^{\infty}(\mathbb{E}_1, \mathbb{C}^{p\times p})$
	is equivalent to that of finding a rational function 
	$\mathbf{Y}_+ \in H^{\infty}(\mathbb{E}_1, \mathbb{C}^{p\times p})$ satisfying
	the interior interpolation conditions \eqref{eq:first_pole}, 
	the boundary interpolation conditions \eqref{eq:Y_cond},
	and the norm condition \eqref{eq:norm_bound}, 
	which is called the  {\em Nevanlinna-Pick interpolation problem with both 
		interior and boundary conditions}; see Appendix~\ref{sec:NPI} for details.
	This interpolation problem is solvable if $R$ satisfies \eqref{eq:ND_min}.
	Once we obtain a solution $\mathbf{Y}_+ \in H^{\infty}(\mathbb{E}_1,\mathbb{C}^{p \times p})$
	of the interpolation problem,
	$\mathbf{Z}_+ \in H^{\infty}(\mathbb{E}_1,\mathbb{C}^{p \times p})$ defined by
	\eqref{eq:Z_plus_def} 
	satisfies the B\'ezout identity \eqref{eq:bezout}.
	\qed
\end{proof}


\begin{lemma}
	\label{eq:K_property}
	Assume that  {\rm $\langle$a\ref{enu_Resol}$\rangle$--$\langle$a\ref{enu_Simple}$\rangle$} hold.
	Suppose that
	a rational function $\mathbf{R}\in H^{\infty}(\mathbb{E}_1,\mathbb{C}^{p \times p})$ satisfies \eqref{eq:ND_min}.
	Let a strictly proper rational function $\mathbf{Y}_+\in H^{\infty}(\mathbb{E}_1,\mathbb{C}^{p \times p})$ 
	and a proper rational function $\mathbf{Z}_+\in H^{\infty}(\mathbb{E}_1,\mathbb{C}^{p \times p})$ satisfy 
	 the interpolation conditions \eqref{eq:Y_cond} and the B\'ezout identity \eqref{eq:bezout}.
	Then the following results hold:
	\begin{enumerate}
		\renewcommand{\labelenumi}{\sf (\alph{enumi})}
		\item
		$\mathbf{Y}_+$ and $\mathbf{Z}_+ - \mathbf{R}\mathbf{Y}_+$
		are right coprime over the set of
		rational functions in $H^{\infty}(\mathbb{E}_1)$.
		\item The rational function defined by
		\begin{equation}
			\label{eq:K_def}
			\mathbf{K} := \mathbf{Y}_+(\mathbf{Z}_+ - \mathbf{R}\mathbf{Y}_+)^{-1}
		\end{equation}
		is strictly proper and satisfies
		\begin{equation}
			\label{eq:K_other}
			\mathbf{K} =\mathbf{Y}_+\big(I - (\mathbf{N}_+ + \mathbf{D}_+\mathbf{R} )\mathbf{Y}_+\big)^{-1}\mathbf{D}_+.
		\end{equation}
		\item 
		There exists a rational function 
		$\widehat{\mathbf{Z}}_+\in H^{\infty}(\mathbb{E}_1,\mathbb{C}^{p \times p})$ such that 
		\begin{subequations}
			\label{eq:tilde_Z}
			\begin{align}
				\label{eq:Z_zeros}
				(\mathbf{Z}_+ - \mathbf{R}\mathbf{Y}_+)(z) &= 
				\prod_{\ell=1}^n(z - e^{i\theta_\ell}) 
				\cdot \widehat{\mathbf{Z}}_+(z) \\
				\label{eq:det_tildeZ}
				\det \widehat{\mathbf{Z}}_+(e^{i\theta_\ell}) &\not= 0.
			\end{align}
		\end{subequations}
	\end{enumerate}
\end{lemma}
\begin{proof}
	{\sf (a)}
	By the B\'ezout identity \eqref{eq:bezout},
	\begin{equation*}
		(\mathbf{N}_+ + \mathbf{D}_+\mathbf{R})\mathbf{Y}_+ + 
		\mathbf{D}_+ (\mathbf{Z}_+ - \mathbf{R} \mathbf{Y}_+) = I.
	\end{equation*}
	Hence $\mathbf{Y}_+$ and $\mathbf{Z}_+ - \mathbf{R}\mathbf{Y}_+$ are right coprime 
	over the sets of
	rational functions in $H^{\infty}(\mathbb{E}_1)$.

	{\sf (b)}
	Since
	$\mathbf{Y}_+(\infty) =0$, it follows from the B\'ezout identity \eqref{eq:bezout} that 
	$\mathbf{Z}_+(\infty)$ is invertible.
	Therefore, $\mathbf{K}(\infty) = 0$ and $\mathbf{K}$ is strictly proper.
	The B\'ezout identity \eqref{eq:bezout} also yields
	\begin{align}
		\label{eq:ZRY}
		\mathbf{Z}_+ - \mathbf{R} \mathbf{Y}_+ = 
		\mathbf{D}_+^{-1} (I - \mathbf{N}_+ \mathbf{Y}_+) - \mathbf{R} \mathbf{Y}_+ 
		=\mathbf{D}_+^{-1} (I - (\mathbf{N}_+ + \mathbf{D}_+\mathbf{R}) \mathbf{Y}_+).
	\end{align}
	Therefore, we obtain \eqref{eq:K_other}.
	
	{\sf (c)}
	To show the existence of a rational function $\widehat{\mathbf{Z}}_+\in H^{\infty}(\mathbb{E}_1,\mathbb{C}^{p \times p})$ 
	satisfying \eqref{eq:tilde_Z}, it suffices to prove
	\begin{align}
		\label{eq:ZRY_diff}
		(\mathbf{Z}_+ - \mathbf{R}\mathbf{Y}_+)(e^{i\theta_\ell}) = 0\quad
		\qquad \forall \ell \in \{1,\dots,n  \}
	\end{align}
	and $(\mathbf{Z}_+ - \mathbf{R}\mathbf{Y}_+)^{\prime}(e^{i\theta_\ell})$ is invertible
	for all $\ell \in \{1,\dots,n  \}$.
	We immediately obtain \eqref{eq:ZRY_diff}  from \eqref{eq:boundary_int} and \eqref{eq:ZRY}.
	We also have
	\begin{align*}
		(\mathbf{Z}_+ - \mathbf{R}\mathbf{Y}_+)^{\prime} (e^{i\theta_\ell})=
		\big(
		\mathbf{D}_+^{-1}
		\big)(e^{i\theta_\ell})
		\big(
		I - (\mathbf{N}_+ + \mathbf{D}_+\mathbf{R}) \mathbf{Y}_+
		\big)^{\prime}(e^{i\theta_\ell})
	\end{align*}
	for every $\ell \in \{1,\dots,n\}$.
	The interpolation condition  \eqref{eq:boundary_int2}  yields
	\begin{align*}
		\big(I-(\mathbf{N}_+ + \mathbf{D}_+\mathbf{R}) \mathbf{Y}_+\big)^{\prime}(e^{i\theta_\ell}) &=
		\mathbf{D}_+(e^{i\theta_\ell})\qquad \forall \ell \in \{1,\dots,n\}.
	\end{align*}
	Thus, $(\mathbf{Z}_+ - \mathbf{R}\mathbf{Y}_+)^{\prime}(e^{i\theta_\ell}) = I$ 
	for all $\ell \in \{1,\dots,n  \}$. This completes the proof.
	\qed
\end{proof}

For $\delta^*$ in \eqref{eq:delta_star_def} and $M$ in \eqref{eq:M_condition},
define 
\begin{equation}
	\label{eq:M1_def}
	M_1 := 	\max 
	\left\{
	2\delta^*, M
	\right\}.
\end{equation}
The following lemma provides a sufficient condition for 
the robust output regulation problem to be solvable.
\begin{lemma}
	\label{lem:G_R_close}
	Assume that {\rm $\langle$a\ref{enu_Resol}$\rangle$--$\langle$a\ref{enu_Simple}$\rangle$} hold.
	Choose a rational function $\mathbf{R} \in H^{\infty}(\mathbb{E}_1,\mathbb{C}^{p \times p})$ so that
	\begin{equation}
		\label{eq:G_R_close}
		\|\mathbf{G}^- - \mathbf{R}\|_{H^{\infty}(\mathbb{E}_1)}  < 
		\frac{1}{M_1\|  \mathbf{D}_+  \|_{H^{\infty}(\mathbb{E}_1)}  }.
	\end{equation}
	Let a strictly proper rational function $\mathbf{Y}_+\in H^{\infty}(\mathbb{E}_1,\mathbb{C}^{p \times p})$ 
	and a proper rational function $\mathbf{Z}_+\in H^{\infty}(\mathbb{E}_1,\mathbb{C}^{p \times p})$ satisfy 
	the interpolation conditions \eqref{eq:Y_cond}, the norm condition
	\eqref{eq:norm_bound}, and
	the B\'ezout identity \eqref{eq:bezout}.
	Then there exists 
	a realization $(P,Q,R)$ of 
	the rational function $\mathbf{K}$ defined by \eqref{eq:K_def} such that
	the controller \eqref{eq:controller} with this realization $(P,Q,R)$ is a solution of Problem~\ref{prob:ROR}.
\end{lemma}
\begin{proof}
	Let a rational function $\mathbf{R} \in H^{\infty}(\mathbb{E}_1,\mathbb{C}^{p \times p})$ satisfy
	\eqref{eq:G_R_close}.
	Since
	\[
	\frac{1}{M_1\| \mathbf{D}_+ \|_{H^{\infty}(\mathbb{E}_1)}} 
	\leq \frac{1}{2\delta^*\|  \mathbf{D}_+  \|_{H^{\infty}(\mathbb{E}_1)} },
	\]
	Lemma~\ref{lem:nonzero_int_cond} shows that $\mathbf{R}$ 
	satisfies \eqref{eq:ND_min}.
	
	Due to Theorem~\ref{thm:internal_model}, it is enough to prove that 
	there exists a realization $(P,Q,R)$ of 
	the rational function $\mathbf{K}$ defined by \eqref{eq:K_def} such that
	$A_{\rm e}$ defined by \eqref{eq:Ae_def} is power stable and \eqref{eq:internal_model} holds.
	
	Let us first find a stabilizable and detectable realization $(P,Q,R)$ of $\mathbf{K}$ satisfying \eqref{eq:internal_model}.
	In the single-input and single-output case $p=1$, 
	 Lemma~A.7.39 of \cite{Curtain1995} directly shows that
	a minimal realization $(P,Q,R)$ of $\mathbf{K}$ satisfies \eqref{eq:internal_model}.
	For the multi-input and multi-output case $p > 1$, we decompose $\mathbf{K}$ and then use 
	Lemma~\ref{lem:no_pole_zero_can}.
	Fix  $a \in (-1,1)$, and
let a strictly proper rational function $\mathbf{Y}_+\in H^{\infty}(\mathbb{E}_1,\mathbb{C}^{p \times p})$ 
and a proper rational function $\mathbf{Z}_+\in H^{\infty}(\mathbb{E}_1,\mathbb{C}^{p \times p})$ satisfy 
the interpolation conditions \eqref{eq:Y_cond}, the norm condition
\eqref{eq:norm_bound}, and
the B\'ezout identity \eqref{eq:bezout}. 
Choose a rational function 
$\widehat{\mathbf{Z}}_+\in H^{\infty}(\mathbb{E}_1,\mathbb{C}^{p \times p})$ satisfying \eqref{eq:tilde_Z}.
	Define 
	\begin{equation}
		\label{eq:K1K2_def}
		\mathbf{K}_1(z) := \prod_{\ell=1}^n\frac{z-a}{z- e^{i\theta_\ell}} I,\quad
		\mathbf{K}_2(z) := \mathbf{Y}_+(z) \big((z-a)^{n} \widehat{\mathbf{Z}}_+(z)\big)^{-1}.
	\end{equation}
	Then $\mathbf{K} = \mathbf{K}_1\mathbf{K}_2$.
	
	For every $\ell \in \{1,\dots,n \}$,
	let $c_\ell \in \mathbb{C}$ be the residue of $\prod_{j=1}^n\frac{z-a}{z- e^{i\theta_j}}$ at $z = e^{i\theta_\ell} $.
	Using the identity matrix $I$ with dimension $p$, we define
	\begin{equation}
		\label{eq:K1_realization}
		P_1 := \text{diag} \big(e^{i\theta_1} I,\dots, e^{i\theta_n}I\big),~
		Q_1:= 
		\begin{bmatrix}
			I \\ \vdots \\ I
		\end{bmatrix},~
		R_1 := 
		\begin{bmatrix}
			c_1 I & \cdots & c_n I
		\end{bmatrix},~
		S_1 := I.
	\end{equation}
	Then $(P_1,Q_1,R_1,S_1)$ is 
	a minimal realization of $\mathbf{K}_1$, and 
	\begin{equation}
	\label{eq:P1_IMP}
	\dim {\rm ker}(e^{i\theta_\ell}I - P_1) \geq p
	\qquad \forall \ell \in \{1,\dots,n \}.
	\end{equation}
	Let $(P_2,Q_2,R_2,S_2)$ be a minimal realization of
	$\mathbf{K}_2$.
	Since 
	$\mathbf{K}_2$ is strictly proper, it follows that  $S_2 = 0$.
	In addition, the realizations
	$(P_1,Q_1,R_1,S_1)$ and $(P_2,Q_2,R_2,S_2)$
	satisfy the conditions in Lemma~\ref{lem:no_pole_zero_can}.
By {\sf (a)} of Lemma~\ref{eq:K_property},
	$\mathbf{Y}_+$ and $(z-a)^{n} \widehat{\mathbf{Z}}_+$ are right coprime.
	Lemma~A.7.39 of \cite{Curtain1995} shows that 
	every $\ell \in \{1,\dots,n \}$ satisfies $e^{i\theta_\ell} \not\in \sigma(P_2)$
	by \eqref{eq:det_tildeZ}. Hence $\sigma(P_1) \cap \sigma(P_2) \cap \cl(\mathbb{E}_1) = \emptyset$.
	By definition,
	$\det \mathbf{K}_1(\lambda) \not=0$ for every
	$\lambda \in \sigma(P_2) \cap \cl(\mathbb{E}_1)$.
	Since the interpolation condition \eqref{eq:boundary_int} implies that 
	$\mathbf{Y}_+(e^{i\theta_\ell})$ is invertible for every $\ell \in \{1,\dots,n \}$,
	it follows that 
	$\det \mathbf{K}_2(\lambda) \not=0$ for every
	$\lambda \in \sigma(P_1) \cap \cl(\mathbb{E}_1) =
	\{
	e^{i\theta_1},\dots, e^{i\theta_n}
	\}$.
	Therefore, Lemma~\ref{lem:no_pole_zero_can} shows that 
	the realization $(P,Q,R)$ of $\mathbf{K}=\mathbf{K}_1\mathbf{K}_2$
	in the form \eqref{eq:K1K2_realization} is stabilizable and detectable.
	By \eqref{eq:P1_IMP}, \eqref{eq:internal_model} is satisfied.
	
	We can see the power stability of $A_{\rm e}$
	from the same argument as in the proofs of 
	Theorem~7 in \cite{Logemann1992} and
	Theorem~9 in \cite{Logemann2013}.
	Using \eqref{eq:G_R_close} and $\|\mathbf{Y}_+\|_{H^{\infty}(\mathbb{E}_1)} < M_1$, we derive
	\begin{align*}
		\|\mathbf{D}_+(\mathbf{G}^- - \mathbf{R}) \mathbf{Y}_+\|_{H^{\infty}(\mathbb{E}_1)} 
		&\leq 
		\|\mathbf{D}_+\|_{H^{\infty}(\mathbb{E}_1)}  \cdot
		 \|\mathbf{G}^- - \mathbf{R} \|_{H^{\infty}(\mathbb{E}_1)}  \cdot
		\|  \mathbf{Y}_+  \|_{H^{\infty}(\mathbb{E}_1)} \\
		&< 1.	
	\end{align*}
	Therefore $\mathbf{U} := (\mathbf{D}_+\mathbf{G}^- +  \mathbf{N}_+ ) \mathbf{Y}_+ 
	+ \mathbf{D}_+ (\mathbf{Z}_+ - \mathbf{R} \mathbf{Y}_+)$
	satisfies
	\[
	\|
	\mathbf{U} - I
	\|_{H^{\infty}(\mathbb{E}_1)} 
	=
	\|\mathbf{D}_+(\mathbf{G}^- - \mathbf{R}) \mathbf{Y}_+\|_{H^{\infty}(\mathbb{E}_1)}  < 1,
	\]
	which yields $\mathbf{U}, \mathbf{U}^{-1} \in {H^{\infty}(\mathbb{E}_1, \mathbb{C}^{p\times p})} $.
	Since 
	\[
	(I + \mathbf{G} \mathbf{K} )^{-1}= 
	(\mathbf{Z}_+- \mathbf{R} \mathbf{Y}_+) \mathbf{U}^{-1}\mathbf{D}_+,\quad 
	\mathbf{G} = \mathbf{D}_+^{-1}(\mathbf{N}_+ + \mathbf{D}_+\mathbf{G}^- ),
	\]
	it follows that 
	\begin{align*}
		\begin{bmatrix}
			I & ~~-\mathbf{K} \\
			\mathbf{G} & ~~I
		\end{bmatrix}^{-1}
		&=
		\begin{bmatrix}
			I-
			\mathbf{K}(I + \mathbf{G} \mathbf{K})^{-1}\mathbf{G} &
			~~~\mathbf{K}
			(I + \mathbf{G} \mathbf{K})^{-1} \\
			-
			(I + \mathbf{G} \mathbf{K})^{-1}\mathbf{G} & 
			~~~(I + \mathbf{G} \mathbf{K})^{-1}
		\end{bmatrix} \\
		&=
		\begin{bmatrix}
			I- \mathbf{Y}_+\mathbf{U}^{-1} (\mathbf{N}_+ + \mathbf{D}_+\mathbf{G}^- )
			&
			~~\mathbf{Y}_+\mathbf{U}^{-1}  \mathbf{D}_+   \\
			-(\mathbf{Z}_+- \mathbf{R} \mathbf{Y}_+)\mathbf{U}^{-1}(\mathbf{N}_+ + \mathbf{D}_+\mathbf{G}^- )
			& 
			~~(\mathbf{Z}_+- \mathbf{R} \mathbf{Y}_+) \mathbf{U}^{-1} \mathbf{D}_+
		\end{bmatrix} \\
		&\in {H^{\infty}(\mathbb{E}_1, \mathbb{C}^{2p\times2p}}).
	\end{align*}
	A routine calculation similar to that for the finite-dimensional case in Lemma~5.3 of \cite{zhou1996}
	shows that
for the transfer functions $\bf G$ of the plant \eqref{eq:plant} 
and ${\bf K}$ of the controller \eqref{eq:controller},
	\[
	\begin{bmatrix}
	I & ~~0 \\ D & ~~I
	\end{bmatrix}
	\begin{bmatrix}
	I & ~~-\mathbf{K} \\
	\mathbf{G} & ~~I
	\end{bmatrix}^{-1}
	\begin{bmatrix}
	I & ~~0 \\ D & ~~I
	\end{bmatrix}
	\]
	is
	the transfer function of the system
	\[
	\left(
	A_{\rm e},~
	\begin{bmatrix}
	B & ~~0 \\ 0  & ~~Q 
	\end{bmatrix}, 
	\begin{bmatrix}
	0 & ~~R \\ -C & ~~0 
	\end{bmatrix},~
	\begin{bmatrix}
	I & ~~0 \\ D & ~~I
	\end{bmatrix}
	\right).
	\]
	Hence
	Theorem~2 of \cite{Logemann1992}
	shows that
	$A_{\rm e}$ is power stable if
	\begin{equation}
		\label{eq:virtual_closed}
		\left(
		A_{\rm e},~
		\begin{bmatrix}
			B & ~~0 \\ 0  & ~~Q 
		\end{bmatrix}
		\right),\quad
		\left(
		\begin{bmatrix}
			0 & ~~R \\ -C & ~~0 
		\end{bmatrix},~
		A_{\rm e}
		\right)	
	\end{equation}
	is stabilizable and detectable, respectively, which 
	is equivalent to the stabilizablity and detectablility of 
	$(A,B,C)$ and $(P,Q,R)$.
	These properties of $(A,B,C)$ follow from $\langle$a\ref{enu_CD}$\rangle$, and
	we have already proved that $(P,Q,R)$ is stabilizable and detectable.
	This completes the proof.
	\qed
\end{proof}

We are now in a position to prove Theorem~\ref{thm:existence_servo_cont}.
\begin{proof}[of Theorem~\ref{thm:existence_servo_cont}]
	Due to
	Lemma~\ref{lem:G_R_close}, it remains to show 
	the existence of a rational function $\mathbf{R} \in H^{\infty}(\mathbb{E}_1,\mathbb{C}^{p \times p})$ 
	satisfying
	\eqref{eq:G_R_close}.
	
	Since $\mathbf{G}^- \in H^{\infty}(\mathbb{E}_\eta,\mathbb{C}^{p \times p})$ for some $\eta \in (0,1)$,
	it follows that 
	the Taylor expansion of $\mathbf{G}^-$ at $\infty$, 
	\[
	\mathbf{G}^- (z) = \sum_{j=0}^\infty G_jz^{-j},
	\]
	converges uniformly in $\mathbb{E}_{1}$, i.e.,
	\[
	\lim_{N \to \infty}
	\sup_{z \in \mathbb{E}_1} 
	\left\|
	\mathbf{G}^-(z) -
	\sum_{j=0}^N G_j z^{-j}
	\right\|_{\mathbb{C}^{p \times p}}	 = 0.
	\]
	Thus \eqref{eq:G_R_close} holds with 
	\[
	\mathbf{R}^-(z) :=  \sum_{j=0}^N G_j z^{-j}
	\]
	for all sufficiently large $N \in \mathbb{N}$.
	\qed
\end{proof}

We summarize the proposed method for the construction of 
finite-dimensional regulating controllers.
The problem of
finding rational functions in the steps 2 and 5 of the procedure below
is called the Nevanlinna-Pick interpolation problem; see
Appendix~\ref{sec:NPI} for details.

\vspace{24pt}
\noindent
\textbf{Design procedure of controllers}
\begin{enumerate}
	\item Obtain a left-coprime factorization $\mathbf{D}_+^{-1}\mathbf{N}_+$ of a rational function $\mathbf{G}^+$
	over the set of rational functions in $H^{\infty}(\mathbb{E}_1)$. 
	\item Find $M >0$ satisfying \eqref{eq:M_condition}.
	\item Set $M_1 > 0$ as in \eqref{eq:M1_def}.
	\item Find a rational function $\mathbf{R} \in H^{\infty}(\mathbb{E}_1,\mathbb{C}^{p \times p})$ satisfying
	the norm condition \eqref{eq:G_R_close}.
	\item Find a rational function $\mathbf{Y}_+ \in H^{\infty}(\mathbb{E}_1,\mathbb{C}^{p \times p})$ satisfying
	the interpolation conditions \eqref{eq:first_pole}, \eqref{eq:Y_cond} and
	the norm condition $\|\mathbf{Y}_+\|_{H^{\infty}(\mathbb{E}_1)} < M_1$.
	\item Define a rational function $\mathbf{Z}_+ \in H^{\infty}(\mathbb{E}_1,\mathbb{C}^{p \times p})$ by 
	\eqref{eq:Z_plus_def}.
	\item Calculate a rational function $\widehat{\mathbf{Z}}_+ \in H^{\infty}(\mathbb{E}_1,\mathbb{C}^{p \times p})$ satisfying
	\eqref{eq:tilde_Z}.
	\item Define the minimal realization $(P_1,Q_1,R_1,S_1)$ as in \eqref{eq:K1_realization} and
	compute a minimal realization $(P_2,Q_2,R_2)$ of $\mathbf{K}_2$ defined by \eqref{eq:K1K2_def}.
	\item Calculate a realization \eqref{eq:K1K2_realization}, which is a realization of a regulating controller.
\end{enumerate}
\mbox{}

In the single-input and single-output case $p=1$, we can remove the assumption $\langle$a\ref{enu_Simple}$\rangle$ and the redundant 
steps 6--8 in the above design procedure. To see this,
let the multiplicity of the zeros $\chi_1,\dots,\chi_\Upsilon$  in $\cl(\mathbb{E}_1)$ of $\det (sI-A^+)$  be
$J_r \in \mathbb{N}$ for  $r\in \{1,\dots,\Upsilon\}$. If $M_1>0$ is sufficiently large,
then there exists a rational function
$\mathbf{Y_+}  \in H^{\infty}(\mathbb{E}_1)$ satisfying the interpolation conditions
\begin{subequations}
	\label{eq:int_SISO}
	\begin{align}
		\tag{\ref{eq:first_pole}$'$a}
		\mathbf{Y}_+ (\infty) &= 0,\quad
		\mathbf{Y}_+(\chi_r) = \frac{1}{\mathbf{N}_+(\chi_r)} \\
		\tag{\ref{eq:first_pole}$'$b}
		\mathbf{Y}_+^{(j)} (\chi_r) &= \frac{-1}{\mathbf{N}_+ (\chi_r) }
		\sum_{\ell=0}^{j-1} 
	\frac{j!}{\ell! (j-\ell)!}
		\mathbf{N}^{(j - \ell)} (\chi_r)  \mathbf{Y}_+^{(\ell)} (\chi_r) 
	\end{align}
\end{subequations}
for all $r\in \{1,\dots,\Upsilon\}$, $j \in \{1,\dots, J_r \}$ and
\begin{align}
	\tag{\ref{eq:Y_cond}$'$}
	\label{eq:bound_SISO}
	\mathbf{Y}_+(e^{i\theta_\ell}) &= \frac{1}{\mathbf{N}_+(e^{i\theta_\ell})  + \mathbf{D}_+(e^{i\theta_\ell})  \mathbf{R}(e^{i\theta_\ell}) }
\end{align}
for all $\ell \in \{1,\dots,n \}$ and 
the norm condition $\|\mathbf{Y}_+\|_{H^{\infty}(\mathbb{E}_1)} < M_1$.
See, e.g., \cite{luxemburg2010} for an algorithm to compute a rational function
$\mathbf{Y_+}  \in H^{\infty}(\mathbb{E}_1)$ satisfying these interpolation and
norm conditions.
For a rational function $\mathbf{R} \in H^{\infty}(\mathbb{E}_1)$ 
satisfying
\eqref{eq:G_R_close}, 
\[
\mathbf{K} := \frac{
	\mathbf{Y}_+\mathbf{D}_+}{1 - (\mathbf{N}_+ + \mathbf{D}_+\mathbf{R} )\mathbf{Y}_+}
\]
is strictly proper, has a pole at $z = e^{i\theta_\ell}$ for all $\ell \in \{1,\dots,n \}$, and
satisfies
\[
\begin{bmatrix}
1 & ~~-\mathbf{K} \\
\mathbf{G} & ~~1
\end{bmatrix}^{-1}\in {H^{\infty}(\mathbb{E}_1, \mathbb{C}^{2\times2}}).
\]
As commented in the proof of Lemma~\ref{lem:G_R_close},
we see from
Lemma~A.7.39 of \cite{Curtain1995} that
a minimal realization $(P,Q,R)$ of $\mathbf{K}$ satisfies \eqref{eq:internal_model}.
Thus,
$(P,Q,R)$ is a realization of a regulating controller.
Since this result can be obtained by a slight modification 
of the argument for the multi-input
multi-output case $p>1$,
we omit the details for the sake of brevity.

\section{Sampled-data output regulation for constant reference and disturbance signals}
\label{sec:SDOR}
In this section, we investigate sampled-data robust output regulation for unstable well-posed systems
with constant reference and disturbance signals.
To this end, we employ the results for discrete-time systems developed in Section~\ref{sec:DTOR}.
However, there remains two issues to be solved: 
\begin{itemize}
\item	{\em
What conditions are required for the original continuous-time system
in order to guarantee the conditions {\em $\langle$a\ref{enu_Resol}$\rangle$--$\langle$a\ref{enu_Simple}$\rangle$}
of the discretized system?
}
\item 
{\em
	Does output regulation at sampling instants imply continuous-time output regulation?
}
\end{itemize}

The main difficulty of the first problem  is to obtain the relationship
between  the transfer function $\mathbf{G}(s)$ of the original continuous-time system
and the transfer function $\mathbf{G}_\tau(z)$ of the discretized system with sampling period $\tau>0$.
We here show that $\mathbf{G}_\tau(1) = \mathbf{G}(0)$.
This equality allows us to check the assumption $\langle$a\ref{enu_Zero}$\rangle$, $\det \mathbf{G}_\tau(1) \not=0$, 
by using only $\mathbf{G}(s)$.
For exponentially stable well-posed systems,
$\mathbf{G}_\tau(1) = \mathbf{G}(0)$
has been proved in Proposition~4.3 of \cite{Logemann1997} and Proposition~3.1 of \cite{Ke2009SCL}.
We extend these results to systems whose unstable part is finite-dimensional.
The point of the proof is to decompose $\mathbf{G}(s)$ into 
the unstable part $\mathbf{G}^+(s)$ and the stable part $\mathbf{G}^-(s)$.

For the second issue,
we first prove that 
the output has the limit as $t \to \infty$ in the ``energy'' sense.
Next, we show that 
this limit coincides the value of the constant reference signal if
 output regulation at sampling instants is achieved.
We further prove that if a smoothing precompensator 
is embedded between the zero-order hold and the plant, then
the output exponentially converges to the constant reference signal in the usual sense 
under a certain regularity condition on the initial states.

In Section~\ref{subsec:prelim_WPS}, we recall briefly some facts on well-posed continuous-time  systems.
In Section~\ref{subsec:CLS_SD}, we introduce sampled-data systems and formulate the problem 
of sampled-data robust output regulation for constant reference and disturbance signals.
We place assumptions on the original continuous-time systems in Section~\ref{subsec:SDassumption} and
reduce them to the assumptions $\langle$a\ref{enu_Resol}$\rangle$--$\langle$a\ref{enu_Simple}$\rangle$ 
on the discretized system in Section~\ref{subsec:properties_DS}.
Finally, Section~\ref{subsec:SDOR} is devoted to solving the sampled-data output regulation problem.

\subsection{Preliminaries on well-posed systems}
\label{subsec:prelim_WPS}
We provide brief preliminaries on well-posed linear systems and
refer the readers to the surveys \cite{Weiss2001, Tucsnak2014} and the book \cite{Staffans2005} for more details.
As a plant, we consider 
a well-posed system $\Sigma$ with state space $X$, input space $\mathbb{C}^p$,
and output space $\mathbb{C}^p$, generating operators $(A,B,C)$, 
transfer function $\mathbf{G}$, and input-output operator $G$.
Here $X$ is a separable complex Hilbert space with norm $\|\cdot\|$ and
$A$ is the generator of
a strongly continuous semigroup $\mathbf{T} = (\mathbf{T}_t)_{t\geq 0}$ on $X$.
The spaces $X_1$ and $X_{-1}$ are the interpolation and extrapolation spaces
associated with $\mathbf{T}$, respectively. For $\lambda \in \varrho(A)$,
the space $X_1$ is defined as $\dom (A)$ endowed with
the norm $\|\zeta \|_{1} := \|(\lambda I - A)\zeta \|$,
and $X_{-1}$ is the completion of $X$ with
respect to the norm $\|\zeta \|_{-1} := \|(\lambda I - A)^{-1}\zeta \|$.
Different choices of $\lambda$ lead to equivalent norms on $X_1$ and $X_{-1}$.
The semigroup $\bf T$ restricts to a strongly continuous semigroup on $X_1$, and
the generator of the restricted semigroup is the part of $A$ in $X_1$. Similarly,
$\bf T$ can be uniquely extended to a strongly continuous semigroup on $X_{-1}$, and
the generator of the extended semigroup is an extension of $A$ with domain $X$.
The restriction and extension of $\bf T$ have the same exponential growth bound
as the original semigroup $\bf T$.
We denote the restrictions and extensions of $\bf T$ and $A$ by the same symbols.
We refer the reader to Section~II.5 of \cite{Engel2000} and  Section~2.10 of \cite{Tucsnak2014}
for more details on the interpolation and extrapolation spaces.

We place the following
conditions for the system node $(A,B,C,\mathbf{G})$  to be well posed:
\begin{itemize}
	\item
	The operator $B$ 
	satisfies $B \in \mathcal{L}(\mathbb{C}^p,X_{-1})$ and is an admissible control operator
	for $\mathbf T$, that is, for every $t \geq 0$, there exists $b_t\geq 0$ such that
	\[
	\left\|
	\int^t_0 {\mathbf T}_{t-s} Bu(s)ds
	\right\| \leq
	b_t \|u\|_{L^2(0,t)}\qquad \forall u\in L^2([0,t],\mathbb{C}^p).
	\]
	\item
	The operator $C$ satisfies
	$C \in \mathcal{L}(X_1,\mathbb{C}^p)$
	and is an admissible observation operator for $\mathbf{T}$, that is, 
	for every $t \geq 0$, there exists $c_t\geq 0$ such that
	\[
	\left(
	\int^t_0 \| C{\mathbf T}_{s}\zeta \|_{\mathbb{C}^{p}}	^2 ds
	\right)^{1/2} \leq 
	c_t \|\zeta \|\qquad \forall \zeta  \in X_1.
	\]

	\item 
	The transfer function $\mathbf{G}:\mathbb{C}_{\omega (\bf T)} \to \mathbb{C}^{p\times p}$ satisfies 
	\begin{equation}
	\label{eq:G_property}
	{\bf G}(s) - {\bf G} (\lambda) = 
	-(s-\lambda)C(sI-A)^{-1}(\lambda I - A)^{-1}B
	\quad
	\forall s,\lambda \in \mathbb{C}_{\omega (\bf T)}
	\end{equation}
	and
	$\mathbf{G} \in H^{\infty}(\mathbb{C}_{\alpha},\mathbb{C}^{p\times p})$
	for every $\alpha > \omega (\mathbf{T})$.
\end{itemize}

The transfer function 
$\mathbf{G}$ may have an analytic extension to a half plane
$\mathbb{C}_{\alpha}$ with $\alpha < \omega(\mathbf{T})$.
If it exists,
we say that $\mathbf{G}$ is holomorphic (meromorphic) on $\mathbb{C}_{\alpha}$ and
use the same symbol $\mathbf{G}$ for an analytic extension to a larger right half plane.
For every $\alpha > \omega (\mathbf{T})$,
the input-output operator $G: L^2_{\rm loc} (\mathbb{R}_+,\mathbb{C}^p) \to L^2_{\rm loc} (\mathbb{R}_+,\mathbb{C}^p)$
satisfies $G \in \mathcal{L} \big(L^2_{\alpha} (\mathbb{R}_+,\mathbb{C}^p),L^2_{\alpha} (\mathbb{R}_+,\mathbb{C}^p)\big)$
and 
\[
\big(
\mathfrak{L}(Gu)(s)
\big) = {\bf G}(s)\big(\mathfrak{L}(u) \big)(s)\qquad
\forall s \in \mathbb{C}_\alpha,~\forall u \in L^2_\alpha (\mathbb{R}_+,\mathbb{C}^p),
\]
where $\mathfrak{L}$ denotes the Laplace transform.

The $\Lambda$-extension $C_{\Lambda}$ of $C$ is defined by
\[
C_{\Lambda} \zeta := 
\lim_{s \to \infty,~\!\!s\in \mathbb{R}} Cs(sI-A)^{-1}\zeta 
\]
with domain $\dom (C_{\Lambda})$ consisting of those $\zeta \in X$ for which the limit exists.
For every $\zeta \in X$, we obtain ${\mathbf T}_t \zeta \in \dom (C_{\Lambda})$ for a.e. $t\geq 0$.
By the admissibility of $C$, for every $t \geq 0$, there exists $c_{t} \geq 0$ such that 
\[
\left(
\int^t_0 \|C_{\Lambda} {\bf T}_s \zeta\|_{\mathbb{C}^{p}}	^2 ds
\right)^{1/2} \leq c_t \|\zeta\|\qquad \forall \zeta \in X.
\]
If we define 
the operator $\Psi :X \to L^2_{\rm loc}(\mathbb{R}_+,\mathbb{C}^p)$ by
\[
(\Psi \zeta)(t) := C_{\Lambda} {\bf T}_t \zeta\qquad \forall \zeta\in X,~{\rm a.e.~} t \geq 0,
\] 
then $\Psi$
satisfies $\Psi \in \mathcal{L}\big(X, L^2_{\alpha}(\mathbb{R}_+,\mathbb{C}^p)\big)$
for every $\alpha > \omega (\bm T)$.
The Laplace transform of $\Phi \zeta$ is given by
$C(sI-A)^{-1}\zeta$ for every $\zeta \in X$ and every $s \in \mathbb{C}_{\omega(\bf T)}$.

Fix $\lambda \in \mathbb{C}_{\omega(\mathbf{T})}$ arbitrarily. Let $x$ and $y$
denote, respectively, the state and output functions of the well-posed system $\Sigma$ with
the initial condition $x(0)=x^0 \in X$ and the input function $u \in L^2_{\rm loc}(\mathbb{R}_+,\mathbb{C}^p)$. 
The state $x$ and the output $y$ satisfy
\begin{equation}
	\label{eq:solution_diff}
	x(t) = \mathbf{T}_t x^0 + \int^t_0 \mathbf{T}_{t-s} Bu(s) ds\qquad \forall t \geq 0,
\end{equation}
$x(t) - (\lambda I - A)^{-1}Bu(t) \in \dom (C_{\Lambda})$ for a.e. $t \geq 0$, and
\begin{subequations}
	\label{eq:well_posed}
	\begin{align}
		\label{eq:diff}
		\dot x(t) &= Ax(t) + Bu(t),\qquad x(0)= x^0 \in X\quad {\rm a.e.~} t \geq 0 \\
		\label{eq:output}
		y(t) &= C_{\Lambda}
		\big(x(t) - (\lambda I - A)^{-1}Bu(t)\big) + \mathbf{G}(\lambda )u(t)\quad {\rm a.e.~} t \geq 0,
	\end{align}
\end{subequations}
where the differential equation \eqref{eq:diff} is interpreted on $X_{-1}$.
We have from \eqref{eq:solution_diff} and \eqref{eq:output} that
for every $u \in L^2_{\rm loc}(\mathbb{R}_+,\mathbb{C}^p)$ and a.e. $t\geq0$, 
the input-output operator $G$ satisfies
\begin{equation}
	\label{eq:IOmap_formula}
	(Gu)(t) = C_{\Lambda}
	\left(
	\int^t_0 \mathbf{T}_{t-s} Bu(s) ds - 
	(\lambda I - A)^{-1}Bu(t)
	\right) + {\bf G}(\lambda)u(t).
\end{equation}

\subsection{Closed-loop system and control objective}
\label{subsec:CLS_SD}
Let $\tau >0$ denote the sampling period.
The zero-order hold operator $\mathcal{H}_{\tau}:F(\mathbb{Z}_+,\mathbb{C}^p)\to 
L^2_{\rm loc}(\mathbb{R}_+,\mathbb{C}^p)$ is defined by
\[
(\mathcal{H}_{\tau}f)(k\tau + t) := f(k) \qquad \forall t \in [0,\tau),~\forall k \in \mathbb{Z}_+.
\]
The generalized sampling operator $\mathcal{S}_{\tau}: L^2_{\rm loc}(\mathbb{R}_+,\mathbb{C}^p)
\to F(\mathbb{Z}_+,\mathbb{C}^p)$ is defined by
\[
(\mathcal{S}_{\tau}g) (k) := \int^\tau_0 w(t) g(k\tau+t) dt \qquad \forall k \in \mathbb{Z}_+,
\] 
where  the scalar weighting function $w$ satisfies $w \in L^2(0,\tau)$ and 
\[
\int^\tau_0 w(t) dt = 1.
\]
The outputs of well-posed systems are in $L^2_{\rm loc}$, and hence
the above type of generalized sampling is reasonable.
Note that controllers connected to the sampler above need to be
strictly causal, i.e., have no feedforward term.

We connect the continuous-time system \eqref{eq:well_posed} and the 
discrete-time controller \eqref{eq:controller} via the following
sampled-data feedback law:
\[
u = \mathcal{H}_{\tau} y_{\rm d} + v\mathds{1}_{\mathbb{R}_+},\qquad 
u_{\rm d} = y_{\rm ref}\mathds{1}_{\mathbb{Z}_+}  - \mathcal{S}_\tau y,
\]
where 
$ y_{\rm ref}\mathds{1}_{\mathbb{Z}_+} $ and 
$v\mathds{1}_{\mathbb{R}_+}$ with $y_{\rm ref} \in \mathbb{C}^p$ and
$v \in \mathbb{C}^p$
are the constant reference  and disturbance signals, respectively.
These signals are
constant, but their values $y_{\rm ref}$ and $v$ are 
unknown when we design controllers.
The dynamics of the sampled-data system is given by
\begin{subequations}
	\label{eq:sampled_data_sys}
	\begin{align}
		\dot x &= Ax + B(\mathcal{H}_{\tau} y_{\rm d} 
		+ v\mathds{1}_{\mathbb{R}_+}),
		\qquad x(0)= x^0 \in X\\\
		y &= C_{\Lambda}\big(x - (\lambda I - A)^{-1}B(\mathcal{H}_{\tau} y_{\rm d} + v\mathds{1}_{\mathbb{R}_+}) \big)+ \mathbf{G}(\lambda )
		(
		\mathcal{H}_{\tau} y_{\rm d} + v\mathds{1}_{\mathbb{R}_+}
		)\\
		x_{\rm d}^{\bigtriangledown}
		&= P x_{\rm d}+ Q(y_{\rm ref}\mathds{1}_{\mathbb{Z}_+}  
		- \mathcal{S}_\tau y),\quad x_{\rm d}(0) = x_{\rm d}^0 \in X_{\rm d} \\
		y_{\rm d}&= R x_{\rm d}.
	\end{align}
\end{subequations}

We define the exponential stability of this sampled-data system.
\begin{definition}[Exponential stability]
	The sampled-data system \eqref{eq:sampled_data_sys} is exponentially stable if
	there exist $\Gamma \geq 1$ and $\gamma >0$ such that 
	\begin{align}
		\label{eq:ES_SD}
		&\left\|
		\begin{bmatrix}
			x(k\tau + t) \\
			x_{\rm d}(k)
		\end{bmatrix}
		\right\|_{X\times X_{\rm d}}\hspace{-3pt}
		\leq 
		\Gamma 
		\left(
		e^{-\gamma(k\tau+t)}
		\left\|
		\begin{bmatrix}
			x^0 \\
			x_{\rm d}^0
		\end{bmatrix}
		\right\|_{X\times X_{\rm d}}\hspace{-3pt}
		+ \|y_{\rm ref}\|_{\mathbb{C}^p} +\|v\|_{\mathbb{C}^p} 
		\right)  \\
		&\qquad\qquad 
		\forall k\in \mathbb{Z}_+,~\forall t \in [0,\tau),~
		\forall x^0 \in X,~\forall x_{\rm d}^0 \in X_{\rm d},~
		\forall y_{\rm ref}, v \in \mathbb{C}^p. \notag
	\end{align}
\end{definition}

We consider a set of perturbed plants $\mathcal{O}_{\rm s}(P,Q,R)$ defined as follows.
\begin{definition}[Set of perturbed plants]
	For given operators $P \in \mathcal{L}(X_{\rm d})$, $Q \in \mathcal{L}(\mathbb{C}^p, X_{\rm d})$, and
	$R \in \mathcal{L}(X_{\rm d}, \mathbb{C}^p)$,
	$\mathcal{O}_{\rm s}(P,Q,R)$ is the set of system nodes
	$(\widetilde{A},\widetilde{B},\widetilde{C},\widetilde{\bf G})$
	satisfying the following two conditions:
	\begin{enumerate}
		\item  
		The operators $(\widetilde{A},\widetilde{B},\widetilde{C})$ and the transfer function $\widetilde{\bf G}$
		generate a well-posed system with state space $X$, input space $\mathbb{C}^p$,
		and output space $\mathbb{C}^p$.

		\item  The perturbed sampled-data system, in which 
		the system node $(A,B,C,{\bf G})$
		is changed to
		$(\widetilde{A},\widetilde{B},\widetilde{C},\widetilde{\bf G})$,
		is exponentially stable.
	\end{enumerate}
\end{definition}

In this section, we study the following sampled-data robust output regulation problem. 
\begin{problem}[Robust output regulation for sampled-data systems]
	\label{prob:OR_SD}
	Find a controller \eqref{eq:controller} such that 
	the following three properties hold for
	the sampled-data system \eqref{eq:sampled_data_sys}:
	\begin{description}
		\item[Stability:]
		The sampled-data system \eqref{eq:sampled_data_sys}  is exponentially stable.
		\item[Tracking:]
		There exist $\Gamma_{\rm ref} >0$ and $\alpha < 0$ such that 
		\begin{align}
			\label{eq:TR_SD}
			\| 
			y - y_{\rm ref}\mathds{1}_{\mathbb{R}_+}
			\|_{L^2_\alpha} \leq 
			\Gamma_{\rm ref} &\left(
			\left\|
			\begin{bmatrix}
				x^0 \\
				x_{\rm d}^0
			\end{bmatrix}
			\right\|_{X\times X_{\rm d}} + \|y_{\rm ref}\|_{\mathbb{C}^p}
			+\|v\|_{\mathbb{C}^p} 
			\right) \\
			&\qquad \forall x^0 \in X,~\forall x_d^0 \in X_d,~
			\forall y_{\rm ref}, v \in \mathbb{C}^p. \notag
		\end{align}

		\item[Robustness:]
		If the system node $(A,B,C,\bf G)$ 
		is changed to
		$(\widetilde{A},\widetilde{B},\widetilde{C},\widetilde{\bf G}) \in \mathcal{O}_{\rm s}(P,Q,R)$,
		then the above tracking property still holds.
	\end{description}
\end{problem}

\subsection{Assumptions on well-posed systems}
\label{subsec:SDassumption}
In what follows, we impose several assumptions on the well-posed system \eqref{eq:well_posed}. 

\begin{enumerate}
	\renewcommand{\labelenumi}{$\langle$b\arabic{enumi}$\rangle$}
	\item $0 \in \varrho(A)$. \label{enu_Resol_SD}
	\item $\det \mathbf{G}(0) \not= 0$. \label{enu_Zero_SD}
	\item There exists $\varepsilon>0$ such that 
	$\sigma(A) \cap  \cl(\mathbb{C}_{-\varepsilon})$ consists of finitely many
	isolated eigenvalues of $A$ with finite algebraic multiplicities. \label{enu_uns_finite}
\end{enumerate}

Under the assumption $\langle$b\ref{enu_uns_finite}$\rangle$, 
we obtain the following spectral decomposition of $X$ for $A$; see, e.g.,
Lemma~2.5.7 of \cite{Curtain1995} or Proposition~IV.1.16 of \cite{Engel2000}.
There exists a rectifiable, closed, 
simple curve $\Phi$ in $\mathbb{C}$ 
enclosing an open set that contains $\sigma(A) \cap  \cl(\mathbb{C}_{0})$ in its interior and
$\sigma(A) \cap  \big( \mathbb{C} \setminus \cl(\mathbb{C}_{0})\big)$
in its exterior. The operator
\begin{equation}
\label{eq:projection}
\Pi := \frac{1}{2\pi i} \int_{\Phi} (sI-A)^{-1}ds
\end{equation}
is a projection on $X$. Define $X^+ := \Pi X$ and $X^- := (I-\Pi)X$.
Then $X = X^+ \oplus X^-$, 
$\dim X^+ < \infty$, and $X^+ \subset X_1$.
The subspaces
$X^+$ and $X^-$ are 
${\mathbf T}_t$-invariant for all $t\geq 0$.

Define 
\[
A^+:= A|_{X^+},\quad 
\mathbf{T}_t^+ := \mathbf{T}_t|_{X^+},\quad
A^-:= A|_{X_1 \cap X^-},\quad
\mathbf{T}_t^- := \mathbf{T}_t|_{X^-}.
\]
Then 
\[
\sigma(A^+) = \sigma(A) \cap \cl(\mathbb{C}_{0}),\quad
\sigma(A^-) = \sigma(A) \cap  \big( \mathbb{C} \setminus \cl(\mathbb{C}_{0})\big),
\]
and
$\mathbf{T}^+ := (\mathbf{T}^+_t)_{t\geq 0}$ and $\mathbf{T}^- := (\mathbf{T}^-_t)_{t\geq 0}$ 
are strongly continuous semigroups on $X^+$ and $X^-$ with
generators $A^+$ and $A^-$, respectively.
The projection operator $\Pi$ on $X$ can be extended to a projection $\Pi_{-1}$
on $X_{-1}$, and $\Pi_{-1}X_{-1}= \Pi X = X^+$. We define
\[
B^+ := \Pi_{-1}B,\quad 
C^+ := C|_{X^+},\quad
B^- := (I-\Pi_{-1})B,\quad
C^- := C|_{X_1\cap X^-}.
\]
We can uniquely extend
the semigroup $\mathbf{T}_t^-$
to a strongly continuous semigroup on $(X^-)_{-1}$,
and the generator of the extended semigroup is an extension of $A^{-}$.
The same symbols $\mathbf{T}_t^-$ and $A^-$ will be used to denote the extensions.
Note that we can identity $(X^-)_{-1}$ and $(X_{-1})^{-} := (I-\Pi_{-1})X_{-1}$ as mentioned
in the footnote 2 on p. 1357 of \cite{Logemann2005}.

We are now in a position to formulate the remaining assumptions.
\begin{enumerate}
	\renewcommand{\labelenumi}{$\langle$b\arabic{enumi}$\rangle$}
	\setcounter{enumi}{3}
	\item The strongly continuous semigroup $\mathbf{T}^- = (\mathbf{T}^-_t)_{t\geq 0}$ 
	is exponentially stable.\label{enu_A_minus_EXS_SD}
	\item $(A^+,B^+,C^+)$ is controllable and observable.\label{enu_CD_SD}
	\item $2\ell \pi i /\tau \not\in \sigma(A^+)$ for every $\ell \in \mathbb{Z} \setminus \{0\}$.  \label{enu_spec_IA}
	\item $\int^{\tau}_0 w(t) e^{\lambda t}dt \not= 0$ for every $\lambda \in \sigma(A^+)$. \label{enu_weight}
	\item $\tau(\lambda - \mu) \not=2\ell \pi i$ for every $\lambda,\mu \in \sigma(A^+)$ 
	and for every $\ell \in \mathbb{Z} \setminus\{0\}$. \label{enu_two_eig_relation}
	\item The zeros of $\det (sI-A^+)$ are simple.  \label{enu_Simple_SD}
\end{enumerate}

As in the discrete-time case,
we assume $\langle$b\ref{enu_Resol_SD}$\rangle$ and 
$\langle$b\ref{enu_Zero_SD}$\rangle$ for output regulation.
For the design of regulating controllers, we place the assumption $\langle$b\ref{enu_Simple_SD}$\rangle$ but
can remove it in the single-input and single-output case $p=1$, as commented in Section~\ref{sec:DTOR}.
Proposition~5 and Theorem~9 of \cite{Logemann2013} show that for the existence of stabilizing controllers,
the conditions $\langle$b\ref{enu_uns_finite}$\rangle$--$\langle$b\ref{enu_weight}$\rangle$ are sufficient, and
the conditions $\langle$b\ref{enu_uns_finite}$\rangle$--$\langle$b\ref{enu_two_eig_relation}$\rangle$ are necessary and sufficient
in the case $p=1$.

Define the input-output operator $G^+$ of the finite-dimensional system $(A^+,B^+,C^+)$ by
\[
(G^+u)(t) := \int^t_0 C^+ e^{A^+(t-s)} B^+ u(s) ds\qquad \forall t \geq 0,~
\forall u \in L^2_{\rm loc}(\mathbb{R}_+,\mathbb{C}^p).
\]
and define $G^- := G- G^+$.
We use the following result on the decomposition of the output:
\begin{lemma}[Lemma~4.2 in \cite{Logemann2005}]
	\label{lem:y_decomp}
	Assume that {\rm $\langle$b\ref{enu_uns_finite}$\rangle$} holds.
	There exists a well-posed system $\Sigma^-$ with generating operator $(A^-,B^-,C^-)$
	and input-output operator $G^- $. For every $x^0 \in X$ and 
	every $u\in L^2_{\rm loc}(\mathbb{R}_+,\mathbb{C}^p)$, the output $y$ of the well-posed system \eqref{eq:well_posed}
	can be written in the form
	\begin{equation}
		\label{eq:y_decomp}
		y(t) = C^+ \Pi x(t) + (C^-)_{\Lambda} {\bf T}_t^- (I- \Pi)x^0 + (G^-u)(t)\qquad
		{\rm a.e.~}t\geq 0.
	\end{equation}
	The $\Lambda$-extension of $C^-$ satisfies
	\begin{equation}
		\label{eq:C_minus_lambda}
		(C^-)_{\Lambda} \zeta = C_{\Lambda}\zeta \qquad
		\forall \zeta \in \dom \big( (C^-)_{\Lambda}\big) = \dom  (C_{\Lambda}) \cap X^-.
	\end{equation}
\end{lemma}

\subsection{Properties of discretized systems}
\label{subsec:properties_DS}
To employ the discrete-time result developed in Section~\ref{sec:DTOR},
we here convert the sampled-data system to a discretized system
and then obtain the properties of the discretized system. 

First, we recall the discrete-time dynamics of the plant combined with
the zero-order hold and the sampler.
Define 
\[
A_{\tau} := \mathbf{T}_{\tau} \in \mathcal{L}(X).
\]
By the admissibility of $B$, the operator $B_{\tau}:
L^2([0,\tau],\mathbb{C}^p) \to X
$ defined by
\[
B_{\tau} g := \int^\tau_0 \mathbf{T}_t Bg(\tau - t)dt\qquad \forall g \in L^2([0,\tau],\mathbb{C}^p)
\]
satisfies $B_\tau \in \mathcal{L}(L^2([0,\tau],\mathbb{C}^p),X)$.
Similarly, by the admissibility of $C$, the operator $C_{\tau}:
X \to \mathbb{C}^p$ defined by
\[
C_{\tau}\zeta := \int^\tau_0 w(t) C_{\Lambda} \mathbf{T}_t \zeta dt
\qquad \forall \zeta \in X
\]
satisfies $C_{\tau} \in \mathcal{L}(X,\mathbb{C}^p)$.
We define the operator $D_{\tau}:L^2([0,\tau],\mathbb{C}^p) \to \mathbb{C}^p$ by
\[
D_{\tau} g := \int^\tau_0 w(t)(Gg)(t)dt\qquad \forall g \in L^2([0,\tau],\mathbb{C}^p),
\]
which satisfies $D_{\tau} \in \mathcal{L}
\big(
L^2([0,\tau],\mathbb{C}^p), \mathbb{C}^p
\big)$.
For simplicity of notation , we set
\begin{align*}
	B_\tau \psi := B_\tau \big(\psi \mathds{1}_{[0,\tau]} \big),\quad 
	D_\tau \psi := D_\tau \big(\psi \mathds{1}_{[0,\tau]} \big)\qquad 
	\forall \psi \in \mathbb{C}^p.
\end{align*}

\begin{lemma}[Lemma~2 of \cite{Logemann2013}]
	\label{lem:discretization_plant}
	Let $u = \mathcal{H}_{\tau}f +g$, where $f \in F(\mathbb{Z}_+,\mathbb{C}^p)$
	and $g \in L^2_{\rm loc}(\mathbb{R}_+,\mathbb{C}^p)$, and let $x^0\in X$.
	Set $x(t)$ as in \eqref{eq:solution_diff}.
	Then 
	\begin{align*}
		x\big((k+1)\tau\big) &= A_\tau x(k\tau) + B_\tau f(k) +B_\tau \mathbf{L}_{k\tau}g \\
		(\mathcal{S}_\tau y)(k) &= C_\tau x(k\tau) + D_\tau f(k) +D_\tau \mathbf{L}_{k\tau}g,
	\end{align*}
	where $\mathbf{L}_{k\tau}g \in L^2([0,\tau],\mathbb{C}^p)$ is defined by
	$(\mathbf{L}_{k\tau}g)(t) = g(k\tau+t)$ for all $t \in [0,\tau]$. 
\end{lemma}

\begin{remark}
Throughout this section, we exploit the discretized system in Lemma~\ref{lem:discretization_plant}.
Another approach for the analysis and synthesis of sampled-data systems is to
lift the plant and then apply a discrete-time technique for the lifted discrete-time plant.
This lifting approach is well established for finite-dimensional systems 
and
has the advantage that one can treat the
intersample behavior of  sampled-data systems in a unified, time-invariant fashion;
see, e.g., \cite{Bamieh1992, Yamamoto1994, Yamamoto1996}.
There are
two major reasons why we do not use the lifting approach in this study. First,
our problem, output regulation for 
constant reference and disturbance signals, is so simple
that we do not need to analyze intersample behaviors of sampled-data systems 
by the lifting approach.
Second, the transfer function of the lifted system is an operator-valued function,
and hence the discrete-time results developed in Section~\ref{sec:DTOR}  is not 
applicable. This is because, to apply the Nevanlinna-Pick interpolation problem,
we consider in Section~\ref{sec:DTOR} discrete-time systems whose transfer function is matrix-valued.
\end{remark}

We provide two lemmas on the discretized system. These lemmas
will be used to guarantee that the assumptions 
$\langle$a\ref{enu_Resol}$\rangle$--$\langle$a\ref{enu_Simple}$\rangle$
introduced in
Section~\ref{sec:DTOR} are satisfied for the discretized system.

Define 
\begin{align*}
A_\tau^+ &:= \mathbf{T}_\tau^+ = A_{\tau}|_{X^+},\quad
B_\tau^+ := \Pi B_\tau,\quad
C_\tau^+:= C_\tau|_{X^+} \\
A_\tau^- &:= \mathbf{T}_\tau^- = A_{\tau}|_{X^-},\quad
B_\tau^- := (I-\Pi) B_\tau,\quad
C_\tau^-:= C_\tau|_{X^-}
\end{align*}
and $D_{\tau}^+ :L^2([0,\tau],\mathbb{C}^p) \to \mathbb{C}^p$ by
\[
D_{\tau}^+ g := \int^\tau_0 w(t)(G^+g)(t)dt\qquad \forall g \in L^2([0,\tau],\mathbb{C}^p).
\]
For $\psi \in \mathbb{C}^p$, we also set
\[
B_{\tau}^+ \psi :=  B_{\tau}^+ (\psi  \mathds{1}_{[0,\tau]}),\quad
B_{\tau}^- \psi :=  B_{\tau}^- (\psi  \mathds{1}_{[0,\tau]}),\quad
D_{\tau}^+\psi := D_{\tau}^+ (\psi \mathds{1}_{[0,\tau]}).
\]
Let 
$\eta \in \big( 
e^{\tau \omega(T^{-})},1
\big)$.
On $\mathbb{E}_{\eta} \cap \varrho(A_{\tau})$, we define the transfer function $\mathbf{G}_{\tau}$ of the discretized system by
\begin{equation}
\label{eq:TF_DS}
\mathbf{G}_{\tau} (z) := C_\tau (zI - A_\tau)^{-1}B_\tau + D_\tau.
\end{equation}

The first lemma provides a property of the resolvent set of $A_{\tau}$.
\begin{lemma}
	\label{lem:1_Atau}
	If {\rm $\langle$b\ref{enu_Resol_SD}$\rangle$}, {\rm $\langle$b\ref{enu_uns_finite}$\rangle$},
	{\rm $\langle$b\ref{enu_A_minus_EXS_SD}$\rangle$}, and {\rm $\langle$b\ref{enu_spec_IA}$\rangle$} hold, then $1 \in \varrho(A_\tau)$.
\end{lemma}
\begin{proof}
	Since $X^+$ and $X^-$ are 
	$A_{\tau}$-invariant, it is enough to show that $1 \in \varrho(A_\tau^+) \cap \varrho(A_\tau^-)$.
	By {\rm $\langle$b\ref{enu_Resol_SD}$\rangle$}, we obtain
	$0 \in \varrho(A^+)$. 
	Together with {\rm $\langle$b\ref{enu_spec_IA}$\rangle$}, this yields 
	$2\ell \pi i /\tau \not\in \sigma(A^+)$ for every $\ell \in \mathbb{Z}$.
	By the spectral mapping theorem,
	\begin{equation}
		\label{eq:SMT}
		\sigma\big(e^{\tau A^+}\big) = e^{\tau \sigma(A^+)}.
	\end{equation}
	Therefore, $1 \not\in \sigma\big(e^{\tau A^+}\big) = \sigma (A_\tau^+)$. On the other hand, 
	{\rm $\langle$b\ref{enu_A_minus_EXS_SD}$\rangle$} leads to the power stability of $A_\tau^-$, and hence
	$1 \in \varrho(A_\tau^-)$. This completes the proof
	\qed
\end{proof}

The second lemma gives a relationship between
the transfer functions of the original continuous-time system
and the discretized system. This result will be used to verify
the assumption $\langle$a\ref{enu_Zero}$\rangle$ on
the discretized system as well as to obtain $\delta^*$ in \eqref{eq:delta_star_def}.
\begin{lemma}
	\label{lem:G_Gtau}
	If {\rm $\langle$b\ref{enu_Resol_SD}$\rangle$}, {\rm $\langle$b\ref{enu_uns_finite}$\rangle$}, and 
	{\rm $\langle$b\ref{enu_A_minus_EXS_SD}$\rangle$} hold, then 
	$\mathbf{G}_\tau(1) = \mathbf{G}(0)$.
\end{lemma}
\begin{proof}
	Define
	\[
	\mathbf{G}^+ (s) := C^+ (sI - A^+)^{-1}B^+, \qquad
	\mathbf{G}^- (s) := \mathbf{G} (s) - \mathbf{G}_+(s) .
	\]
	Clearly, 
	$\mathbf{G}^+$ is the transfer function of a finite-dimensional system with
	generating matrices
	$(A^+,B^+,C^+)$ and input-output operator $G^+$.
	By
	Lemma~\ref{lem:y_decomp},
	$\mathbf{G}^-$ is the transfer function of the exponentially stable well-posed system 
	with generating operators $(A^-,B^-,C^-)$  and input-output operator $G^- $.
	
	We first show that
	\begin{equation}
		\label{eq:G+0}
		\mathbf{G}^+(0)\psi = -C_{\tau}^+(A^+)^{-1}B^+\psi  + D_\tau^+\psi \qquad \forall \psi  \in \mathbb{C}^p,
	\end{equation}
	where $A^+$ is invertible by {\rm $\langle$b\ref{enu_Resol_SD}$\rangle$}. Since if $g(t) \equiv \psi \in \mathbb{C}^p$,
	then
	\[
	(G^+ g)(t) = C^+ (e^{A^+ t} - I) (A^+)^{-1}B^+\psi,
	\]
	it follows from $\int^\tau_0 w(t)dt = 1$ that 
	\begin{align*}
		D_{\tau}^+\psi  &=
		C_\tau^+ (A^+)^{-1}B^+\psi - C^+(A^+)^{-1}B^+\psi\qquad \forall \psi \in \mathbb{C}^p.
	\end{align*}
	Thus, \eqref{eq:G+0} holds.
	
	Since $\mathbf{T}^-$ is exponentially stable by {\rm $\langle$b\ref{enu_A_minus_EXS_SD}$\rangle$}, $A^-$ is boundedly invertible.
	Next we shall prove that
	\begin{equation}
		\label{eq:G-0}
		\mathbf{G}^-(0)\psi  = -C_{\tau}^-(A^-)^{-1}B^{-}\psi  + D_\tau\psi  - D_\tau^+\psi \qquad \forall \psi  \in \mathbb{C}^p.
	\end{equation}
	By definition,
	\[
	D_\tau g - D_\tau^+ g = 
	\int^\tau_0 w(t) (G^-g)(t)dt\qquad \forall g\in L^2([0,\tau],\mathbb{C}^p). 
	\]
	Similarly to \eqref{eq:IOmap_formula}, we obtain
	\begin{align*}
		&(G^-g)(t) = 
		(C^-)_{\Lambda}	\left(
		\int^t_0 \mathbf{T}^-_{s} B^- g(t-s)ds + (A^-)^{-1}B^-g(t)
		\right)
		+\mathbf{G}^-(0)g(t) \\
		&\hspace{180pt}
		\forall g \in L^2_{\rm loc}(\mathbb{R}_+,\mathbb{C}^p),~{\rm a.e.~} t\geq 0,
	\end{align*}
	Using
	\[
	\int^t_0 \mathbf{T}^-_s  B^- \psi ds 
	= \mathbf{T}^-_t (A^-)^{-1}B^- \psi -  (A^-)^{-1}B^- \psi \qquad \forall \psi \in \mathbb{C}^p,
	\]
	and $\int^\tau_0 w(t)dt = 1$, we obtain
	\begin{align*}
		\int^\tau_0 w(t) \big(G^- (\psi \mathds{1}_{[0,\tau]})\big)(t)dt 
		= 
		\int^\tau_0 w(t) 
		(C^-)_{\Lambda}	
		\mathbf{T}^-_t (A^-)^{-1}B^- \psi dt
		+\mathbf{G}^-(0)\psi
	\end{align*}
	for every $\psi \in \mathbb{C}^p$.
	By \eqref{eq:C_minus_lambda},
	\begin{align*}
		D_\tau \psi - D_\tau^+ \psi 
		&= 
		\int^\tau_0 w(t)
		C_{\Lambda} \mathbf{T}_t (A^-)^{-1}B^- \psi dt 
		+\mathbf{G}^-(0)\psi\\
		&=C_\tau^{-}(A^-)^{-1}B^- \psi+\mathbf{G}^-(0)\psi \qquad \forall \psi \in \mathbb{C}^p,
	\end{align*}
	and \eqref{eq:G-0} holds.

	By definition
	\[
	\mathbf{G}_\tau(z)\psi  = C_\tau^+ (zI - A_\tau^+)^{-1}B_\tau^+\psi  + C_\tau^- (zI - A_\tau^-)^{-1}B_\tau^-\psi  + D_\tau\psi 
	\]
	for every $\psi  \in \mathbb{C}^p$
	and every $z \in \mathbb{E}_{\eta} \cap \varrho(A_{\tau})$ with $\eta \in \big( 
	e^{\tau \omega(T^{-})},1
	\big)$.
	Combining \eqref{eq:G+0}, \eqref{eq:G-0}, and 
	\begin{align*}
	B_{\tau}^+ \psi = (A_{\tau}^+ - I) (A^+)^{-1} B^+ \psi,\quad 
	B_{\tau}^- \psi = (A_{\tau}^- - I) (A^-)^{-1} B^- \psi \qquad \forall \psi \in \mathbb{C}^p,
	\end{align*}
	we obtain
	\begin{align*}
		\mathbf{G}_\tau(1) \psi 
		&=
		C_\tau^+ (I - A_\tau^+)^{-1}B_\tau^+ \psi 
		+
		C_\tau^- (I - A_\tau^-)^{-1}B_\tau^- \psi  + D_\tau  \psi \\
		&=
		- C_\tau^+ (A^+)^{-1}B^+ \psi   
		- C_\tau^- (A^-)^{-1}B^- \psi  + D_\tau  \psi \\
		&=
		(\mathbf{G}^+(0) - D_\tau^+)\psi   + (\mathbf{G}^-(0) - D_\tau + D_\tau^+)\psi  
		+ D_\tau  \psi  \\
		&= \mathbf{G}^+(0)\psi +  \mathbf{G}^-(0) \psi = \mathbf{G}(0)\psi \qquad \forall \psi \in \mathbb{C}^p.
	\end{align*}
	Thus we obtain $\mathbf{G}_\tau(1)=\mathbf{G}(0)$.
	\qed
\end{proof}

\subsection{Output regulation by a finite-dimensional digital controller}
\label{subsec:SDOR}
Using Theorem~\ref{thm:existence_servo_cont},
here we present two results on sampled-data output regulation for 
constant reference and disturbance signals.
First, we show that 
the output converges to the constant reference signal in the ``energy'' sense.
Next, 
we consider sampled-data systems with smoothing precompensators.
The output of such a sampled-data system is continuous
under a certain regularity condition on the initial states.
Hence we can prove that the output exponentially converges to
the constant reference signal in the usual sense.

The following lemma, which is a part of Proposition 3 in \cite{Logemann2013}, 
connects the power stability of the discretized system
and the exponential stability of the sampled-data system.
\begin{lemma}[Proposition 3 in \cite{Logemann2013}]
	\label{lem:PS_ES}
	The sampled-data system \eqref{eq:sampled_data_sys} is exponentially stable
	if and only if the operator $A_{\rm e}$ defined by
	\begin{align}
		\label{eq:Discretized_CLS_A}
	A_{\rm e} :=
	\begin{bmatrix}
	A_\tau & ~~B_\tau R \\
	-QC_\tau  & ~~P - QD_\tau R
	\end{bmatrix}
	\end{align}
	is power stable.
\end{lemma}

\begin{theorem} 
	\label{thm:servo_tracking}
	Assume that {\rm $\langle$b\ref{enu_Resol_SD}$\rangle$--$\langle$b\ref{enu_Simple_SD}$\rangle$} hold.
	There exists a finite-dimensional controller \eqref{eq:controller} 
	that is a solution of Problem~\ref{prob:OR_SD}.
\end{theorem}
\begin{proof}
	One can say that the constant reference and disturbance signals 
	$y_{\rm ref} ,v  \in \mathbb{C}^p$ are 
	generated from the exosystem \eqref{eq:exosystem} with $S = 1$:
	\begin{subequations}
		\label{eq:exosystem_SD}
	\begin{align}
		\xi^{\bigtriangledown}(k) &= \xi(k),\qquad \xi(0) = \xi^0 \in  \mathbb{C}\\
		v(k) &= E\xi(k) \\
		y_{\rm ref}(k) &= F\xi(k)
	\end{align}
	\end{subequations}
	for some unknown constant matrices $E\in \mathbb{C}^{p\times 1}$ and 
	$F \in \mathbb{C}^{p\times 1}$.
	Since
	\[
u= \mathcal{H}_{\tau}y_{\rm d}+ v\mathds{1}_{\mathbb{R}_+} 
	\]
	Lemma~\ref{lem:discretization_plant} yields the following closed-loop dynamics at sampling instants:
	\begin{subequations}
		\label{eq:closed_sample}
		\begin{align}
			\label{eq:closed_state_sample}
			x_{\rm e}^{\bigtriangledown}(k)&= A_{\rm e}x_{\rm e}(k) + B_{\rm e} \xi^0,\quad x_{\rm e}(0) = x_{\rm e}^0 \\
			\label{eq:closed_error_sample}
			e(k) &= C_{\rm e}x_{\rm e}(k) + D_{\rm e}\xi^0,
		\end{align}
	\end{subequations}
	where $e(k) := y_{\rm ref}- (\mathcal{S}_\tau y)(k)$, $x_{\rm e}(k) := 
	\begin{bmatrix}
	x(k\tau) \\ x_{\rm d}(k)
	\end{bmatrix}$, 
	$x_{\rm e}^0 := 
	\begin{bmatrix}
	x^0\\ x_{\rm d}^0
	\end{bmatrix}$, $A_{\rm e}$ is defined by \eqref{eq:Discretized_CLS_A}, and
	\begin{align}
		B_{\rm e} :=
		\begin{bmatrix}
			B_\tau E \\
			Q(F-D_\tau E)
		\end{bmatrix},\quad 
		C_{\rm e} :=
		-
		\begin{bmatrix}
			C_\tau  &
			~~D_\tau R
		\end{bmatrix},\quad
		D_{\rm e} :=
		F-D_\tau E.
	\end{align}
	
	To employ the discrete-time result, Theorem~\ref{thm:existence_servo_cont},
	we first show that the assumptions in Theorem~\ref{thm:existence_servo_cont}
	are satisfied for
	the discrete-time plant $(A_\tau,B_\tau,C_\tau,D_\tau)$.
	By Lemmas~\ref{lem:1_Atau} and \ref{lem:G_Gtau}, we find that
	\begin{enumerate}
		\renewcommand{\labelenumi}{$\langle$a\arabic{enumi}$^\prime$$\rangle$}
		\item $1 \in \varrho(A_\tau)$;
		\item $\det \mathbf{G}_\tau(1) \not= 0$. 
	\end{enumerate}
	The assumption $\langle$b\ref{enu_uns_finite}$\rangle$
	implies that 
	\begin{enumerate}
		\renewcommand{\labelenumi}{$\langle$a\arabic{enumi}$^\prime$$\rangle$}
		\setcounter{enumi}{2}
		\item There exist subspaces $X^+$ and $X^-$ with $\dim X^+ < \infty$ such that 
		$X = X^+ \oplus X^-$. 
		\item $A_\tau X^+ \subset X^+ $ and $A_\tau X^- \subset X^-$
	\end{enumerate}
	By $\langle$b\ref{enu_uns_finite}$\rangle$--$\langle$b\ref{enu_two_eig_relation}$\rangle$,
	the following conditions hold: 
	\begin{enumerate}
		\renewcommand{\labelenumi}{$\langle$a\arabic{enumi}$^\prime$$\rangle$}
		\setcounter{enumi}{4}
		\item $\sigma(A_\tau) \cap \cl(\mathbb{E}_1)$ consists of
		finitely many eigenvalues with finite algebraic multiplicities, $\sigma(A^+_\tau) = 
		\sigma(A_\tau) \cap \cl(\mathbb{E}_{1}) $, and there exists $\eta \in (0,1)$ such that $\sigma(A^-_\tau) = \sigma(A_\tau) \cap 
		\big(\mathbb{C} \setminus \cl(\mathbb{E}_{\eta}) \big)$. 
		\item  $(A^+_\tau,B^+_\tau,C^+_\tau)$ is controllable and observable. \label{enu_CD_SD_dis}
	\end{enumerate}
	Here we used  Proposition~5 and Theorem~9 in \cite{Logemann2013} to see that
	$\langle$A\ref{enu_CD_SD_dis}$^\prime$$\rangle$ holds.
	Finally we find from $\langle$b\ref{enu_two_eig_relation}$\rangle$, $\langle$b\ref{enu_Simple_SD}$\rangle$, and 
	the spectral mapping theorem \eqref{eq:SMT} that 
	\begin{enumerate}
		\renewcommand{\labelenumi}{$\langle$a\arabic{enumi}$^\prime$$\rangle$}
		\setcounter{enumi}{6}
		\item 
		The zeros of $\det (zI-A^+_\tau)$ are simple. 
	\end{enumerate}
	Thus,
	Theorem~\ref{thm:existence_servo_cont} shows
	the existence of a finite-dimensional controller  that
	is a solution of the robust output regulation problem, Problem~\ref{prob:ROR},
	for
	the discrete-time plant $(A_\tau,B_\tau,C_\tau,D_\tau)$ and the exosystem \eqref{eq:exosystem_SD}.
	The power stability of $A_{\rm e}$ is equivalent to the exponential stability \eqref{eq:ES_SD}
	by Lemma~\ref{lem:PS_ES}.
	
	We next show that the tracking property holds.
	Let $x^0 \in X$,  $x_{\rm d}^0 \in X_{\rm d}$, and 
	$y_{\rm ref}, v  \in \mathbb{C}^p$ be given.
	Since $A_{\rm e}$ is power stable, it follows that $(I-A_{\rm e})$ is invertible.
	By \eqref{eq:closed_state_sample}, 
	\begin{align*}
	x_{\rm e}^{\bigtriangledown}(k) - 
	(I-A_{\rm e})^{-1} B_{\rm e} \xi^0 &=
	A_{\rm e} x_{\rm e}(k) + 
	(
	I - (I-A_{\rm e})^{-1}
	) B_{\rm e}  \xi^0 \\ &=
	A_{\rm e} (
	x_{\rm e}(k) - (I-A_{\rm e})^{-1}
	B_{\rm e}  \xi^0
	)\qquad \forall k \in \mathbb{Z}_+.
	\end{align*}
	Using again the power stability of $A_{\rm e}$, we find that
	there exist $\Gamma_1 > 0$ and $\rho \in (0,1)$ such that 
	\begin{equation}
		\label{eq:xe_conv}
		\|
		x_{\rm e}(k) - (I-A_{\rm e})^{-1} B_{\rm e} \xi^0
		\|_{X\times X_{\rm d}} \leq \Gamma_1 \rho^k\big(
		\|x_{\rm e}^0\|_{X\times X_{\rm d}} +\|y_{\rm ref}\|_{\mathbb{C}^p}  + \|v\|_{\mathbb{C}^p}
		\big).
	\end{equation}
	Define 
	\[
	\begin{bmatrix}
	x^{\infty} \\ x_{\rm d}^{\infty}
	\end{bmatrix}
	:= (I - A_{\rm e})^{-1}B_{\rm e}\xi^0,\quad
	u^{\infty} := Rx_{\rm d}^{\infty} + v.
	\]
	As shown in the proof of Theorem~10 in \cite{Logemann2013},
	we have from the assumptions $\langle$b\ref{enu_uns_finite}$\rangle$, $\langle$b\ref{enu_A_minus_EXS_SD}$\rangle$, 
	and
	$\langle$b\ref{enu_spec_IA}$\rangle$ that
	\begin{equation}
\label{eq:x_u_infty}
Ax^{\infty} + Bu^{\infty} = 0
\end{equation}
and
	\[
	x^{\infty} = {\bf T}_t x^{\infty} + \int^t_0 {\bf T}_s Bu^{\infty} ds\qquad \forall t \in [0,\tau].
	\]
	Since 
	\[
	x(k\tau+t) = {\bf T}_t x(k\tau) + \int^t_0 {\bf T}_s B (Rx_{\rm d}(k)+v) ds\qquad \forall t \in [0,\tau],~
	\forall k \in \mathbb{Z}_+,
	\]
	together with the admissibility of $B$ (or Lemma~2.2 of \cite{Logemann2003}),
	\eqref{eq:xe_conv} implies that there exists $\Gamma_2 >0$  such that  
	\begin{align*}
		\big\|
		x(k\tau+t) - x^{\infty}
		\big\|
		&\leq 
		\|{\bf T}_t\| \cdot \big\|x(k\tau) -x^{\infty} \big\| + 	
		\left\|
		\int^t_0 
		{\bf T}_s BR(x_{\rm d}(k) - x_{\rm d}^{\infty}) ds
		\right\| \\
		&\leq \Gamma_2 \rho^k
		\big(
		\|x_{\rm e}^0\|_{X\times X_{\rm d}} + \|y_{\rm ref}\|_{\mathbb{C}^p } +\|v\|_{\mathbb{C}^p}
		\big)
	\end{align*}
	for all $t \in [0,\tau]$ and all $k \in \mathbb{Z}_+$.
	Using \eqref{eq:xe_conv} again, we have that for 
	$\Gamma_3 := \|R\| \Gamma_1$,
	\[
	\big\|
	u(k\tau+t)- u^{\infty}
	\big\|_{\mathbb{C}^p}\leq 
	\Gamma_3 \rho^k	\big(
	\|x_{\rm e}^0\|_{X\times X_{\rm d}} + \|y_{\rm ref}\|_{\mathbb{C}^p  }+ \|v\|_{\mathbb{C}^p}
	\big)
	\]
	for all $t \in [0,\tau]$ and all $k \in \mathbb{Z}_+$.
	Therefore, there exist $\Gamma_4>0$ and $\alpha_1 < 0$ such that 
	\begin{equation}
		\label{eq:x_u_L2}
		\|x - x^{\infty}\mathds{1}_{\mathbb{R}_+}\|_{L^2_{\alpha_1 }} 
		+ \|u - u^{\infty}\mathds{1}_{\mathbb{R}_+}\|_{L^2_{\alpha_1 }}\leq 
		\Gamma_4 
		\big(
		\|x_{\rm e}^0\|_{X\times X_{\rm d}}  + \|y_{\rm ref}\|_{\mathbb{C}^p }+\|v\|_{\mathbb{C}^p}
		\big).
	\end{equation}

	Define 
	\[
	x^{\infty}_- := (I-\Pi)x^{\infty},\quad
	x^0_- := (I-\Pi)x^{0},\quad
	y^\infty := {\bf G}^{-}(0) u^{\infty} +  C^+\Pi x^{\infty}.
	\]
	Recall that the output $y$ can be written in the form \eqref{eq:y_decomp}.
	Then we obtain
	\begin{equation}
	\label{eq:y_y_inf}
	y(t) - y^{\infty}\mathds{1}_{\mathbb{R}_+}
	=
	y_1(t) + y_2(t) + y_3(t)\qquad
	{\rm a.e.~}t\geq 0,
	\end{equation}
	where
	\begin{align*}
	y_1 &:= (C^-)_{\Lambda} {\bf T}^-  x_-^{\infty} + 
	G^- (u^{\infty}\mathds{1}_{\mathbb{R}_+}) -   {\bf G}^{-}(0) u^{\infty}\mathds{1}_{\mathbb{R}_+}  \\
	y_2 &:= (C^-)_{\Lambda} {\bf T}^-(x_-^0 - x_-^{\infty}) 
	+
	G^- (u - u^{\infty}\mathds{1}_{\mathbb{R}_+}) 
	\\
	y_3 &:= 
	C^+\Pi (x - x^{\infty}\mathds{1}_{\mathbb{R}_+}).
\end{align*}
	By \eqref{eq:x_u_infty},
\begin{equation*}
A^- x^{\infty}_- + B^- u^{\infty} =	
(I - \Pi_{-1}) (Ax^{\infty} + Bu^{\infty}) = 0.
\end{equation*}
	Since \eqref{eq:G_property} yields
\begin{align*}
\mathcal{L}(
G^- (u^{\infty}\mathds{1}_{\mathbb{R}_+}) -   {\bf G}^{-}(0) u^{\infty} \mathds{1}_{\mathbb{R}_+}
)(s)
&=
\frac{{\bf G}^{-}(s) - {\bf G}^{-}(0)}{s} u^{\infty}\\
&=
C^{-}(sI-A^-)^{-1}(A^-)^{-1} B^-u^{\infty}
\end{align*}
for every $s \in \mathbb{C}_0$,
the Laplace transform of $y_1$ satisfies
\[
\mathfrak{L}(y_1)(s) 
=
C^{-}(sI-A^-)^{-1}(A^-)^{-1} (A^-x^\infty_- + B^-u^{\infty}) = 0
\qquad \forall s \in \mathbb{C}_0.
\]
The uniqueness of the Laplace transform (see, e.g., 
Theorem~1.7.3 in \cite{Arendt2001}) yields
\begin{equation}
\label{eq:y1}
y_1(t) = 0\qquad \text{a.e. $t \geq 0$.}
\end{equation}

By the exponential stability of ${\bf T}_t^-$
and the admissibility of $C^-$, 
there exists $\Gamma_5 >0$ and $\alpha_2 < 0$ such that 
\begin{equation}
\label{eq:y2}
\|y_2\|_{L^2_{\alpha_2}} \leq \Gamma_5 \big(
\|x^0 - x^\infty\| + \|u - u^\infty\mathds{1}_{\mathbb{R}_+} \|_{L^2_{\alpha_2}} 
\big).
\end{equation}
By definition, there exists $\Gamma_6>0$ such that 
\begin{equation}
\label{eq:y2_2}
\|x^{\infty}\| \leq \Gamma_6 \big(
 \|y_{\rm ref}\|_{\mathbb{C}^p} + \|v\|_{\mathbb{C}^p} 
\big).
\end{equation}
In terms of $y_3$, we obtain
\begin{equation}
\label{eq:y3}
\|y_3\|_{L^2_{\alpha_1}} \leq \|C^+ \Pi \|_{\mathcal{L}(X,\mathbb{C}^p)}
\cdot \|x- x^{\infty}\mathds{1}_{\mathbb{R}_+} \|_{L^2_{\alpha_1}}.
\end{equation}
Combining \eqref{eq:y1}--\eqref{eq:y3} with \eqref{eq:y_y_inf},
we have that there exists $\Gamma_7 >0$ and $\alpha_3:= \max\{\alpha_1,\alpha_2\}<0$ such that 
\begin{equation}
\label{eq:y_CPi}
\|y-y^{\infty} \mathds{1}_{\mathbb{R}_+}\|_{L^2_{\alpha_3}} 
\leq \Gamma_7 	\big(
\|x_{\rm e}^0\|_{X\times X_{\rm d}} +\|v\|_{\mathbb{C}^p} + \|y_{\rm ref}\|_{\mathbb{C}^p}
\big),
\end{equation}
 which yields 
	\[
	\int^\tau_0 
	\left\|
	y(k\tau + t) - 
	y^{\infty} 
	\right\|^2_{\mathbb{C}^p} dt \to 0\qquad (k \to \infty).
	\]
	Since $\int^\tau_0 w(t) dt = 1$, it follows that 
	\begin{align*}
		\left\|(\mathcal{S}_\tau y) (k) -   
		y^{\infty} 
		\right\|_{\mathbb{C}^p}  
		& \leq 
		\int^\tau_0 
		\|
		w(t) 
		(
		y(k\tau + t) - 
		y^{\infty} 
		) 
		\|_{\mathbb{C}^p} dt \\
		&\leq 
		\sqrt{\int^\tau_0 
		|w(t)|^2 dt} \cdot 
		\sqrt{\int^\tau_0	
		\|
		y(k\tau + t) - 
		y^{\infty} 
		\|^2_{\mathbb{C}^p} dt} \to 0
	\end{align*}
	as $k \to \infty$.
	Therefore, the sampled output $\mathcal{S}_\tau y$
	converges to $y^{\infty}$.
	
	On the other hand, 
	the tracking property and the robustness property with respect to
	exosystems of the discretized system 
	implies that
	for every $y_{\rm ref},v \in \mathbb{C}^p$, 
	$(S_{\tau}y)(k) \to y_{\rm ref}$ as $k \to \infty$.
	This means that $y^{\infty} = y_{\rm ref}$.
	Thus, the tracking property is obtained from \eqref{eq:y_CPi}.
	
	Finally, we prove the robustness property.
	Let $(P,Q,R)$ be the realization of the controller \eqref{eq:controller} and
	$(\widetilde{A},\widetilde{B},\widetilde{C},\widetilde{\bf G})$ be the perturbed system node
	in $\mathcal{O}_{\rm s}(P,Q,R)$.
	Define the operator $\widetilde{A}_{\rm e}$ as in \eqref{eq:Discretized_CLS_A} by using
	$(\widetilde{A},\widetilde{B},\widetilde{C},\widetilde{\bf G})$.
	By assumption, the perturbed sampled-data system is exponentially stable.
	Using Lemma~\ref{lem:PS_ES}, we find that $\widetilde{A}_{\rm e}$ is power stable.
	Hence Theorem~\ref{thm:internal_model} shows that 
	for every $y_{\rm ref},v \in \mathbb{C}^p$, 
	the sampled output $S_{\tau}y$ of the perturbed plant satisfies
	$\lim_{k \to \infty}(S_{\tau}y)(k) = y_{\rm ref}$.
	In the argument to obtain \eqref{eq:y_CPi}, we used only
	the well-posedness of the system node 
	$(A,B,C,\bf G)$, the power stability of $A_{\rm e}$, 
	the assumptions $\langle$b\ref{enu_uns_finite}$\rangle$, $\langle$b\ref{enu_A_minus_EXS_SD}$\rangle$, 
	and
	$\langle$b\ref{enu_spec_IA}$\rangle$.
	The perturbed system node $(\widetilde{A},\widetilde{B},\widetilde{C},\widetilde{\bf G})$ is 
	well posed by assumption.
	The power stability of $\widetilde{A}_{\rm e}$ has  been already proved.
	By Proposition~3 and Theorem~9 in \cite{Logemann2013},
	the assumptions $\langle$b\ref{enu_uns_finite}$\rangle$, $\langle$b\ref{enu_A_minus_EXS_SD}$\rangle$, 
	and
	$\langle$b\ref{enu_spec_IA}$\rangle$
	hold for $(\widetilde{A},\widetilde{B},\widetilde{C},\widetilde{\bf G})$.
	Hence, repeating the argument as above, we obtain the tracking property of 
	the perturbed sampled-data system.
%
	\qed
\end{proof}

\begin{remark}
	\label{rem:convergence}
		As seen in the proof of Theorem~\ref{thm:servo_tracking},
		the states $x(t)$ and $x_{\rm d}(k)$ exponentially converge to $x^\infty$ and $x_{\rm d}^\infty$, respectively,
		where
		\[
		\begin{bmatrix}
		x^\infty \\ x_{\rm d}^\infty
		\end{bmatrix}
		=
		\begin{bmatrix}
		I-A_{\tau} & ~~B_{\tau}R \\
		-QC_{\tau} & ~~I - P + QD_{\tau}R
		\end{bmatrix}^{-1}
		\begin{bmatrix}
		B_\tau v \\
		Q (y_{\rm ref} - D_\tau v)
		\end{bmatrix}.
		\]
\end{remark}

Since the output $y$ may not be continuous,
Theorem~\ref{thm:servo_tracking} does not guarantee that 
$y(t) \to y_{\rm ref}$ as $t \to \infty$.
To address this issue, we use a smoothing stable precompensator $\Sigma_{\rm p}$ of the form
\begin{equation}
	\label{eq:pre_compensater}
	\dot x_{\rm p} = -ax_{\rm p} + u_{\rm p},\quad
	x_{\rm p}(0) = x_{\rm p}^0 \in \mathbb{C}^p,
\end{equation}
where $a > 0$.
Consider the sampled-data system consisting of the digital controller \eqref{eq:controller},
the well-posed plant \eqref{eq:well_posed}, the precompensator \eqref{eq:pre_compensater}, and the feedback law
\[
u = x_{\rm p},\qquad u_{\rm p} = \mathcal{H}_{\tau} y_{\rm d} + v\mathds{1}_{\mathbb{R}_+},\qquad 
u_{\rm d} = y_{\rm ref}\mathds{1}_{\mathbb{Z}_+}  - \mathcal{S}_\tau y.
\]
Fig.~\ref{fig:sampled_data_sys_pre} illustrates the sampled-data system 
with a precompensator.
\begin{figure}[tb]
	\centering
	\includegraphics[width = 6cm]{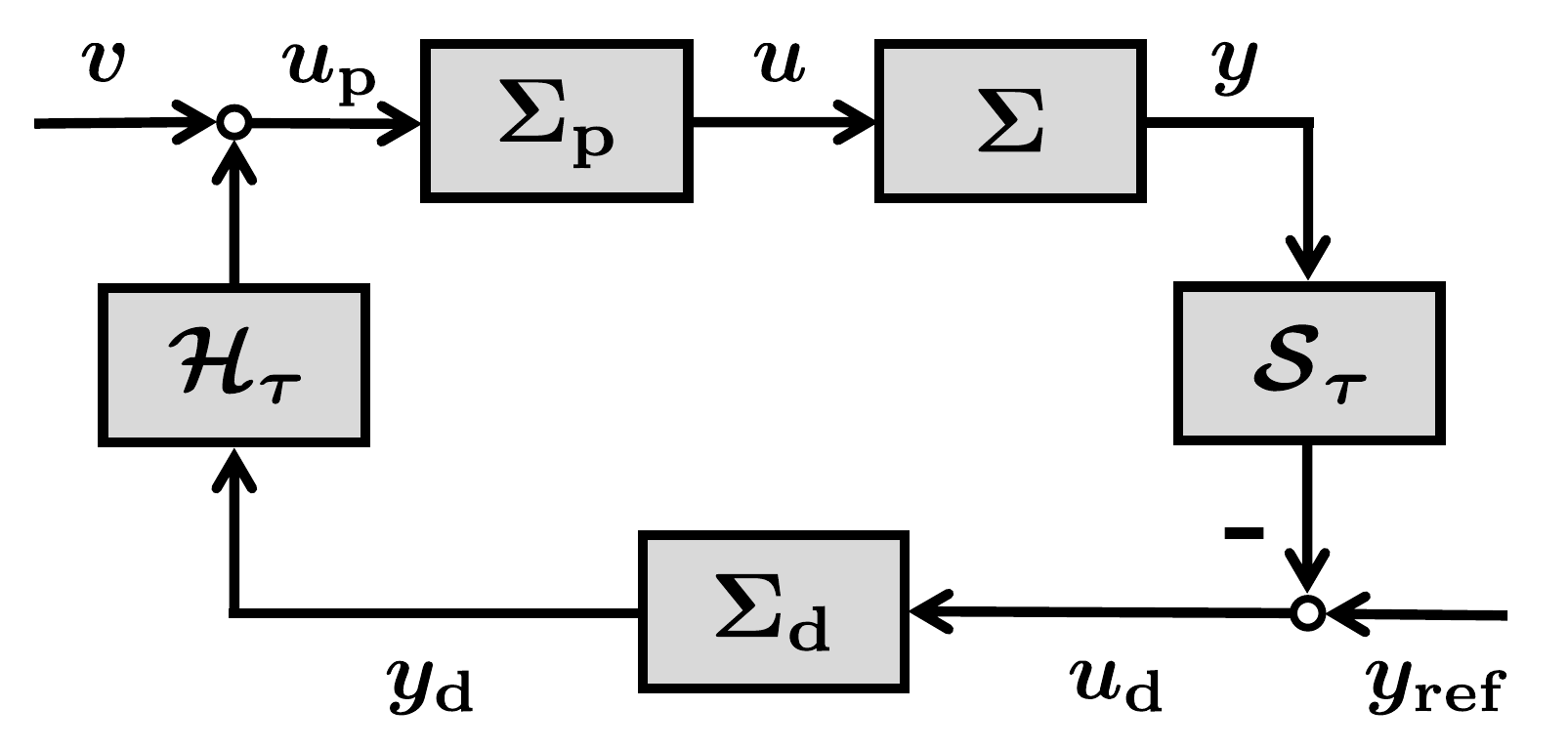}
	\caption{Sampled-data system with precompensator.}
	\label{fig:sampled_data_sys_pre}
\end{figure}

The new plant $\widehat \Sigma$, which is
the interconnection of the plant $\Sigma$ and the precompensator $\Sigma_{\rm p}$, is a well-posed system 
with state space $\widehat X := X \times \mathbb{C}^p$, input space $\mathbb{C}^p$,
and output space $\mathbb{C}^p$.
The generating operators $(\widehat A,\widehat B,\widehat C)$ of $\widehat \Sigma$ are given by
\begin{align*}
	\widehat A &:= 
	\begin{bmatrix}
		A & B \\ 0 & -aI
	\end{bmatrix}
	\text{~with~}
	\dom (\widehat A) := 
	\left\{
	\begin{bmatrix}
		x \\ x_{\rm p}
	\end{bmatrix} \in X \times \mathbb{C}^p
	: Ax + Bx_{\rm p} \in X
	\right\}\\
	\widehat B &:= 
	\begin{bmatrix}
	0 \\ I
	\end{bmatrix},\quad 
	\widehat C 
	\begin{bmatrix}
		x \\ x_{\rm p}
	\end{bmatrix} :=
	C(x - (\lambda I - A)^{-1} Bx_{\rm p} ) + \mathbf{G}(\lambda) x_{\rm p}
	~~ \forall 	\begin{bmatrix}
		x \\ x_{\rm p}
	\end{bmatrix} \in \dom (\widehat  A),
\end{align*}
where $\lambda \in \varrho(A)$.
The
transfer function $\widehat{\bf G}$ of $\widehat \Sigma$ is 
$\widehat{\bf G}(s) := {\bf G}(s) /(s+a)$.
\begin{theorem}
	\label{thm:servo_tracking_y}
	If the assumptions {\rm $\langle$b\ref{enu_Resol_SD}$\rangle$--$\langle$b\ref{enu_Simple_SD}$\rangle$} hold, then
	there exists a finite-dimensional controller \eqref{eq:controller}  that 
	is a solution of Problem~\ref{prob:OR_SD} in the context of the interconnected plant $\widehat \Sigma$.
	Furthermore, if a controller in the form \eqref{eq:controller} satisfies
	the stability property and the tracking property in 
	Problem~\ref{prob:OR_SD} for  the interconnected plant $\widehat \Sigma$,
	then the following convergence property holds:
	Let  $x_{\rm d}^0 \in X_{\rm d}$ and 
	$v, y_{\rm ref} \in \mathbb{C}^p$ be arbitrary and let the initial states
	$x^0 \in X$ and $x_{\rm p}^0 \in \mathbb{C}^p$ be such that 
	$\mathbf{T}_{t_0}(Ax^0 + Bx_{\rm p}^0) \in X$ for some $t_0 \geq0$. Then
	there exist
a function $y_{\rm c}: \mathbb{R}_+ \to \mathbb{C}^p$ and
a constant $\alpha < 0$
such that 
\begin{enumerate}
	\item
	$y_{\rm c}$ coincides with the output $y$ of $\widehat \Sigma$ for a.e. $t \geq 0$,
	is continuous on $[t_0,\infty)$, and satisfies
	\begin{equation*}
	\lim_{t \to \infty} (y_{\rm c}(t) - y_{\rm ref} ) e^{-\alpha t} = 0;
	\end{equation*}
	\item 
	$\alpha$ is independent of 
	$x^0 \in X, x_{\rm p}^0 \in \mathbb{C}^p$, $x_{\rm d}^0 \in X_{\rm d}$, and
	$y_{\rm ref},v \in \mathbb{C}^p$.
\end{enumerate}
\end{theorem}
\begin{proof}
	Due to Theorem~\ref{thm:servo_tracking},
	the first assertion follows if 
	the assumptions {\rm $\langle$b\ref{enu_Resol_SD}$\rangle$--$\langle$b\ref{enu_Simple_SD}$\rangle$} 
	are satisfied in the context of  the interconnected plant $\widehat \Sigma$.
	Among these assumptions,  {\rm $\langle$b\ref{enu_uns_finite}$\rangle$--$\langle$b\ref{enu_weight}$\rangle$} 
	hold in the context of $\widehat \Sigma$  by Proposition~5 and the proof of
	Theorem~11 in \cite{Logemann2013}.
	By the definition of $\widehat A$ and 
	$\widehat{\mathbf{G}}$,
	the remaining assumptions 
	{\rm $\langle$b\ref{enu_Resol_SD}$\rangle$}, {\rm $\langle$b\ref{enu_Zero_SD}$\rangle$}, {\rm $\langle$b\ref{enu_two_eig_relation}$\rangle$},
	and {\rm $\langle$b\ref{enu_Simple_SD}$\rangle$} hold in the context of $\widehat \Sigma$.
	
	We prove the second assertion.
	Define the operator $\widehat A_{\rm e}$ as in
	\eqref{eq:Discretized_CLS_A} by using the interconnected plant $\widehat \Sigma$.
	By Lemma~\ref{lem:PS_ES},
	the stability property implies the power stability of $\widehat A_{\rm e}$.
	By Proposition~3 and Theorems~9 and 10 in \cite{Logemann2013},
	the assumptions $\langle$b\ref{enu_uns_finite}$\rangle$, $\langle$b\ref{enu_A_minus_EXS_SD}$\rangle$, 
	and
	$\langle$b\ref{enu_spec_IA}$\rangle$
	hold in the context of both $\Sigma$ and $\widehat \Sigma$. 
	
	Let  $x_{\rm d}^0 \in X_{\rm d}$ and 
	$v, y_{\rm ref} \in \mathbb{C}^p$ be given, and
	let $t_0 \geq0$, $x^0 \in X$, and $x_{\rm p}^0 \in \mathbb{C}^p$ be
	such that
	$\mathbf{T}_{t_0}(Ax^0 + Bx_{\rm p}^0) \in X$.
	It can be shown as
	in the proof of Theorem~\ref{thm:servo_tracking} that there exist 
	$x^{\infty} \in X$, $x^{\infty}_{\rm p} \in \mathbb{C}^p$, $x^{\infty}_{\rm d} \in X_{\rm d}$,
	$\widehat \Gamma > 0$, and $\widehat \rho \in (0,1)$ such that 
	\begin{align}
		\label{eq:state_conv_filter}
		&\|
		x(k\tau + t) - x^{\infty}
		\| +
		\|
		x_{\rm p}(k\tau + t) - x^{\infty}_{\rm p} 
		\|_{\mathbb{C}^p} + 
		\|
		x_{\rm d}(k) - x^{\infty}_{\rm d} 
		\|_{X_{\rm d}} 	\\
		&~ \leq \widehat\Gamma \widehat\rho^{\hspace{1.5pt}k}\big(
		\|x^0\| +\|x^0_{\rm p} \|_{\mathbb{C}^p} +\|x^0_{\rm d} \|_{X_{\rm d}}   + \|y_{\rm ref}\|_{\mathbb{C}^p}+\|v\|_{\mathbb{C}^p}
		\big)~~\forall t \in [0,\tau),~
		\forall k \in \mathbb{Z}_+.\notag 
	\end{align}
	Similarly to \eqref{eq:x_u_infty},
	we obtain
	\begin{equation}
		\label{eq:stady}
		\widehat A
		\begin{bmatrix}
			x^\infty \\ x^{\infty}_{\rm p}
		\end{bmatrix} + 
		\widehat B  (Rx_{\rm d}^{\infty} + v) = 0.
	\end{equation}
	Using the projection $\Pi$ on $X$ given in \eqref{eq:projection},
	we define 
	\[
	x^{\infty}_- := (I-\Pi)x^{\infty},\quad
	x^0_- := (I-\Pi)x^{0},\quad
	y^\infty := {\bf G}^{-}(0) x_{\rm p}^{\infty} +  C^+\Pi x^{\infty}.
	\]
	Lemma~\ref{lem:y_decomp} yields
	\[
	y(t) - y^{\infty}\mathds{1}_{\mathbb{R}_+}
	=
	y_1(t) + y_2(t) + y_3(t)\qquad
	{\rm a.e.~}t\geq 0,
	\]
	where
	\begin{align*}
		y_1 &:= (C^-)_{\Lambda} {\bf T}^-  x_-^{\infty} + 
		G^- (x_{\rm p}^{\infty}\mathds{1}_{\mathbb{R}_+}) -   {\bf G}^{-}(0) x_{\rm p}^{\infty}\mathds{1}_{\mathbb{R}_+}  \\
		y_2 &:= (C^-)_{\Lambda} {\bf T}^-(x_-^0 - x_-^{\infty}) 
		+
		G^- (x_{\rm p} - x_{\rm p}^{\infty}\mathds{1}_{\mathbb{R}_+}) 
		\\
		y_3 &:= 
		C^+\Pi (x - x^{\infty}\mathds{1}_{\mathbb{R}_+}).
	\end{align*}
	
	We can show that $y_1(t) = 0$ for a.e. $t\geq0$ in the same way
	as in the proof of Theorem~\ref{thm:servo_tracking}. In fact,
	using \eqref{eq:stady}, 
	we obtain
	\begin{equation}
		\label{eq:A-_B-}
		A^- x^{\infty}_- + B^- x_{\rm p}^{\infty} =	
		(I - \Pi_{-1}) (Ax^{\infty} + Bx_{\rm p}^{\infty}) = 0.
	\end{equation}
	By \eqref{eq:G_property},
	\begin{align*}
		\mathcal{L}(
		G^- (x_{\rm p}^{\infty}\mathds{1}_{\mathbb{R}_+}) -   {\bf G}^{-}(0) x_{\rm p}^{\infty} \mathds{1}_{\mathbb{R}_+}
		)(s)
		&=
		\frac{{\bf G}^{-}(s) - {\bf G}^{-}(0)}{s} x_{\rm p}^{\infty}\\
		&=
		C^{-}(sI-A^-)^{-1}(A^-)^{-1} B^-x_{\rm p}^{\infty}
	\end{align*}
	for every $s \in \mathbb{C}_0$.
	Hence
	the Laplace transform of $y_1$ is given by
	\[
	\mathfrak{L}(y_1)(s) 
	=
	C^{-}(sI-A^-)^{-1}(A^-)^{-1} (A^-x^\infty_- + B^-x_{\rm p}^{\infty}) = 0
	\qquad \forall s \in \mathbb{C}_0.
	\]
	Thus we obtain
	$
	y_1(t) = 0
	$
	for a.e.~$t \geq 0$.
	
	We next investigate continuity and convergence of $y_2$.
By Proposition~2.1 of \cite{Logemann2005},
if
\begin{equation}
\label{eq:error_inclusion}
{\bf T}_{t_0}^- \big(A^- (x_-^0 - x_-^{\infty})  + B^- (x^0_{\rm p} - x^\infty_{\rm p}  )\big) \in X^-
\end{equation}
and if
$x_{\rm p} - x_{\rm p}^{\infty}\mathds{1}_{\mathbb{R}_+} \in  L^{2}_{\beta_2}(\mathbb{R}_+,\mathbb{C}^p)$ with
$\dot x_{\rm p} \in L^{2}_{\beta_2}(\mathbb{R}_+,\mathbb{C}^p)$
for some $\beta_2 \in (\omega({\bf T}^-),0)$,
then there exists
a function $y_{2,\rm c} :\mathbb{R}_+ \to \mathbb{C}^p$  such that 
$y_{2,\rm c}$  coincides with $y_2$ for a.e.~$ t \geq 0 $, is continuous on $[t_0, \infty)$, 
and satisfies $\lim_{t\to \infty}y_{2,\rm c}(t) e^{-\beta_2 t}=0$.
	
	Since 
	${\bf T}_{t_0} (Ax^0 + Bx_{\rm p}^0) \in X$ by assumption, it follows that
	\[
	{\bf T}_{t_0}^- (A^- x_-^0 + B^- x^0_{\rm p}) =
	{\bf T}_{t_0}^- (I - \Pi_{-1}) (Ax^0 + Bx_{\rm p}^0) \in X^-.
	\]
	This together with \eqref{eq:A-_B-} yields
	\eqref{eq:error_inclusion}.
	
	Let us show that $x_{\rm p} - x_{\rm p}^{\infty}\mathds{1}_{\mathbb{R}_+} \in  
	L^{2}_{\beta_2}(\mathbb{R}_+,\mathbb{C}^p)$ and
$\dot x_{\rm p} \in L^{2}_{\beta_2}(\mathbb{R}_+,\mathbb{C}^p)$
for some $\beta_2 \in (\omega({\bf T}^-),0)$.
Recall that
\[
\dot x_{\rm p} = -ax_{\rm p} + \mathcal{H}_{\tau} Rx_{\rm d} + v\mathds{1}_{\mathbb{R}_+}.
\]
Since  \eqref{eq:stady} yields
\[
-ax^{\infty}_{\rm p} + Rx_{\rm d}^{\infty} + v = 0,
\]
it follows that 
\[
\dot x_{\rm p}  = -a(x_{\rm p} - x^{\infty}_{\rm p} \mathds{1}_{\mathbb{R}_+}) + \mathcal{H}_{\tau} R(x_{\rm d} - 
x_{\rm d}^{\infty} \mathds{1}_{\mathbb{Z}_+}).
\]
By
\eqref{eq:state_conv_filter},
there exists $\beta_2 \in (\omega({\bf T}^-),0)$ such that 
$x_{\rm p} - x_{\rm p}^{\infty}\mathds{1}_{\mathbb{R}_+} \in  L^{2}_{\beta_2}(\mathbb{R}_+,\mathbb{C}^p)$ and
$\dot x_{\rm p} \in L^{2}_{\beta_2}(\mathbb{R}_+,\mathbb{C}^p)$.
	
	Since $x$ is continuous, it follows that $y_3$ is also continuous.
	Invoking \eqref{eq:state_conv_filter}, we have that 
	$\lim _{t\to \infty} y_3(t) e^{-\beta_3 t} = 0$ for some $\beta_3 < 0$. Thus
	$y_{\rm c}:=y_{2,\rm c} + y_3 + y^{\infty}\mathds{1}_{\mathbb{R}_+}$ 
	coincides $y$ almost everywhere in $\mathbb{R}_+$,
	is continuous on $[t_0, \infty)$, and $\lim_{t \to \infty} (y_{\rm c}(t)- y^{\infty})e^{-\alpha t} = 0$ for $\alpha :=
	\max\{\beta_2,\beta_3\} < 0$. By construction,
	$\alpha$ is independent of 
	$x^0 \in X, x_{\rm p}^0 \in \mathbb{C}^p$, $x_{\rm d}^0 \in X_{\rm d}$, and
	$y_{\rm ref},v \in \mathbb{C}^p$.

	Finally, we prove that $y_{\rm ref} = y^{\infty}$.
	Since $\int^\tau_0 w(t) dt = 1$, it follows that, 
	for every $k \in \mathbb{Z}_+$ with $k\tau > t_0$,
	\begin{align*}
		\big\|(\mathcal{S}_\tau y)(k) - y^\infty \big\|_{\mathbb{C}^p}
		&\leq \int^\tau_0 \|w(t)(y_{\rm c}(k\tau+t) - y^\infty)\|_{\mathbb{C}^p} dt \\
		&\leq\sqrt{ \tau}  \|w\|_{L^2(0,\tau)} \max_{0\leq t \leq \tau} \|y_{\rm c}(k\tau+t) - y^\infty\|_{\mathbb{C}^p}.
	\end{align*}
	Therefore,
	$
	\lim_{k\to\infty}(\mathcal{S}_\tau y)(k) = y^\infty.
	$
	On the other hand, from the tracking property,
	it follows that $\| 
	y - y_{\rm ref}\mathds{1}_{\mathbb{R}_+}
	\|_{L^2} < \infty$. Hence
		\begin{align*}
	\left\|(\mathcal{S}_\tau y) (k) -   
	y_{\rm ref}
	\right\|_{\mathbb{C}^p}  
	& \leq 
	\int^\tau_0 
	\|
	w(t) 
	(
	y(k\tau + t) - 
	y_{\rm ref}
	) 
	\|_{\mathbb{C}^p} dt \\
	&\leq 
	\sqrt{\int^\tau_0 
		|w(t)|^2 dt} \cdot 
	\sqrt{\int^\tau_0	
		\|
		y(k\tau + t) - 
		y_{\rm ref}
		\|^2_{\mathbb{C}^p} dt} \to 0
	\end{align*}
	as $k \to \infty$.
	Thus,
	$y_{\rm ref} = y^{\infty}$.
	This completes the proof.
	\qed
\end{proof}



\section{Application to delay systems}
In this section, we study sampled-data output regulation for 
systems with state and output delays. This illustrates
Theorem~\ref{thm:servo_tracking} and the design procedure of
finite-dimensional regulating controllers in Section~\ref{sec:DTOR}.
For delay systems, the problem of output regulation has been investigated in
\cite{Fridman2003, Toledo2003, Yoon2016} and the reference therein.
Recently,
the solvability of the output regulation problem for delay systems with 
infinite-dimensional state spaces has been characterized by the associated regulator equations in 
\cite{Paunonen2017ACC}.
In the studies above, continuous-time output regulation is considered, whereas
we here study sampled-data output regulation for delay systems, focusing on
constant reference and disturbance signals.

First, the delay system we consider and its state-space representation are introduced. Next,
in Section~\ref{subsec:Decomp_delay}, we  decompose delay systems
into a finite-dimensional unstable part and an infinite-dimensional stable part, 
and then approximate the infinite-dimensional stable part by a finite-dimensional system
for the design of regulating controllers.
In Section~\ref{subsec:NS}, we finally present a numerical example to illustrate the proposed design method.
Throughout this section, we use the same
notation as in Section~\ref{sec:SDOR}.

For $q, \widehat q \in \mathbb{N}$,
let $h_q > h_{q-1} > \cdots > h_1 > 0$ and 
$h_q \geq \widehat h_{\widehat q}> \widehat h_{\widehat q-1}> \cdots > \widehat h_1 \geq 0$.
Consider the following delay system:
\begin{subequations}
	\label{eq:delay_system}
	\begin{align}
	\dot z(t) &=A_0z(t)  + \sum_{j=1}^q A_j z(t-h_j)+b u(t), \quad  t \geq 0\\
	y(t) &=  \sum_{\ell=1}^{\widehat q} c_\ell z(t-\widehat h_\ell),  \quad t \geq 0\\
	z(0) &= z^0,\qquad z(\theta) = \varpi (\theta),\quad \theta \in [-h_q,0],
	\end{align}
\end{subequations}
where $z(t) \in \mathbb{C}^n, u(t), y(t) \in \mathbb{C}$ are the state,
the input, and the output of the system, respectively, $A_j \in \mathbb{C}^{n \times n}$, $b \in \mathbb{C}^{n \times 1}$, 
$c_\ell \in \mathbb{C}^{1 \times n}$ for every $j\in\{0,\dots,q\}$ and for every $\ell \in \{1,\dots, \widehat q \}$, 
$ z^0 \in \mathbb{C}^n$,
and
$\varpi \in L^2\big([-h_q,0],\mathbb{C}^n \big)$.
In \eqref{eq:delay_system}, 
$h_1,\dots, h_q$ and 
$\widehat h_1,\dots, \widehat h_{\widehat q}$ represent the state delay and the output delay, respectively.
We assume that the input $u$ satisfies $u \in L^2_{\rm loc}(\mathbb{R}_+)$.

The state space of the delay system
\eqref{eq:delay_system} is given by
$X = \mathbb{C}^n \oplus L^2\big([-h_q,0],\mathbb{C}^n \big)$ with the standard inner product:
\[
\left(
\begin{bmatrix}
\zeta_1 \\ \varpi _1
\end{bmatrix},
\begin{bmatrix}
\zeta_2 \\ \varpi _2
\end{bmatrix}
\right)
:=
(\zeta_1,\zeta_2)_{\mathbb{C}^n} 
+ (\varpi _1,\varpi _2)_{L^2(-h_q,0)}.
\] 
The generating operators $(A,B,C)$ of
the delay system \eqref{eq:delay_system} are given by
\begin{align*}
A 
\begin{bmatrix}
\zeta \\ \varpi 
\end{bmatrix}
=
\begin{bmatrix}
A_0 \zeta + \sum_{j=1}^q A_j \varpi (-h_j) \\
\frac{d\varpi}{d\theta}
\end{bmatrix}
\end{align*}
with domain
\[
\dom (A) = 
\left\{
\begin{bmatrix}
\zeta \\ \varpi 
\end{bmatrix}
\in \mathbb{C}^n \oplus W^{1,2}\big([-h_q,0],\mathbb{C}^n \big) :
\varpi (0) = \zeta
\right\}
\]
and
\begin{align*}
Bs &= 
\begin{bmatrix}
bs \\ 0
\end{bmatrix}\qquad \forall s \in \mathbb{C} \\
C \begin{bmatrix}
\zeta \\ \varpi 
\end{bmatrix} &= 
\sum_{\ell=1}^{\widehat q}c_\ell
\varpi \big(\hspace{-2pt}-\widehat h_\ell \big)\qquad
\forall  \begin{bmatrix}
\zeta \\ \varpi 
\end{bmatrix} \in X_1.
\end{align*}
The transfer function of the delay system \eqref{eq:delay_system} is given by
\[
\mathbf{G}(s) = \sum_{\ell=1}^{\widehat q} e^{-\widehat h_\ell s}c_\ell \Delta(s)^{-1}b,
\quad \text{where~}
\Delta(s) :=sI - 
A_0 - \sum_{j=1}^q A_j e^{-h_j s}.
\]
The derivation of the generating operators and  the transfer function of delay systems  can be found, e.g.,
in Chapters 2--4 of \cite{Curtain1995} (for the case without output delays).
One can see from  Lemma 2.4.3 in \cite{Curtain1995}
that  $C$ is admissible. Hence,
Theorem~5.1 in \cite{Curtain1989} implies that 
the delay system \eqref{eq:delay_system} defines a well-posed system.
See, e.g, \cite{Hadd2005,Bounit2006} for the well-posedness of more general delay systems.

Let $\bf T$ be the strongly continuous semigroup generating $A$, and
define
\begin{equation}
\label{eq:well_posed_delay}
\begin{bmatrix}
x_1(t) \\ x_2(t)
\end{bmatrix} 
:=
x(t) 
:=
{\bf T}(t) x^0 + \int^t_0 {\bf T}(t-s)Bu(s)ds,\quad
x^0 :=
\begin{bmatrix}
z^0 \\ \varpi
\end{bmatrix}.
\end{equation}
It is shown
in Example~3.1.9 of \cite{Curtain1995} that 
$x_1(t) = z(t)$ and $x_2(t) = z(t+\cdot)$ hold for all $t \geq 0$, where
$z$ is the solution of \eqref{eq:delay_system}.
Furthermore, $z$ is absolutely continuous on $[0,\infty)$; see, e.g., 
Theorem~2.4.1 in \cite{Curtain1995}.
Hence $x(t) \in X_1$ for every $t \geq h_q$, and $y$ is (absolutely) continuous on
$[{\widehat h}_{\widehat q},\infty)$. 
For completeness, we show in Appendix~\ref{sec:lambda_extension} that
$y(t)= C_{\Lambda}x(t)$ for a.e. $t \geq 0$.


The output $y$ of this delay system exponentially converges to a constant
reference signal without a precompensator. In fact,
once we construct a controller that is a solution of Problem~\ref{prob:OR_SD},
$z(t)$ exponentially converges to some $z^{\infty} \in \mathbb{C}^n$;
see, e.g, Remark~\ref{rem:convergence}.
Since
$y$ is continuous on $[{\widehat h}_{\widehat q},\infty)$,
we have from the argument in the last paragraph of the proof of Theorem~\ref{thm:servo_tracking_y} 
that 
$y(t)$ also exponentially converges to 
$y_{\rm ref} = \sum_{\ell=1}^{\widehat q} c_{\ell} z^{\infty}$.

\subsection{Decomposition of delay systems into stable and unstable parts}
\label{subsec:Decomp_delay}
By Theorem~2.4.6 of  \cite{Curtain1995}, all elements of $\sigma(A)$ are the eigenvalues of $A$ with 
finite multiplicities, and
\[
\sigma(A) = \{s \in \mathbb{C}: {\rm det}\Delta(s)=  0 \}.
\]
For every $\varepsilon \in \mathbb{R}$, $\sigma(A) \cap \cl(\mathbb C_{-\varepsilon})$ consists of
finitely many isolated eigenvalues of $A$. Hence the assumption 
$\langle$b\ref{enu_uns_finite}$\rangle$ in Section~\ref{subsec:SDassumption} holds.
We place the following assumption on the eigenvalues of $A$ in $\cl(\mathbb C_0)$.

%

\begin{assumption}
	\label{assump:simple_zeros}
	The zeros, $\gamma_1,\dots,\gamma_N$, of $\det \Delta$ in $ \cl(\mathbb C_0)$ are simple.
\end{assumption}

Using Lemma~\ref{lem:det_zeros}, we find that $\dim \ker \Delta(\gamma_m) = 1$ for every $m\in \{1,\dots,N\}$ 
under Assumption~\ref{assump:simple_zeros}.
By
Theorem~2.4.6 and Corollary~2.4.7 of \cite{Curtain1995},
the order and the multiplicity of 
the eigenvalues $\gamma_1,\dots,\gamma_N$ of $A$ are both one.
For $m \in \{1,\dots,N\}$,
let nonzero vectors $\varsigma_m, \nu_m \in \mathbb{C}^n$ satisfy
$\Delta(\gamma_m)\varsigma_m = 0$ and $\Delta({\widebar \gamma}_m)^*\nu_m = 0$, respectively.
By 
Theorem~2.4.6 and Lemma 2.4.9 of \cite{Curtain1995},
the eigenvector $\phi_m$ of $A$ corresponding to the eigenvalue $\gamma_m$ and 
the eigenvector $\psi_m$ of $A^*$ corresponding to the  eigenvalue $\widebar{\gamma}_m$
are given by
\begin{equation}
\label{eq:eigen_vector}
\phi_m := 
\begin{bmatrix}
\varsigma_m \\ \varpi_m \varsigma_m
\end{bmatrix},\qquad
\psi_m := 
\frac{1}{
	{\widebar d}_m
}
\begin{bmatrix}
\nu_m \\ 
\sum_{j=1}^q e^{-\widebar{\gamma}_m h_j}  \frac{\mathds{1}_{[-h_j,0]}}{\varpi_m^*} A_j^*\nu_m
\end{bmatrix},
\end{equation}
where $\varpi_m(\theta) := e^{\gamma_m \theta}$, 
$\varpi_m^*(\theta) := e^{\widebar{\gamma}_m \theta}$ for every $\theta \in [-h_q,0]$ and
\[
d_m :=
(\varsigma_m, \nu_m)_{\mathbb{C}^n} + 
\sum_{j=1}^q
h_j e^{-\gamma_m h_j} (A_j\varsigma_m,\nu_m)_{\mathbb{C}^n} .
\]
By definition,
$\phi_m$ and $\psi_m$ satisfy $(\phi_m,\psi_m) = 1$ for every $m \in \{1,\dots,N\}$.
In addition, since
\begin{align*}
\gamma_m(\phi_m, \psi_j)
&= 
(A\phi_m, \psi_j)=
(\phi_m, A^*\psi_j) =
\gamma_j (\phi_m, \psi_j)\qquad
\forall m,j \in \{1,\dots,N \},
\end{align*}
it follows that $(\phi_m, \psi_j) = 0$ if 
$m \not=j$.

Let 
$\Phi$ be a  rectifiable, closed, 
simple curve $\Phi$ in $\mathbb{C}$ 
enclosing an open set that contains $\sigma(A) \cap  \cl(\mathbb{C}_{0})$ in its interior and
$\sigma(A) \cap  \big( \mathbb{C} \setminus \cl(\mathbb{C}_{0}) \big)$
in its exterior.
The spectral projection $\Pi$ corresponding to 
$\sigma(A) \cap  \big( \mathbb{C} \setminus \cl(\mathbb{C}_{0}) \big)$ is 
defined by
\[
\Pi x := \frac{1}{2\pi i}
\int_\Phi (sI-A)^{-1}x ds,
\]
which, by Lemma 2.5.7 of \cite{Curtain1995}, satisfies
\begin{equation}
\label{eq:Pi_delay}
\Pi x  = \sum_{m=1}^N (x,\psi_m) \phi_m \qquad \forall x \in X.
\end{equation}
Hence
\begin{align*}
X^+ &:= \Pi X = 
\left\{
\sum_{m=1}^N s_m \phi_m: s_m\in \mathbb{C}\quad \forall m\in \{1,\dots,N\}
\right\} 
\end{align*}
and for  $s, s_1,\dots,s_N \in \mathbb{C}$, the 
operators $A^+$, $B^+$, $C^+$, and $\mathbf{T}^+$ defined as
in Section~\ref{subsec:SDassumption} satisfy
\begin{align*}
A^+ 
\left(
\sum_{m=1}^Ns_m \phi_m
\right) &=
\sum_{m=1}^Ns_m \gamma_m \phi_m\\
B^+s &= s
\sum_{m=1}^N \frac{(b,\nu_m)_{\mathbb{C}^n}}{d_m} \phi_m \\
C^+
\left(
\sum_{m=1}^Ns_m \phi_m
\right) &= 
\sum_{m=1}^N s_m
\sum_{\ell=1}^{\widehat q}e^{-\gamma_m \widehat h_\ell} c_\ell \varsigma_m\\
\mathbf{T}^+_t \left(
\sum_{m=1}^Ns_m \phi_m
\right) &= 
\sum_{m=1}^Ns_m e^{\gamma_m t} \phi_m \qquad \forall t \geq 0.
\end{align*}
We obtain
\begin{align*}
\mathbf{G}^+(s) := 
C^+(sI-A^+)^{-1}B^+ =  
\sum_{m=1}^N\frac{\kappa_m}{s-\gamma_m},
\end{align*}
where
\[
\kappa_m := \frac{ (b,\nu_m)_{\mathbb{C}^n}}{d_m}
\sum_{\ell=1}^{\widehat q}e^{-\gamma_m \widehat h_\ell}c_\ell \varsigma_m.
\]
Furthermore, as shown in b. of the proof of Theorem~5.2.12 of \cite{Curtain1995},
$\mathbf T^-$ is exponentially stable. Therefore 
the assumption $\langle$b\ref{enu_A_minus_EXS_SD}$\rangle$ in Section~\ref{subsec:SDassumption} is satisfied.

Since for every $s, s_1,\dots,s_N \in \mathbb{C}$,
the operators $(A_\tau^+,B_\tau^+,C_\tau^+)$ defined as in Section~\ref{subsec:properties_DS}
satisfy
\begin{align*}
A_\tau^+
\left(
\sum_{m=1}^Ns_m \phi_m
\right) &= \sum_{m=1}^Ns_m e^{\gamma_m \tau}\phi_m\\
B_\tau^+ s &= s \sum_{m=1}^N
\frac{(b,\nu_m)_{\mathbb{C}^n}}{d_m}
\frac{e^{\gamma_m \tau} - 1}{\gamma_m}  \phi_m \\
C_\tau^+\left(
\sum_{m=1}^Ns_m \phi_m
\right)&=
\sum_{m=1}^N s_m \int^\tau_0 w(t) e^{\gamma_m t} dt
\sum_{\ell=1}^{\widehat q}e^{-\gamma_m \widehat h_\ell} c_\ell \varsigma_m,
\end{align*}
it follows that 
\[
\mathbf{G}^+_\tau (z) :=
C_\tau^+ (zI-A_\tau^+)^{-1}B_\tau^+ = 
\sum_{m=1}^N
\frac{\alpha_m}{z-e^{\gamma_m \tau}},
\]
where
\[
\alpha_m := \kappa_m \frac{e^{\gamma_m \tau} - 1}{\gamma_m}\int^\tau_0 w(t) e^{\gamma_m t} dt.
\]

As in Example on pp.~1221--1223 of \cite{Logemann2013},
one can construct the approximation $\mathbf{R}$ of $\mathbf{G}_\tau^- = \mathbf{G}_\tau - \mathbf{G}^+_\tau$ as follows.
Define the input-output map $G^+: L^2_{\rm loc}(\mathbb{R}_+)\to L^2_{\rm loc}(\mathbb{R}_+)$ by
\[
(G^+u)(t) := \sum_{m=1}^N  \kappa_m \int^t_0 e^{\gamma_m (t-s)} u(s)ds,
\]
whose transfer function is given by $\mathbf{G}^+$.
Similarly, we denote by $G_\tau^+: F(\mathbb{Z}_+) \to F(\mathbb{Z}_+)$ the discrete-time input-output operator 
associated with the 
transfer function $\mathbf{G}_\tau^+$:
\[
(G_\tau^+ f)(k) := \sum_{m=1}^N\alpha_m \sum_{\ell=0}^{k-1} e^{(k - \ell - 1)\gamma_m \tau} f(\ell).
\]
A routine calculation shows that 
\begin{equation*}
\mathcal{S}_\tau G^+ \mathcal{H}_\tau = 
G_\tau^+
+
\sum_{m=1}^N 
\frac{\kappa_m(\beta_m - 1)}{\gamma_m}I,
\quad
\text{where~}
\beta_m := \int^\tau_0 w(t) e^{\gamma_m t}dt.
\end{equation*}
This yields
\[
\mathcal{S}_\tau G \mathcal{H}_\tau  = 
G_\tau^+ + \sum_{m=1}^N 
\frac{\kappa_m(\beta_m - 1)}{\gamma_m}I+ \mathcal{S}_\tau G^- \mathcal{H}_\tau .
\]
Note that $\mathcal{S}_\tau G \mathcal{H}_\tau$ is the discrete-time input-output operator associated with the 
transfer function
$\mathbf{G}_\tau$.
Then we obtain
\begin{align*}
\mathbf{G}^-_\tau = \mathbf{G}_\tau -   \mathbf{G}^+_\tau 
=   \sum_{m=1}^N 
\frac{\kappa_m(\beta_m - 1)}{\gamma_m} + \mathbf{H}_\tau,
\end{align*}
where $\mathbf{H}_\tau$ is the transfer function of the discrete-time input-output operator $\mathcal{S}_\tau G^- \mathcal{H}_\tau$.
Choose a rational function $\mathbf{R} \in H^{\infty}(\mathbb{E}_1)$
as a constant function 
\[
\mathbf{R}(z) \equiv \sum_{m=1}^N 
\frac{\kappa_m(\beta_m - 1)}{\gamma_m} .
\] 
A simple calculation gives
\[
\|\mathcal{H}_\tau\|_{\mathcal{L}(l^2(\mathbb{Z}_+), L^2(\mathbb{R}_+))} = \sqrt{\tau},\qquad
\|
\mathcal{S}_\tau
\|_{\mathcal{L}(L^2(\mathbb{R}_+), l^2(\mathbb{Z}_+))} = \|w\|_{L^2(0,\tau)}.
\]
Noting that 
the transfer function $\mathbf{G}^- = \mathbf{G} - \mathbf{G}^+$ of an exponentially stable 
well-posed system satisfies $\mathbf{G}^- \in H^{\infty}(\mathbb{C}_0)$,
we obtain
\[
\|\mathbf{G}^-_\tau - \mathbf{R}
\|_{H^{\infty}(\mathbb{E}_1)} =
\|\mathbf H_\tau \|_{H^{\infty}(\mathbb{E}_1)} 
\leq \sqrt{\tau} \|w\|_{L^2(0,\tau)} \cdot \|\mathbf{G}^-\|_{H^{\infty}(\mathbb{C}_0)}.
\]
Thus, if 
\begin{equation}
\label{eq:G+_cond_delay}
\|\mathbf{G}^-\|_{H^{\infty}(\mathbb{C}_0)} = \|\mathbf{G} - \mathbf{G}^+\|_{H^{\infty}(\mathbb{C}_0)} < 
\frac
{1}{\sqrt{\tau}M_1 \|w\|_{L^2(0,\tau)} \cdot  \| \mathbf{D}_+  \|_{H^{\infty}(\mathbb{E}_1)}},
\end{equation}
then we can design a regulating controller,
where $M_1 >0$ is defined as in \eqref{eq:M1_def}
and $\mathbf{N}_+/\mathbf{D}_+$  is a coprime factorization of $\mathbf{G}_\tau^+$
over the set of rational functions in $\in H^{\infty}(\mathbb{E}_1)$.

\subsection{Numerical simulation}
\label{subsec:NS}
In what follows, we consider the case 
$q=\widehat q = 1$
$A_0 =A_1= 0.2$, $b=1$, $c_1 = 1$, $h_1 = 1$, $\widehat h_1 = 0.1$, $\tau = 2$, $w(t) \equiv 1/2$.
We first show that 
\[g(s) := s-A_0-A_1e^{-h_1s} = s-0.2-0.2e^{-s}
\] 
has only one zero in $\cl(\mathbb{C}_0)$
in a way similar to Example~5.2.13 of \cite{Curtain1995}.
Define $g_1(s):= s-1$ and $g_2(s) := 0.8 - 0.2e^{-s}$.
For every $s \in \mathbb{C}_0$ satisfying $|s| >2$, 
we obtain $|g_1(s)| \geq |s| - 1 > 1$ and $|g_2(s)| \leq 0.8 + 0.2|e^{-s}| \leq 1$.
Therefore, $|g_1(s)| > |g_2(s)|$ for every $s \in \mathbb{C}_0$ with $|s| > 2$.
On the other hand,
for every $\omega \in \mathbb{R}$, 
$|g_1(i\omega)|^2 = 1+\omega^2$ and $|g_2(i\omega)|^2 = 0.68 - 0.32 \cos \omega$. Hence
\begin{equation}
\label{eq:g12_imaginary}
|g_1(i\omega)| > |g_2(i\omega)|\qquad \forall \omega \in \mathbb{R}.
\end{equation}
Rouche's theorem shows that 
$g_1$ and $g = g_1 + g_2$ have the same number of zeros in $\mathbb{C}_0$, where
each zero is counted as many times as its multiplicity.
Thus, $g$ has only one simple zero in $\mathbb{C}_0$.
Moreover, 
\eqref{eq:g12_imaginary} yields
\[
|g(i\omega)| \geq |g_1(i\omega)| - |g_2(i\omega)| > 0\qquad \omega \in \mathbb{R},
\]
and hence $g$ has no zeros on the imaginary axis. 
Since
$g(s)$
is negative at $s = 0$ and
positive at $s = +\infty$, it follows that 
the zero of $g$ in $\cl(\mathbb{C}_0)$ is real. 
Thus, the generator $A$ has only an eigenvalue at $s=\gamma\approx 0.3421$ in $\cl(\mathbb{C}_0)$, and
the assumption $\langle$b\ref{enu_Resol_SD}$\rangle$ in Section~\ref{subsec:SDassumption} 
is satisfied.

The transfer function of the delay system \eqref{eq:delay_system} is given by
\[
\mathbf{G}(s) = \frac{e^{-\widehat h_1s}}{s-A_0-A_1e^{-h_1s}}.
\]
Since 
\[
\mathbf{G}(0) = \frac{-1}{A_0+A_1}= -\frac{5}{2} \not = 0,
\] 
it follows that the assumption $\langle$b\ref{enu_Zero_SD}$\rangle$ in Section~\ref{subsec:SDassumption} holds. 

By \eqref{eq:eigen_vector},
the eigenvectors $\phi$ of $A$ and $\psi$ of $A^*$ 
corresponding to the eigenvalue $\gamma$
are given by
\[
\phi= 
\begin{bmatrix}
1 \\ \varpi_\gamma
\end{bmatrix},\qquad
\psi = 
\frac{1}{d}
\begin{bmatrix}
1 \\ A_1e^{-\gamma h_1} /\varpi_\gamma
\end{bmatrix},
\]
where $\varpi_\gamma(\theta) := e^{\gamma \theta}$ for every $\theta \in [-h_1,0]$ and
$d := 1+A_1h_1e^{-\gamma h_1}$.
Then $\phi$ and $\psi$ satisfy $(\phi,\psi) = 1$.
It follows from \eqref{eq:Pi_delay} that the projection $\Pi$ is given by
\[
\Pi x = (x,\psi) \phi\qquad \forall x \in X.
\]
Hence, $X^+ = \Pi X = 
\{
s \phi: s\in \mathbb{C}
\}$. For $s \in \mathbb{C}$,
\[
A^+(s \phi) = s\gamma \phi,\quad
B^+s= \frac{s}{d} \phi ,\quad
C^+(s \phi) = se^{-\gamma \widehat h_1}
\]
and 
\[
\mathbf{T}^+_t(s \phi) = se^{\gamma t} \phi\qquad \forall t\geq 0.
\]
In the previous subsection, we have showed that the assumptions
$\langle$b\ref{enu_uns_finite}$\rangle$ and $\langle$b\ref{enu_A_minus_EXS_SD}$\rangle$ in Section~\ref{subsec:SDassumption} hold.
The assumptions $\langle$b\ref{enu_CD_SD}$\rangle$--$\langle$b\ref{enu_Simple_SD}$\rangle$
are clearly satisfied.
The transfer function of the unstable part of $\mathbf{G}$ is given by
\begin{align*}
\mathbf{G}^+(s)  = \frac{\kappa}{s-\gamma}	\qquad
\text{where~~}\kappa := \frac{e^{-\gamma \widehat h_1}}{d}.
\end{align*}
Similarly, the transfer function of the unstable part of $\mathbf{G}_\tau$ is
\begin{align*}	
\mathbf{G}^+_\tau (z) =
\frac{\alpha}{z-e^{\gamma \tau}},
\qquad
\text{where~~}
\alpha := 
\kappa
\frac{e^{\gamma \tau} - 1}{\gamma}
\int^\tau_0 w(t) e^{\gamma t} dt.
\end{align*}

Define 
\[
\mathbf{N}_+(z) := \frac{\alpha}{z-a},\quad
\mathbf{D}_+(z) := \frac{z-e^{\gamma \tau}}{z-a},\qquad
\text{where $a := 0.9$.}
\]
Then $\mathbf{N}_+/ \mathbf{D}_+$ is a coprime factorization
of $\mathbf{G}^+_\tau$ over the set of rational functions in $H^{\infty}(\mathbb{E}_1)$.
Using Lemma~\ref{lem:G_Gtau}, we obtain
\[
\delta^* := 
\left|\frac{1}{\mathbf{D}_+(1) \mathbf{G}_\tau(1)}\right| = 
\left|
\frac{1}{
	\mathbf{D}_+(1) \mathbf{G}(0)} 
\right| = 
\frac{(1-a)(A_0+A_1)}{e^{\gamma \tau}-1}.
\]
There exists a rational function $\mathbf{Y}_+ \in H^{\infty}(\mathbb{E}_1)$ satisfying
interpolation conditions \eqref{eq:first_pole}
and the norm condition 	$\|\mathbf{Y}_+\|_{H^{\infty}(\mathbb{E}_1)} < 1 =: M$.

Choose a rational function $\mathbf{R} \in H^{\infty}(\mathbb{E}_1)$
as a constant function 
\[
\mathbf{R}(z) \equiv \frac{\kappa (\beta -1)}{\gamma},
\quad
\text{where~}
\beta := \int^\tau_0 w(t) e^{\gamma t}dt,
\]
and let us next show that \eqref{eq:G+_cond_delay} is satisfied.
A numerical computation shows that 
\[
\|\mathbf{G} -\mathbf{G}^+ \|_{H^{\infty}(\mathbb{C}_0)} = \sup_{\omega \in \mathbb{R}} |
\mathbf{G}(i\omega) -\mathbf{G}^+ (i\omega)
| < 0.1.
\]
Define 
\[
M_1 := \max 
\left\{
2\delta^*, M
\right\} = M = 1.
\]
Since  
$ \|w\|_{L^2(0,\tau)} =1/\sqrt{\tau}$ for the case $w(t) \equiv 1/\tau$,
it follows that 
\[
\frac
{1}{\sqrt{\tau}M_1 \|w\|_{L^2(0,\tau)} \|  \mathbf{D}_+  \|_{H^{\infty}(\mathbb{E}_1)}}
= \frac{1}{M_1\|\mathbf{D}_+ 
	\|_{H^{\infty}(\mathbb{E}_1)} } \approx 0.1018,
\]
and hence \eqref{eq:G+_cond_delay} is satisfied.

Define
a rational function $\mathbf{Y}_+ \in H^{\infty}(\mathbb{E}_1)$ by
\[
\mathbf{Y}_+(z) := 
\frac{0.7712 z - 0.7602}{z^2 - 0.7328 z},
\]
which satisfies
the interpolation conditions \eqref{eq:first_pole}, \eqref{eq:Y_cond} and 
the norm condition $\|\mathbf{Y}_+\|_{H^{\infty}(\mathbb{E}_1)} < M_1$.
By the construction used in the proof of Theorem~\ref{thm:existence_servo_cont},
a minimal realization of the digital controller
\[
\mathbf{K}(z) :=
\frac{
	(\mathbf{Y}_+\mathbf{D}_+)(z)}{1 - (\mathbf{N}_++ \mathbf{D}_+\mathbf{R} )(z)\mathbf{Y}_+(z)}
=
\frac{
	0.7712 z - 0.7602}
{  z^2 - 0.4814 z - 0.5186}
\]
is a solution of Problem~\ref{prob:OR_SD}.

Figs.~\ref{fig:output} and~\ref{fig:input} illustrate the time responses 
of the output $y$ and the input $u$, respectively.
The
initial states of the plant and the controller are chosen as
$z^0 = 2$, $\varpi(\theta) \equiv 2$, and $x_{\rm d}^0 =\begin{bmatrix}
0  & ~~0
\end{bmatrix}^*$, respectively.
The reference and disturbance signals are given by
$y_{\rm ref} = 1$ and $v \in \{-1,0,1\}$, respectively.

\begin{figure}[tb]
	\centering
	\includegraphics[width = 7cm]{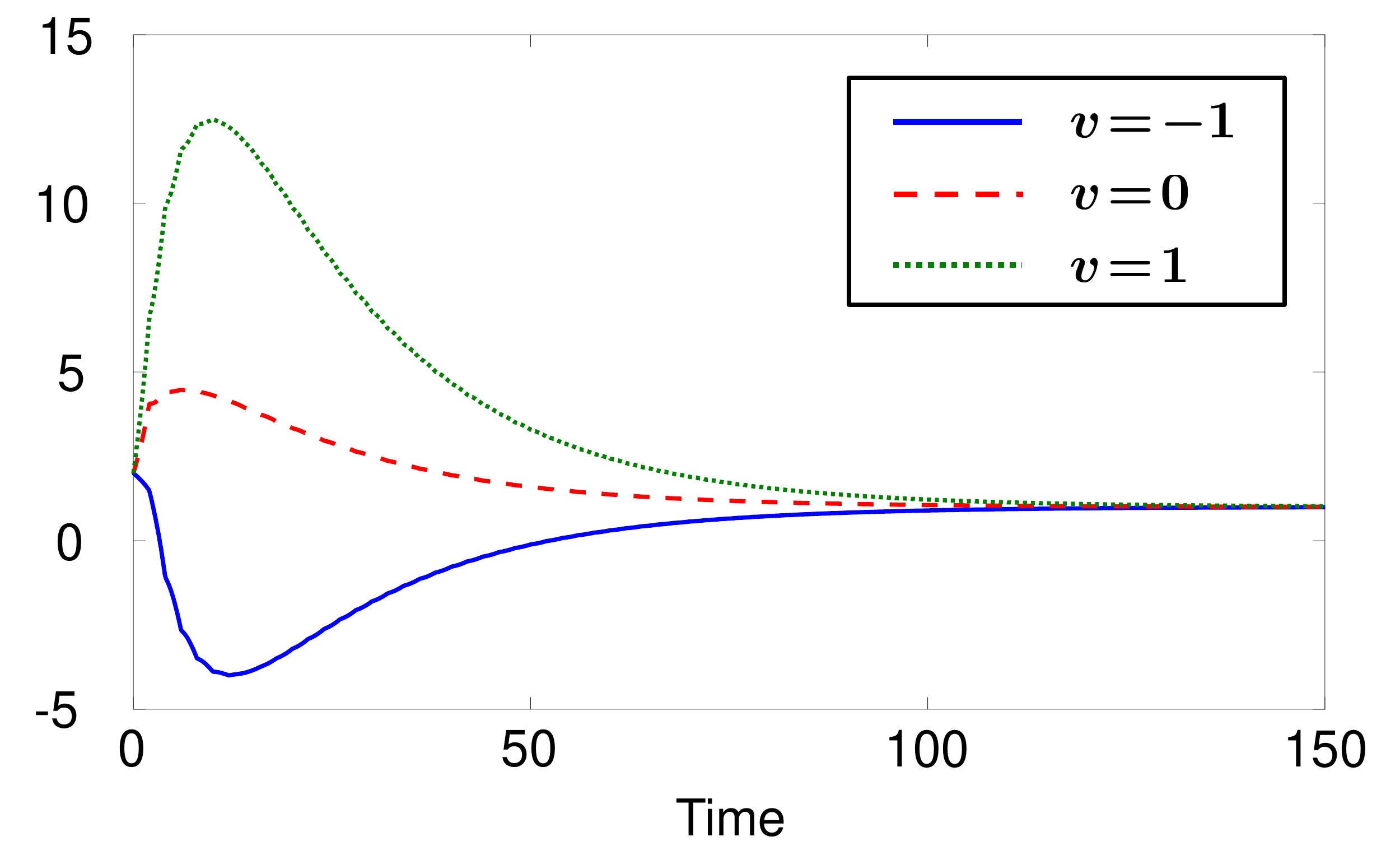}
	\caption{Time response of $y$ with $y_{\rm ref} = 1$.}
	\label{fig:output}
\end{figure}

\begin{figure}[tb]
	\centering
	\includegraphics[width = 7cm]{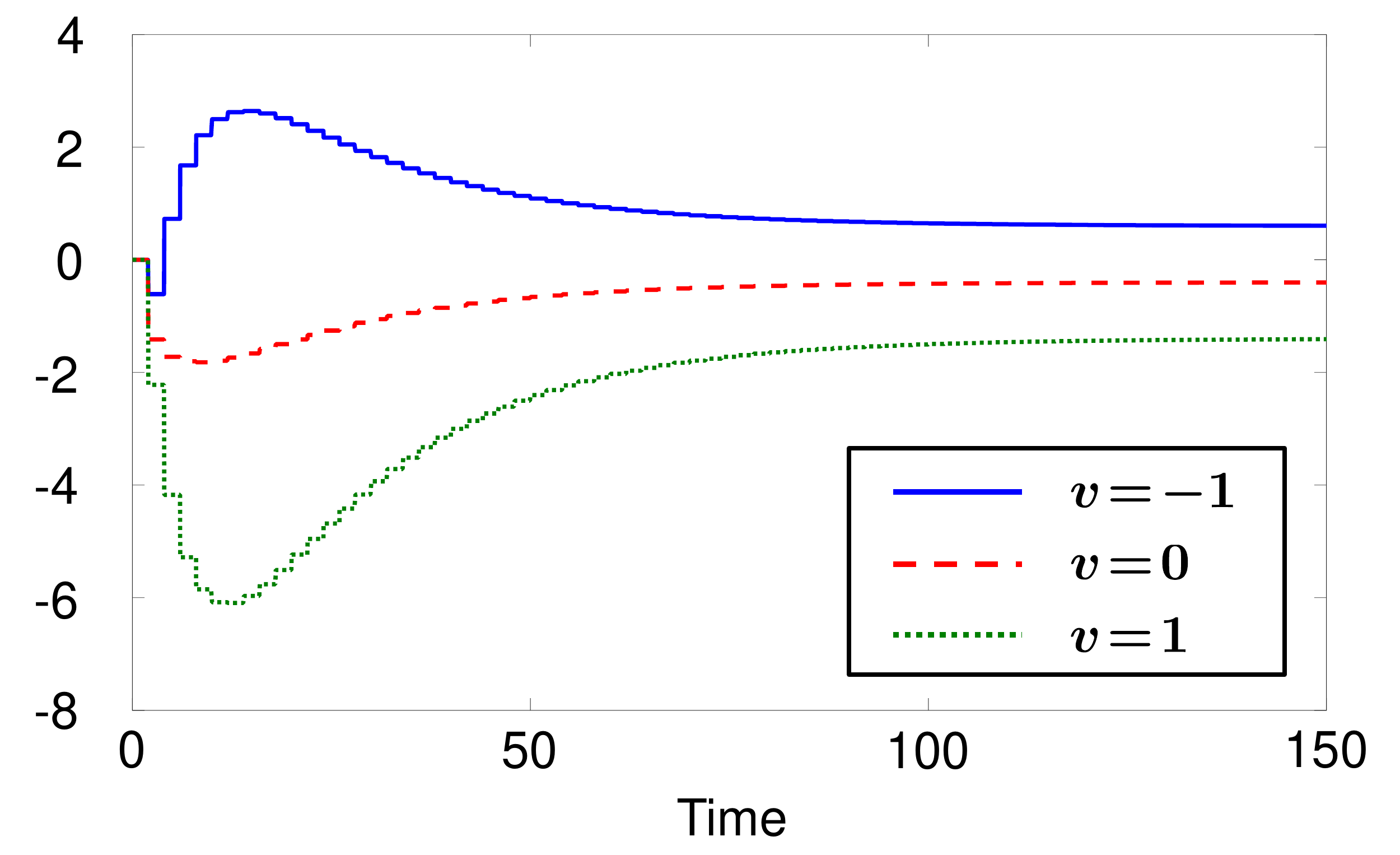}
	\caption{Time response of $u$ with $y_{\rm ref} = 1$.}
	\label{fig:input}
\end{figure}

\section{Conclusion}
We have studied the sampled-data output regulation problem
for infinite-dimensional systems with constant reference and disturbance signals.
Our main contribution is to obtain a sufficient condition for
this control problem to be solvable with a finite-dimensional controller.
To this end, we have proposed a design method of
finite-dimensional controllers for the robust output regulation of
infinite-dimensional discrete-time  systems.
In the controller design,
the discrete-time output regulation problem has been
reduced to the Nevanlinna-Pick interpolation problem.
We have also applied the obtained results to systems with state and output delays.
In future work on sampled-data output regulation,
we are planning to design generalized hold functions for infinite-dimensional systems
with general reference and disturbance signals.

\appendix
{\normalsize
	\renewcommand{\thesection}{\Alph{section}}
\section{Nevanlinna-Pick interpolation problem}
\label{sec:NPI}

In this section, we obtain a necessary and sufficient condition for 
the solvability of the interpolation problem to which we reduce
the design problem of regulating controllers. 
In the process, we also show 
how to construct a solution of the interpolation problem.
Although we consider $H^{\infty}(\mathbb{E}_1,\mathbb{C}^{p \times q})$ in Section~\ref{sec:DTOR},
 the standard theory of the Nevanlinna-Pick interpolation problem 
 uses $H^{\infty}(\mathbb{D},\mathbb{C}^{p \times q})$.
Hence, it is convenient to map $\mathbb{E}_1$ to $\mathbb{D}$ via the 
bilinear transformation
$\varphi: \mathbb{E}_1 \to \mathbb{D}:z \mapsto 1/z$.

In Section~\ref{subsec:Int_pro_interior}, we recall basic facts on
the Nevanlinna-Pick interpolation problem only with conditions on the interior $\mathbb{D}$.
Section~\ref{subsec:Int_pro_IB} is devoted to solving
the Nevanlinna-Pick interpolation problem with conditions on both the interior $\mathbb{D}$ and the boundary $\mathbb{T}$.
As in \cite{luxemburg2010,wakaiki2012}, we transform this problem
into the Nevanlinna-Pick interpolation problem only with conditions on the boundary $\mathbb{T}$,
which is always solvable.

\subsection{Interpolation problem only with interior conditions}
\label{subsec:Int_pro_interior}
First we consider interpolation problems only with interior 
interpolation conditions.
\begin{problem}[Chapter 18 in \cite{ball1990}, Section II in \cite{kimura1987}]
	\label{probl:NPI_int}
	Suppose that  $\alpha_1,\dots,\alpha_n \in \mathbb{D}$ are distinct.
	Let vector pairs $(\xi_\ell, \eta_\ell)\in\mathbb{C}^p \times 
	\mathbb{C}^q$ satisfy
	\begin{equation}
		\label{eq:tangential_data_cond}
		\|\xi_\ell\|_{\mathbb{C}^p} > \|\eta_\ell\|_{\mathbb{C}^q}  \qquad \forall \ell\in \{1,\dots, n\}.
	\end{equation}
	Find $\Phi \in H^{\infty}(\mathbb{D}, \mathbb{C}^{p\times q})$ such that
	$\|\Phi\|_{H^{\infty}(\mathbb{D})} < 1$ and 
	\begin{align*}
		\xi_\ell^* \Phi(\alpha_\ell) &= \eta_\ell^*\qquad \forall \ell \in \{1,\dots, n\}.
	\end{align*}
\end{problem}

We call this problem the {\em Nevanlinna-Pick interpolation problem with $n$ interpolation data
	$(\alpha_\ell, \xi_\ell,\eta_\ell)_{\ell=1}^n$}.
The solvability of Problem~\ref{probl:NPI_int} can be 
characterized by the so-called Pick matrix.
\begin{theorem}[Theorem~18.2.3 in \cite{ball1990}, Theorem~2 in \cite{kimura1987}]
	\label{thm:Pick}
	Consider Problem~\ref{probl:NPI_int}.
	Define the Pick matrix $P$ by
	\begin{equation*}
		P := 
		\begin{bmatrix}
			P_{1,1} & \cdots & P_{1,n} \\
			\vdots & & \vdots \\
			P_{n,1} & \cdots & P_{n,n}
		\end{bmatrix},
		\quad
		\text{where~}
				P_{j,\ell} := 
		\frac{\xi_j^*\xi_\ell - \eta_j^*\eta_\ell}{1-\alpha_j{\widebar \alpha}_\ell}\quad 
		\forall j,\ell \in \{1,\dots,n\}.
	\end{equation*}
	Problem~\ref{probl:NPI_int} is solvable if and only if
	$P$ is positive definite.
\end{theorem}

Let us next introduce an algorithm to construct a solution of Problem~\ref{probl:NPI_int}.
To this end, 
define 
\[
\mathcal{B} := \{E \in \mathbb{C}^{p \times q} : \|E\|_{\mathbb{C}^{p\times q}}  < 1 \}.
\]
Let $I_p$ and $I_q$ be the identity matrix with dimension $p$ and $q$,
respectively.
For a matrix $E \in \mathcal{B}$, define
\begin{subequations}
	\label{def:ABCD_Tangential}
	\begin{align}
		A(E) &:= (I_p-EE^*)^{-1/2},\quad B(E) := -(I_p-EE^*)^{-1/2}E \\
		C(E) &:= -(I_q-E^*E)^{-1/2}E^*,\quad D(E) := (I_q-E^*E)^{-1/2},
	\end{align}
\end{subequations}
where $M^{-1/2}$ denotes the inverse of the Hermitian square root of 
a positive definite matrix $M$. 
Define the maps $ U_E$ and $V_E$ by 
\begin{align*}
	U_E&:\mathbb{C}^{p} \times \mathbb{C}^{q} \to \mathbb{C}^{p}:
	(\xi,\eta) \mapsto A(E)\xi + B(E)\eta \\
	V_E&:\mathbb{C}^{p} \times \mathbb{C}^{q} \to \mathbb{C}^{q}:
	(\xi,\eta) \mapsto C(E)\xi + D(E)\eta.	
\end{align*}

The mapping $T_E$ in the lemma below is useful for solving Problem~\ref{probl:NPI_int}.
\begin{lemma}[Lemma~6.5.10 in \cite{vidyasagar1985}]
	\label{lem:T_E_norm_preserve}
	For a matrix $E \in \mathcal{B}$, define the 
	matrices $A(E)$, $B(E)$, $C(E)$, and $D(E)$ by \eqref{def:ABCD_Tangential}. 
	The mapping
	\begin{equation}
		\label{def:T_E}
		T_E:\mathcal{B} \to \mathcal{B}:X \mapsto 
		\big(A(E)X+B(E)\big)
		\big(C(E)X+D(E)\big)^{-1}
	\end{equation}
	is well-defined and bijective. 
\end{lemma}

A routine calculation shows that 
the inverse of $T_E$ is given by 
\begin{align}
	T_E^{-1}(Y) &= \big(A(E)-YC(E)\big)^{-1}
	\big(YD(E)-B(E)\big) \notag \\
	&= 
	\big(A(E)Y-B(E)\big)
	\big(\!-C(E)Y+D(E)\big)^{-1}.
	\label{eq:TE_inv}
\end{align}

\begin{lemma}[Lemma~1 in \cite{kimura1987}]
	\label{thm:Nevanlinna_Algorithm}
	Consider 
	Problem~\ref{probl:NPI_int} with $n$ interpolation data
	$(\alpha_\ell, \xi_\ell,\eta_\ell)_{\ell=1}^n$.
	Set $E := \xi_1  \eta_1^*/\| \xi_1\|^2_{\mathbb{C}^p} $ and
	define $A(E)$, $B(E)$, $C(E)$, and $D(E)$ as in \eqref{def:ABCD_Tangential}.
	Define also $\nu := U_E(\xi_1,\eta_1)$ and
	\begin{align}
		\kappa (z) &:= 
		\begin{cases}
			\frac{|\alpha_1|}{\alpha_1}\frac{z- \alpha_1}{1-{\widebar \alpha}_1 z} & \text{if $\alpha_1 \not= 0$} \\
			z & \text{if $\alpha_1 = 0$} 
		\end{cases},\quad 
		X := I_p + (\kappa -1)\frac{\nu \nu^*}{\| \nu\|^2_{\mathbb{C}^p} }. \label{eq:X_def_no_epsilon}
	\end{align}
	Problem~\ref{probl:NPI_int} with
	$n$ interpolation data 
	$(\alpha_\ell, \xi_\ell,\eta_\ell)_{\ell=1}^n$
	is solvable if and only if Problem~\ref{probl:NPI_int}  with
	$n-1$ interpolation data
	\begin{equation}
		\label{eq:int_data_reduced}
		\big(\alpha_\ell, X(\alpha_\ell)^*U_E(\xi_\ell,\eta_\ell),V_E(\xi_\ell,\eta_\ell)\big)_{\ell=2}^n
	\end{equation}
	is solvable. Moreover, 
	if  $\Phi_{n-1}$ is a solution 
	of the problem with $n-1$ interpolation data given in 
	\eqref{eq:int_data_reduced}, then
	\begin{equation}
		\label{eq:tangential_iteration}
		\Phi_n := T_{-E}\left(X \Phi_{n-1}\right)=
		\big(A(E)X \Phi_{n-1}-B(E)\big)
		\big(\!-C(E)X \Phi_{n-1}+D(E)\big)^{-1}
	\end{equation}
	is a solution $\Phi_n$ of the original problem 
	with $n$ interpolation data $(\alpha_\ell, \xi_\ell,\eta_\ell)_{\ell=1}^n$.
\end{lemma}

The iterative algorithm derived from Lemma~\ref{thm:Nevanlinna_Algorithm} is 
called the {\em Schur-Nevanlinna algorithm}.
Lemma~\ref{thm:Nevanlinna_Algorithm} also shows that
if the problem is solvable, then there exist always solutions whose elements are rational functions.

Note that $\nu$ given in Lemma~\ref{thm:Nevanlinna_Algorithm} is nonzero.
In fact, 
since
$
\|\xi_1\|_{\mathbb{C}^p}  > \|\eta_1\|_{\mathbb{C}^q} ,
$
it follows that 
\begin{equation*}
	A(E)^{-1}\nu = \xi_1 - E \eta_1 = \xi_1 - \frac{\|\eta_1\|_{\mathbb{C}^q}^2}{\|\xi_1\|_{\mathbb{C}^p}^2} \xi_1 \not= 0,
\end{equation*}
and hence $\nu \not= 0$. Furthermore, the matrix
$X$ defined by \eqref{eq:X_def_no_epsilon} satisfies
$X(\lambda)^{-1} = X(\lambda)^*$ for all $\lambda \in \mathbb{T}$ and
$\|X(z)\|_{\mathbb{C}^{p\times p}}  < 1$ for all $z \in \mathbb{D}$.

\subsection{Interpolation problem with both interior and boundary conditions}
\label{subsec:Int_pro_IB}
We next study interpolation problems with both interior and 
boundary conditions.
\begin{problem}
	\label{probl:NPI_int_bound}
	Suppose that  $\alpha_1,\dots,\alpha_n \in \mathbb{D}$ and
	$\lambda_1,\dots,\lambda_m \in \mathbb{T}$ are distinct.
	Consider vector pairs $(\xi_\ell, \eta_\ell)\in\mathbb{C}^p \times 
	\mathbb{C}^q$ for $ \ell \in \{1,\dots, n\}$ and matrices $F_j,G_j \in \mathbb{C}^{p\times q}$
	for $j\in \{1,\dots, m\}$, and 
	suppose that 
	\begin{subequations}
		\label{eq:tangent_matrix_nec}
	\begin{align}
	\|\xi_\ell\|_{\mathbb{C}^p} &> \|\eta_\ell\|_{\mathbb{C}^q}\qquad \forall \ell \in \{1,\dots, n\} \\
	\|F_j\|_{\mathbb{C}^{p\times q}} &< 1 \qquad \forall j\in \{1,\dots, m\}.
	\end{align}
	\end{subequations}
	Find a rational function $\Phi \in H^{\infty}(\mathbb{D}, \mathbb{C}^{p\times q})$ such that
	$\|\Phi\|_{H^{\infty}(\mathbb{D})} < 1$ and 
	\begin{subequations}
		\begin{align}
			\label{eq:int_cond}
			\xi_\ell^* \Phi(\alpha_\ell) &= \eta_\ell^*\qquad \forall \ell \in \{1,\dots, n\}\\
			\label{eq:bound_cond}
			\Phi(\lambda_j) &= F_j,\quad \Phi^{\prime}(\lambda_j) = G_j
			\qquad \forall j \in \{1,\dots, m\}.
		\end{align}
	\end{subequations}
\end{problem}

Problem~\ref{probl:NPI_int_bound} is called 
the {\em Nevanlinna-Pick interpolation problem with interior 
	interpolation data $(\alpha_\ell, \xi_\ell,\eta_\ell)_{\ell=1}^n$
	and boundary interpolation data $	(\lambda_j,F_j,G_j)_{j=1}^m$.
}
The scalar-valued case $p=q=1$ with more general interpolation conditions
has been studied in \cite{luxemburg2010}.

The following theorem implies that the solvability of
Problem~\ref{probl:NPI_int_bound}
depends only on its interior interpolation data.
\begin{theorem}
	\label{thm:No_boundary}
		Problem~\ref{probl:NPI_int_bound} with interior 
		interpolation data $(\alpha_\ell, \xi_\ell,\eta_\ell)_{\ell=1}^n$
		and boundary interpolation data $(\lambda_j,F_j,G_j)_{j=1}^m$ is
		solvable if and only if 
		Problem~\ref{probl:NPI_int}
		with interpolation data 
		$(\alpha_\ell,\xi_\ell,\eta_\ell)_{\ell=1}^n$ is solvable.
\end{theorem} 

To solve 
Problem~\ref{probl:NPI_int_bound},
we transform it
to the following problem with boundary conditions only:
\begin{problem}
	\label{probl:NPI_bound}
	Suppose that  
	$\lambda_1 ,\dots,\lambda_m \in \mathbb{T}$ are distinct.
	Consider  matrices $F_j,G_j \in \mathbb{C}^{p\times q}$
	for $j\in \{1,\dots, m\}$, and 
	suppose that 
	\begin{equation}
		\label{eq:Fj_cond} 
		\|F_j\|_{\mathbb{C}^{p\times q}} < 1 \qquad \forall j\in \{1,\dots, m\}.
	\end{equation}
	Find a rational function $\Phi \in H^{\infty}(\mathbb{D}, \mathbb{C}^{p\times q})$ such that
	$\|\Phi\|_{H^{\infty}(\mathbb{D})} < 1$ and 
	\begin{align*}
		\Phi(\lambda_j) &= F_j,\quad \Phi^{\prime}(\lambda_j) = G_j
		\qquad \forall j \in \{1,\dots, m\}.
	\end{align*}
\end{problem}
This problem is referred to as the {\em boundary Nevanlinna-Pick interpolation problem with
	interpolation data $(\lambda_j,F_j,G_j)_{j=1}^m$}.
The condition \eqref{eq:Fj_cond} is necessary for the
solvability for Problem~\ref{probl:NPI_bound}, and
the lemma below shows that the condition \eqref{eq:Fj_cond}  
is also sufficient.
We can prove the sufficiency 
by extending the Schur-Nevanlinna algorithm in Lemma~\ref{thm:Nevanlinna_Algorithm}.
\begin{lemma}
	\label{lem:boundary}
	Problem
	\ref{probl:NPI_bound} 
	is always solvable.
\end{lemma}
\begin{proof}
	Consider Problem
	\ref{probl:NPI_bound} with
	interpolation data $(\lambda_j,F_j,G_j)_{j=1}^m$.
	We first find
	$m-1$ interpolation data
	such that if Problem~\ref{probl:NPI_bound} with these $m-1$ data
	is solvable, then the original problem with $m$ interpolation data 
	$
	(\lambda_j,F_j,G_j)_{j=1}^m
	$
	is also solvable. 
	To that purpose, we extend
	the technique developed in \cite{luxemburg2010} for the scalar-valued case.
	
	Define $A:= A(F_1)$, $B:= B(F_1)$, $C:= C(F_1)$, and $D:= D(F_1)$ as in 
	\eqref{def:ABCD_Tangential}. For $\epsilon >0$, set
	\begin{align*}
		\kappa_{\epsilon}(z) &:= \frac{1}{\lambda_1}\frac{z-\lambda_1}{(1+\epsilon)-\widebar{\lambda}_1 z} \\
				\widehat F_1 &:= \epsilon \lambda_1(I_p-F_1F_1^*)^{-1/2} G_1 (I_q-F_1^*F_1)^{-1/2}\\
	\end{align*}
		and
	\begin{align*}
		\widehat F_j &:= \frac{1}{\kappa_{\epsilon}(\lambda_j)}T_{F_1}(F_j)
		\\
		\widehat G_j &:= \frac{1}{\kappa_{\epsilon}(\lambda_j)}
		(A-\kappa_\epsilon(\lambda_j)\widehat F_j C) G_j (CF_j +D)^{-1} - 
		\frac{\kappa^{\prime}_{\epsilon}(\lambda_j)}{\kappa_{\epsilon}(\lambda_j)}  \widehat F_j
	\end{align*}
	 for $j \in \{2,\dots,m\}$.
	Let us  show that there exists $\epsilon>0$ such that
	\begin{equation}
		\label{eq:hatFnorm}
		\|\widehat F_j\|_{\mathbb{C}^{p\times q}}<1\qquad \forall j \in \{1,\dots,m \}.  
	\end{equation}
	By definition,
	\[
	\|\widehat F_1\|_{\mathbb{C}^{p\times q}} \leq \epsilon 
	\|G_1\|_{\mathbb{C}^{p\times q}} \cdot \big\|(I_p-F_1F_1^*)^{-1/2} \big\|_{\mathbb{C}^{p\times p}}
	\cdot
	\big\|(I_q-F_1^*F_1)^{-1/2} \big\|_{\mathbb{C}^{q\times q}},
	\]
	and hence if
	\begin{equation}
		\label{eq:epsilon_cond_1}
		\epsilon < \frac{1}{\|G_1\|_{\mathbb{C}^{p\times q}} \cdot \big\|(I_p-F_1F_1^*)^{-1/2}\big\|_{\mathbb{C}^{p\times p}}
			\cdot \big\|(I_q-F_1^*F_1)^{-1/2} \big\|_{\mathbb{C}^{q\times q}}},
	\end{equation}
	then $\|{\widehat F}_1\|_{\mathbb{C}^{p\times q}} < 1$.
	Let $j \in \{2,\dots,m  \}$ be given. We obtain
	\begin{align}
		\|\widehat F_j \|_{\mathbb{C}^{p\times q}} 	
		\leq  
		\left( 
		1+ \frac{\epsilon }{| \lambda_j - \lambda_1|}
		\right)
		\| T_{F_1}(F_j)\|_{\mathbb{C}^{p\times q}}. \label{eq:hatF_cond}
	\end{align}
	Since $F_j \in \mathcal{B}$, it follows that
	$\|T_{F_1}(F_j)\|_{\mathbb{C}^{p\times q}} < 1$ by Lemma~\ref{lem:T_E_norm_preserve}. 
	If we choose $\epsilon > 0$ so that
	\begin{equation}
		\label{eq:epsilon_cond_2}
		\epsilon
		< \min_{j=2,\dots,m} 
		\left(
		|\lambda_j - \lambda_1|
		\left(
		\frac{1}{\|T_{F_1}(F_j)\|_{\mathbb{C}^{p\times q}} } - 1
		\right)
		\right),
	\end{equation}
	then $\|{\widehat F}_j\|_{\mathbb{C}^{p\times q}} < 1$
	for every $j \in \{2,\dots,m  \}$.
	Thus, 
	we obtain the desired inequality \eqref{eq:hatFnorm} for
	$\epsilon >0$ satisfying \eqref{eq:epsilon_cond_1}
	and \eqref{eq:epsilon_cond_2}.
	
	Assume that there exists a rational solution $\Psi_{m-1} 
	\in H^{\infty}(\mathbb{D}, \mathbb{C}^{p\times q})$ such that 
	\begin{subequations}
		\label{eq:Phi_m_1_cond}
	\begin{align}
	&\|\Phi_{m-1}\|_{H^{\infty}(\mathbb{D})}<1 \\
	& \Psi_{m-1} (\lambda_j) = \widehat F_j\qquad \forall
	j \in \{1,\dots,m \} \\
	&\Psi_{m-1} ^{\prime}(\lambda_j) = \widehat G_j\qquad \forall j \in \{2,\dots,m\}
	\end{align}
	\end{subequations}
	We shall show that 
	$
	\Psi_{m} := T_{F_1}^{-1}(\kappa_{\epsilon}\Psi_{m-1})
	$
	is a  solution of the original problem with $m$ interpolation data
	$(\lambda_j, F_j,G_j)_{j=1}^m$.
	By definition, $\Psi_m$ is rational. 
	Since $\|\kappa_{\epsilon}\|_{H^{\infty}(\mathbb{D})} < 1$ and $\|\Psi_{m-1}\|_{H^{\infty}(\mathbb{D})} < 1$, it follows that 
	\begin{equation*}
		\kappa_{\epsilon}(z)\Psi_{m-1}(z) \in \mathcal{B}\qquad \forall z \in \cl(\mathbb{D}).
	\end{equation*}
	Together with this, Lemma~\ref{lem:T_E_norm_preserve} yields
	$\Psi_m \in H^{\infty}(\mathbb{D}, \mathbb{C}^{p\times q})$ and
	$\|\Psi_m\|_{H^{\infty}(\mathbb{D})}<1$. 
	
	We now prove that $\Psi_m$ satisfies the interpolation conditions 
	$\Psi_m(\lambda_j) = F_j$ and $\Psi_m'(\lambda_j) = G_j$ for every $j \in \{
	1,\dots,m\}$.
	For the case  $j=1$, $\kappa_{\epsilon}(\lambda_1) = 0$ yields
	\begin{align*}
		\Psi_m (\lambda_1) =	T_{F_1}^{-1}\big(\kappa_{\epsilon}(\lambda_1)\Psi_{m-1}(\lambda_1)\big) 
		=F_1.
	\end{align*}
	By \eqref{eq:TE_inv}, we obtain
	\[
	(A- \kappa_{\epsilon}\Psi_{m-1}C
	)
	\Psi_m
	=
	\kappa_{\epsilon} \Psi_{m-1}D-B,
	\]
	which implies 
	\begin{equation}
		\label{eq:psi_diff}
		(\kappa_{\epsilon}
		\Psi_{m-1}^{\prime}+ \kappa_{\epsilon}^{\prime}\Psi_{m-1})  (C\Psi_m+D)= 
		(A-\kappa_{\epsilon} \Psi_{m-1} C)\Psi_m^{\prime}.
	\end{equation}
	Therefore, 
	\[
	\Psi_m^{\prime}(\lambda_1) = 
	\kappa_{\epsilon}^{\prime}(\lambda_1) A^{-1} \widehat F_1(CF_1+D).
	\]
	Since
	\[
	\kappa_{\epsilon}^{\prime}(z) = \frac{1}{\lambda_1} \frac{\epsilon}{\big((1+\epsilon) - \widebar{\lambda}_1 z \big)^2},
	\]
	it follows that $\kappa_{\epsilon}^{\prime}(\lambda_1) = 1/(\epsilon \lambda_1 )$.
	Using
	\[
	A^{-1} = (I_p-F_1F_1^*)^{1/2},\quad CF_1+D =  (I_q-F_1^*F_1)^{1/2},
	\]
	we derive
	$
	\Psi^{\prime}_m(\lambda_1) = G_1.
	$
	
	For $j\in \{2,\dots,m\}$, we have by the definition of $\widehat F_j$ that,
	\begin{align*}
		\Psi_m(\lambda_j) 	
		=	T_{F_1}^{-1}(\kappa_{\epsilon}(\lambda_j)\widehat F_j) =	T_{F_1}^{-1}\big(T_{F_1}(F_j)\big) = F_j.
	\end{align*}
	Using \eqref{eq:psi_diff} again, we obtain
	\[
	\kappa_{\epsilon}(\lambda_j) \widehat G_j+ 
	\kappa_{\epsilon}^{\prime}(\lambda_j) \widehat F_j =
	(A - \kappa_{\epsilon}(\lambda_j) \widehat F_j C) \Psi_m^{\prime}(\lambda_j) (CF_j +D)^{-1}.
	\]
	By the definition of $\widehat G_j$, we find that 
	\[
	\Psi_m^{\prime}(\lambda_j) = G_j\qquad \forall j \in \{2,\dots,m  \}.
	\]
	Thus $\Phi_m$ is a solution of the original problem with $m$ interpolation conditions.
	
	If we apply this procedure again to the resulting interpolation problem, i.e., the problem of finding
	a rational solution $\Psi_{m-1} 
	\in H^{\infty}(\mathbb{D}, \mathbb{C}^{p\times q})$ such that 
	the conditions given in \eqref{eq:Phi_m_1_cond} hold,
	then
	the interpolation condition at $z = \lambda_1$ is removed.
	Therefore, 
	Problem~\ref{probl:NPI_bound}
	with $m$ interpolation data can be reduced 
	to Problem~\ref{probl:NPI_bound} with $m-1$ interpolation data.  
	Continuing in this way, we finally obtain
	Problem~\ref{probl:NPI_bound}
	with no interpolation conditions, which always admits a solution.  
	Thus Problem~\ref{probl:NPI_bound}
	is always solvable.   
	\qed
\end{proof}

By Lemmas~\ref{thm:Nevanlinna_Algorithm} and \ref{lem:boundary}, 
we obtain a proof of
Theorem~\ref{thm:No_boundary}.

\begin{proof}[of Theorem~\ref{thm:No_boundary}]
	The necessity is straightforward. 
	We prove the sufficiency.
	To this end, it is enough to show that  the following problem always has a solution:
	
	\begin{problem}
		\label{prob:proof}
		Assume that 
		Problem~\ref{probl:NPI_int} with 
		$n$ interior interpolation data $(\alpha_\ell,\xi_\ell,\eta_\ell)_{\ell=1}^n$ is solvable
		and that $\|F_j\|_{\mathbb{C}^{p \times q}} < 1$ for every $j \in \{1,\dots,m \}$.
		Find a solution of Problem~\ref{probl:NPI_int_bound} 
		with 
		$n$ interior interpolation data $(\alpha_\ell,\xi_\ell,\eta_\ell)_{\ell=1}^n$ 
		and $m$ boundary interpolation data $(\lambda_j, F_j, G_j)_{j=1}^m$.
	\end{problem}
	
	Suppose that 
	Problem~\ref{probl:NPI_int} with 
	$n$ interior interpolation  data $(\alpha_\ell,\xi_\ell,\eta_\ell)_{\ell=1}^n$ is solvable. 
	Define the matrix $E$ and the function $X$ as in Lemma~\ref{thm:Nevanlinna_Algorithm}. Then
	this lemma
	shows that 
	Problem~\ref{probl:NPI_int} with 
	$n-1$ interior interpolation  data 
	\begin{equation}
		\label{eq:n_1_interior_cond}
		\big(\alpha_\ell,X(\alpha_\ell)^*U_E(\xi_\ell,\eta_\ell),V_E(\xi_\ell,\eta_\ell)\big)_{\ell=2}^n
	\end{equation} 
	is solvable.	
	Set $A:= A(E)$, $B:= B(E)$, $C:= C(E)$, and $D:= D(E)$ as in \eqref{def:ABCD_Tangential}.
	For $j \in \{1,\dots,m\}$,
	define also
	\begin{align*}
		\widehat F_j &:= X(\lambda_j)^{-1} T_{-E}^{-1}(F_j) \\
		\widehat G_j &:= X(\lambda_j)^{-1} (A+F_jC)^{-1}G_j(-CX(\lambda_j)\widehat F_j + D) - X(\lambda_j)^{-1} X^{\prime}(\lambda_j)\widehat F_j.
	\end{align*}
	Since 
	$X(\lambda_j)^{-1} = X(\lambda_j)^*$ for every $j \in \{1,\dots,m \}$, 
	we obtain $\|X(\lambda_j)^{-1}\|_{\mathbb{C}^{p\times p}} = 1$
	and hence $\|\widehat F_j\|_{\mathbb{C}^{p\times p}} < 1$ for every $j \in \{1,\dots,m \}$. 
	Suppose that $\Phi_{n-1}$ is a solution of 
	Problem~\ref{probl:NPI_int_bound}
	with
	$n-1$ interior interpolation data given in
	\eqref{eq:n_1_interior_cond}
	and $m$ boundary interpolation data
	$
	(\lambda_j, \widehat F_j, \widehat G_j)_{j=1}^m.
	$
	Then $
	\Phi_n := T_{-E}(X \Phi_{n-1})
	$
	is a solution of Problem~\ref{probl:NPI_int_bound} with
	$n$ interior interpolation data $(\alpha_\ell,\xi_\ell,\eta_\ell)_{\ell=1}^n$ 
	and $m$ boundary interpolation data $(\lambda_j, F_j, G_j)_{j=1}^m$. In fact,
	Lemma~\ref{thm:Nevanlinna_Algorithm} shows that 
	$\Phi_n$ satisfies $\|\Phi_n\|_{H^{\infty}(\mathbb{D})} < 1$ and  $\xi_{\ell}^* \Phi_n(\alpha_\ell) = \eta_\ell^*$ for every
	$\ell \in \{1,\dots,n\}$.
	It remains to show that the boundary conditions hold. We obtain
	\[
	\Phi_n(\lambda_j) =  T_{-E}\big(X(\lambda_j) \widehat{F}_j \big)
	= 
	T_{-E}\big(T_{-E}^{-1}(F_j) \big)
	=F_j\qquad \forall j \in \{1,\dots,m \}.
	\]
	By the definition of $T_{-E}$, we obtain
	\[
	\Phi_n(-CX\Phi_{n-1}+D) = (AX\Phi_{n-1}-B),
	\]
	and hence
	\[
	\Phi_n^{\prime} (-CX \Phi_{n-1}+D) =
	(A+\Phi_nC)(X\Phi_{n-1})^{\prime}.
	\]
	This yields
	\[
	\Phi_n^{\prime}(\lambda_j) = 
	(A+F_jC)(X(\lambda_j) \widehat G_j + X^{\prime}(\lambda_j) \widehat F_j) (-CX(\lambda_j) \widehat F_j +D)^{-1} = G_j.
	\]
	Thus, we can reduce Problem~\ref{prob:proof} with $n$ interior data to
	that with $n-1$ interior data.
	Continuing in this way, we reduce
	Problem~\ref{probl:NPI_int_bound} 
	to Problem~\ref{probl:NPI_bound},
	which is always solvable by Lemma~\ref{lem:boundary}.
	This completes the proof.
	\qed
\end{proof}

In the construction of regulating controllers in Section~\ref{sec:DTOR},
a rational function $\mathbf{Y}_+\in H^{\infty}(\mathbb{E}_1,\mathbb{C}^{p \times p})$ needs
to satisfy the  interpolation condition $\mathbf{Y}_+(\infty) = 0$.
Its counterpart in $H^{\infty}(\mathbb{D},\mathbb{C}^{p \times p})$ under the transformation
$\varphi: \mathbb{E}_1 \to \mathbb{D}:z \mapsto 1/z$ is given by
the interpolation condition $(\mathbf{Y}_+\circ \varphi^{-1} )(0) = 0$.
Such a condition is excluded in Problem~\ref{probl:NPI_int_bound}, but
we can easily incorporate it into the problem.

\begin{corollary}
	Suppose that  $\alpha_1,\dots,\alpha_n \in \mathbb{D}\setminus \{0\}$ and
	$\lambda_1,\dots,\lambda_m \in \mathbb{T}$ are distinct.
	Consider vector pairs $(\xi_\ell, \eta_\ell)\in\mathbb{C}^p \times 
	\mathbb{C}^q$ for $ \ell \in \{1,\dots, n\}$ and matrices $F_j,G_j \in \mathbb{C}^{p\times q}$
	for $j\in \{1,\dots, m\}$, and 
	suppose that the norm conditions \eqref{eq:tangent_matrix_nec} are satisfied.
	Then
	the following three statements are equivalent:
	
	\begin{enumerate}
		\renewcommand{\labelenumi}{\arabic{enumi})}
		\item
		There exists a rational function $\Phi \in H^{\infty}(\mathbb{D}, \mathbb{C}^{p\times q})$ such that 
		$\|\Phi\|_{H^{\infty}(\mathbb{D})} < 1$, $\Phi(0) = 0$, and the interpolation conditions \eqref{eq:int_cond} and 
		\eqref{eq:bound_cond} hold.
		
		\item
		There exists a rational function $\Phi \in H^{\infty}(\mathbb{D}, \mathbb{C}^{p\times q})$ such that 
		$\|\Phi\|_{H^{\infty}(\mathbb{D})} < 1$, $\Phi(0) = 0$, and the interpolation conditions \eqref{eq:int_cond} hold.
		
		\item 
		The Pick matrix $P$ defined by 
		\[
		P := 
		\begin{bmatrix}
		P_{1,1} & \cdots & P_{1,n} \\
		\vdots & & \vdots \\
		P_{n,1} & \cdots & P_{n,n}
		\end{bmatrix},
		~~
		\text{where~}
		P_{j,k} := \frac{\alpha_j \widebar{\alpha}_\ell \xi_j^*\xi_\ell - \eta_j^*\eta_\ell}{1- \alpha_j\widebar{\alpha}_\ell}
		~~
		\forall j,\ell \in \{1,\dots,n\}
		\]
		is positive definite.
	\end{enumerate}
\end{corollary}
\begin{proof}
	By a straightforward calculation, we have the following fact:
	A rational function $\Phi\in H^{\infty}(\mathbb{D}, \mathbb{C}^{p\times q})$ satisfies the conditions of 1) 
	if and only if
	${\widehat \Phi}(z) := \Phi(z)/z$ is a solution of Problem~\ref{probl:NPI_int_bound} with the 
	interior interpolation data $(\alpha_\ell, \widebar{\alpha}_\ell \xi_\ell, \eta_\ell)_{\ell=1}^n$ and the 
	boundary interpolation data 
	$(\lambda_j, F_j/\lambda_j, G_j/\lambda_j - F_j/\lambda_j^2)_{j=1}^m$.
	This fact 
	together with Theorem~\ref{thm:No_boundary} shows that
	1) is true if and only if Problem~\ref{probl:NPI_int} with
	the  interpolation data $(\alpha_\ell, \widebar{\alpha}_\ell \xi_\ell, \eta_\ell)_{\ell=1}^n$
	is solvable.
	Hence, we obtain 1) $\Leftrightarrow$ 3) by Theorem~\ref{thm:Pick}.
	Using the fact mentioned above again, we obtain 1) $\Leftrightarrow$ 2).
	This completes the proof. 
	\qed
\end{proof}

\begin{remark}
		Suppose that the interpolation data have conjugate symmetry 
		in Problem~\ref{probl:NPI_int_bound}. In other words,
		suppose that 
		both $(\alpha, \xi,\eta)$ and $({\widebar \alpha}, {\widebar \xi}, {\widebar \eta})$ are in
		its interior interpolation data and that 
		$(\lambda,F,G)$ and $({\widebar\lambda},{\widebar F},{\widebar G})$ are in its
		boundary interpolation data.
		If the interpolation problem is solvable, then
		there exists a solution that is a rational function with real coefficients.
		In fact, for every rational function $\Phi$, 
		there uniquely exist rational functions $\Phi_R$ and $\Phi_I$ with real coefficients such that $\Phi = \Phi_R + i \Phi_I$.
		If a rational function $\Phi$ is a solution of the interpolation problem, then
		one can easily prove that its real part $\Phi_R$ is also a solution. 
\end{remark}

\begin{remark}
		Let $\lambda \in \mathbb{T}$. For a vector pair
		$(\xi, \eta) \in \mathbb{C}^p \times 
		\mathbb{C}^q$, define a matrix $F:=\xi\eta^*/\|\xi\|_{\mathbb{C}^{p}}^2$.
		If $\|\xi\|_{\mathbb{C}^{p}}  > \|\eta\|_{\mathbb{C}^{q}} $, then $\|F\|_{\mathbb{C}^{p \times q}}  < 1$.
		Further, if a rational function
		$\Psi \in H^{\infty}(\mathbb{D}, \mathbb{C}^{p\times q})$ satisfies
		$\Psi(\lambda) = F$, then
		$
		\xi^* \Psi(\lambda) = \eta^*
		$.
		In this way, we can transform the tangential interpolation condition $
		\xi^* \Psi(\lambda) = \eta^*
		$
		to the matrix-valued interpolation condition $\Psi(\lambda) = F$.
		This transformation is used in the design procedure of regulating controllers
		in Section~\ref{sec:DTOR} if unstable eigenvalues of $A$ lie on the boundary $\mathbb{T}$.
		Moreover,
		the above observation and Theorem~\ref{thm:No_boundary}
		indicate that for
		$\lambda \in \mathbb{T}$ and $(\xi, \eta) \in \mathbb{C}^p \times 
		\mathbb{C}^q$ with $\|\xi\|_{\mathbb{C}^{p}}  > \|\eta\|_{\mathbb{C}^{q}}$, 
		boundary interpolation conditions of the form
		$
		\xi^* \Psi(\lambda) = \eta^*
		$
		can be also ignored when we determine the solvability of the Nevanlinna-Pick interpolation problem.

%
\end{remark}

\section{$\Lambda$-extension of output operator of delay systems}
\label{sec:lambda_extension}
Consider the delay system \eqref{eq:delay_system}, and 
define $x$ as in \eqref{eq:well_posed_delay}.
The objective of this section is to show for a.e. $t \geq 0$, 
\begin{equation}
	\label{eq:lambda_extension_delay}
	\sum_{\ell=1}^{\widehat q} c_\ell z(t-\widehat h_\ell) = C_{\Lambda}x(t).
\end{equation}
Since $x(t) \in X_1$ for every $t \geq h_q$ and since 
$C_{\Lambda} \zeta = C \zeta$ for every $\zeta \in X_1$, 
it suffices to show \eqref{eq:lambda_extension_delay} a.e. on
$[0,h_q)$.
For simplicity of notation, we consider 
the case $\widehat q = 1$ and define $\widehat h := \widehat h_1$ and
$c := c_{1}$.

By Lemma 2.4.5 of \cite{Curtain1995}, there exists $s_0 >0$ such that 
\[
(sI-A)^{-1} 
x(t)
=\begin{bmatrix}
g_1(t) \\
g_2(t)
\end{bmatrix}
\qquad \forall s > s_0,~\forall t \in [0,h_q),
\]
where 
\begin{align*}
	g_1(t) &:= \Delta(s)^{-1} 
	\left(
	z(t) + 
	\sum_{j=1}^{q}
	\int^0_{-h_j} e^{-s(\theta+h_j) }A_jz(t+\theta) d\theta
	\right) \\
	\big(g_2(t)\big)(\theta)&:=
	e^{s\theta} g_1(t) - \int^\theta_0 e^{s(\theta-\nu)} z(t+\nu)d\nu
	\qquad \forall \theta \in [-h_q,0].
\end{align*}
Hence 
for every $s>s_0$ and every $t \in [0,h_q)$, we obtain
\begin{align*}
	Cs(sI-A)^{-1}x(t) &=
	s c \big( g_2(t) \big)(-\widehat h) \\
	&=
	sc
	\left(
	e^{-s \widehat h} g_1(t) +
	\int^{\widehat h}_0 e^{-s(\widehat h-\nu)} z(t-\nu)d\nu
	\right).
\end{align*}
Since 
\[
\lim_{s\to \infty,~\!\! s \in \mathbb{R}} s\Delta(s)^{-1} = I
\quad \text{and} \quad
z \in L^1 \big((-h_q,h_q), \mathbb{C}^n\big),
\] 
Lebesgue's dominated convergence theorem implies that in the case  $\widehat h = 0$, 
\begin{align*}
	&\lim_{s\to \infty,~\!\!s \in \mathbb{R}}
	sc\left(
	e^{-s \widehat h} g_1(t) +
	\int^{\widehat h}_0 e^{-s(\widehat h-\nu)} z(t-\nu)d\nu\right) \\
	&\qquad  = 
	\lim_{s\to \infty,~\!\!s \in \mathbb{R}}
	s c\Delta(s)^{-1}	\left(
	z(t) + 
	\sum_{j=1}^{q}
	\int^0_{-h_j} e^{-s(\theta+h_j) }A_jz(t+\theta) d\theta
	\right)  \\
	&\qquad 
	= cz(t) \qquad \forall t  \in [0,h_q).
\end{align*}
Thus, we obtain $cz(t- \widehat h ) = C_{\Lambda}x(t)$ for every $ t  \in [0,h_q)$
if $\widehat h = 0$.

In the case $\widehat h \in (0,h_q)$, 
we obtain
\[
\lim_{s \to \infty,~\!\! s\in \mathbb{R}} s e^{-s \widehat h} g_1(t) = 0\qquad \forall t \in [0,h_q).
\]
Since $B \in \mathcal{L}(U,X)$, it follows that 
$x(t) \in  \dom  (C_{\Lambda}) $ for a.e. $t\geq 0$ and
\begin{align}
C_{\Lambda}x(t) &= 
\lim_{s \to \infty,~\!\!s\in \mathbb{R}}Cs(sI-A)^{-1}x(t) \notag \\
&= 	\lim_{s \to \infty,~\!\!s\in \mathbb{R}} s
\int^{{\widehat h}}_0 e^{-s({\widehat h}-\nu)} \zeta(t-\nu)d\nu
\qquad \text{a.e.~$t\geq 0$}, \label{eq:C_Lambda_conv_ap}
\end{align}
where $\zeta := cz$.
For each $n \in \mathbb{N}$, define
 \[
f_n(t) :=  n\int^{\widehat h}_0 e^{-n(\widehat h - \nu)} \zeta(t-\nu)d\nu\qquad \forall t \in [0,h_q).
\]
We will show that
there exists a subsequence $\{f_{n_{\ell}}:\ell \in \mathbb{N}\}$ such that 
$\lim_{\ell \to \infty }f_{n_\ell}(t)=  \zeta(t - \widehat h)$ for a.e. $t \in [0,h_q)$.
Together with \eqref{eq:C_Lambda_conv_ap}, this yields
$\zeta(t-\widehat h) = C_{\Lambda}x(t)$ for a.e. $ t  \in [0,h_q)$
in the case $\widehat h \in (0,h_q)$.


Let $s > s_0$.
Define 
\[
\varphi(s) := s\int^{\widehat h}_0 e^{-s(\widehat h - \nu)} d\nu = 1 - e^{-{\widehat h}s}.
\]
Since $\zeta \in L^1(-h_q,h_q)$, it follows from Fubini's theorem that 
\begin{align*}
	&
	\int^{h_q}_0 \left| \zeta(t-{\widehat h}) - s\int^{\widehat h}_0 e^{-s(\widehat h - \nu)} \zeta(t-\nu)d\nu \right|dt\\
	&~ \leq
	\int^{h_q}_0 \left| \big(1 - \varphi(s)\big)\zeta(t-{\widehat h})\right| dt
	+ s 
	\int^{h_q}_0  \int^{\widehat h}_0 e^{-s(\widehat h - \nu)} \big|\zeta(t- \widehat h) - \zeta(t-\nu) \big| d\nu dt
	\\
	&~ \leq 
	e^{-\widehat{h} s} \|\zeta\|_{L^1(-h_q,h_q)} + s
	\int^{\widehat h}_0 e^{-s(\widehat h - \nu)} \int^{h_q}_0  \big|\zeta(t-\widehat h) - \zeta(t-\nu) \big| dt d\nu.
\end{align*}

Choose $\varepsilon>0$ arbitrarily.
By  the strong continuity of
the left translation semigroup on $L^1(-h_q,h_q)$ (see, e.g., 
Example I.5.4 in \cite{Engel2000}), 
there exists $\delta_0 \in (0, \widehat h)$ such that 
\[
\int^{h_q}_0 |
\zeta(t-\widehat h) - \zeta(t-\widehat h +\delta)|dt < \varepsilon\qquad \forall \delta \in [0,\delta_0).
\]
Therefore, 
\begin{align*}
	s \int^{\widehat h}_{\widehat h - \delta_0} e^{-s(\widehat h - \nu)}  \int^{h_q}_0  \big|\zeta(t-\widehat h) - \zeta(t-\nu) \big| dt d\nu
	< \varepsilon(1-e^{-\delta_0 s}) < \varepsilon.
\end{align*}
Since
\begin{align*}
	s \int^{\widehat h - \delta_0}_{0} e^{-s(\widehat h - \nu)}  &\int^{h_q}_0  \big|\zeta(t-\widehat h) - \zeta(t-\nu) \big| dt d\nu \\
	&\qquad \leq 2\|\zeta\|_{L^1(-h_q,h_q)} (e^{-\delta_0 s} - e^{-\widehat h s}),
\end{align*}
it follows that there exists $s_1 > s_0$ such that for every $s > s_1$,
\begin{align*}
	e^{-\widehat{h} s} \|\zeta\|_{L^1(-h_q,h_q)} < \varepsilon,\quad 
	s \int^{\widehat h - \delta_0}_{0} e^{-s(\widehat h - \nu)}  \int^{h_q}_0  \big|\zeta(t-\widehat h) - \zeta(t-\nu) \big| dt d\nu
	< \varepsilon.
\end{align*}
Hence we obtain
\begin{align*}
	&
	\int^{h_q}_0 \left| \zeta(t-{\widehat h}) - s\int^{\widehat h}_0 e^{-s(\widehat h - \nu)} \zeta(t-\nu)d\nu \right|dt
	< 3\varepsilon.
\end{align*}	
Since 
$\varepsilon >0$ was arbitrary,
we have that 
$\lim_{n \to \infty }\|\zeta(\cdot - \widehat h) - f_n\|_{L^1(0,h_q)}  = 0$.
Then there exists a subsequence $\{f_{n_{\ell}}:\ell \in \mathbb{N}\}$ such that 
$\lim_{\ell \to \infty }f_{n_\ell}(t)=  \zeta(t - \widehat h)$ for a.e. $t \in [0,h_q)$;
see, e.g., Theorem~3.12 in \cite{Rudin1987}.
This completes the proof.
\qed
}

\begin{acknowledgements}
The authors would like to thank Professor Lassi Paunonen
for helpful advice on robust output regulation for 
infinite-dimensional discrete-time systems.
Furthermore, 
we would  like to thank the anonymous reviewers
for their careful reading of our manuscript and many insightful comments.
\end{acknowledgements}

\end{document}